\documentclass[a4paper,english,10 pt]{article}
\usepackage{tikz-dependency}
\usepackage[T1]{fontenc}
\usepackage[latin1]{inputenc}
 \usepackage{amsmath,amssymb}
\usepackage{graphicx}
\usepackage{amsthm}
\usepackage{float}
\usepackage{caption}
\usepackage{framed}
\usepackage[all]{xy}
\usepackage{multirow}
\usepackage{mathtools}
\usepackage{color}
\usepackage{hyperref}
\usepackage{amsfonts}
\usepackage{etex}
\usepackage{mathrsfs}
\usepackage{latexsym}
\usepackage{geometry}
\usepackage{makeidx}
\usepackage{shuffle}
\usepackage{perpage} %the perpage package
\MakePerPage{footnote} 
\usepackage{xcolor}
\usepackage{tikz}
\usepackage{babel}
\usetikzlibrary{arrows,chains,matrix,positioning,scopes}
\date{}
\makeatletter
\tikzset{join/.code=\tikzset{after node path={%
\ifx\tikzchainprevious\pgfutil@empty\else(\tikzchainprevious)%
edge[every join]#1(\tikzchaincurrent)\fi}}}
\makeatother
\tikzset{>=stealth',every on chain/.append style={join},
         every join/.style={->}}
\tikzstyle{labeled}=[execute at begin node=$\scriptstyle,
   execute at end node=$]
\makeatletter
\providecommand{\boldsymbol}[1]{\mbox{\boldmath $#1$}}
\newtheorem{theo}{Theorem}[section]

\newtheorem{defi}[theo]{Definition}
\newtheorem{coro}[theo]{Corollary}
\newtheorem{lemm}[theo]{Lemma}

\newtheorem*{theom}{Theorem}

\numberwithin{equation}{section}
\makeindex
\addto\extrasfrench{\providecommand{\fg}{\ifdim\lastskip>\z@\unskip\fi~\frqq}}

\newcommand{\twolines}[2][c]{
  \begin{tabular}[#1]{@{}c@{}}#2\end{tabular}}  %pour avoir deux lignes aux cellules tabular

\makeindex

\begin{document}
\title{Motivic unipotent fundamental groupoid of $\mathbb{G}_{m} \setminus \mu_{N}$ for $N=2,3,4,6,8$ and Galois descents.}
\author{Claire \textsc{Glanois}}

\maketitle

\begin{abstract}

We study Galois descents for categories of mixed Tate motives over $\mathcal{O}_{N}[1/N]$, for $N\in \left\{2, 3, 4, 8\right\}$ or  $\mathcal{O}_{N}$ for $N=6$, with $\mathcal{O}_{N}$ the ring of integers of the $N^{\text{th}}$ cyclotomic field, and construct families of motivic iterated integrals with prescribed properties. In particular this gives a basis of multiple zeta values via multiple zeta values at roots of unity $\mu_{N}$. It also gives a new proof, via Goncharov's coproduct, of Deligne's results ($\cite{De}$): the category of mixed Tate motives over $\mathcal{O}_{k_{N}}[1/N]$, for $N\in \left\{2, 3, 4,8\right\}$ is spanned by the motivic fundamental groupoid of $\mathbb{P}^{1}\setminus\left\{0,\mu_{N},\infty \right\}$ with an explicit basis. By applying the period map, we obtain a generating family for multiple zeta values relative to $\mu_{N}$.

\end{abstract}

\tableofcontents

\section{\textsc{Introduction}}

The goal of this paper is to study the Galois action on the motivic fundamental groupoid of $\mathbb{P}^{1}\setminus\left\{0,\mu_{N},\infty \right\}$ for some particular values of $N$: $N\in \left\{2^{a}3^{b}, a+2b \leq 3 \right\}= \left\{1,2, 3, 4, \textquoteleft 6 \textquoteright, 8\right\}$.\footnote{The quotation marks underline that we consider the unramified category for $N=6$.}\\
For such a fixed $N$, let $k_{N}=\mathbb{Q}(\xi_{N})$, where $\xi_{N}\in\mu_{N}$ is a primitive $N^{\text{th}}$ root of unity, and $\mathcal{O}_{N}$ is the ring of integers of $k_{N}$. The subscript or exponent $N$ will be omitted when it is not ambiguous.\\ 

Recall that \textit{multiple zeta values relative to $\mu_{N}$} (periods of the corresponding motivic multiple zeta values) are given by the coefficients of a version of Drinfeld's associator, which are explicitly:
\begin{equation}\label{eq:mzv} \text{     }  \zeta\left(x_{1}, \cdots , x_{p} \atop  \epsilon_{1} , \cdots ,\epsilon_{p} \right)\mathrel{\mathop:}= \sum_{0<n_{1}<n_{2} \cdots <n_{p}} \frac{\epsilon_{1}^{n_{1}} \cdots \epsilon_{p}^{n_{p}}}{n_{1}^{x_{1}} \cdots n_{p}^{x_{p}}} \text{, } \epsilon_{i}\in \mu_{N} \text{, } (x_{p},\epsilon_{p})\neq (1,1).
\end{equation}
The weight is $\omega=\sum x_{i}$ and the depth is $p$. Denote by $\mathcal{Z}^{N}$ the $\mathbb{Q}$-vector space spanned by these multiple zeta values at arguments $x_{i}\in \mathbb{N}$, $\epsilon_{i}\in\mu_{N}$. We will consider the motivic versions of those multiple zeta values (MMZV), denoted $\zeta^{\mathfrak{m}}$ which span the $\mathbb{Q}$-vector space of motivic multiple zetas relative to $\mu_{N}$, denoted $\mathcal{H}^{N}$. There is a surjective homomorphism called the \textit{period map}, conjectured to be an isomorphism:
\begin{equation}\label{eq:per} per :\mathcal{H}^{N} \rightarrow \mathcal{Z}^{N} \text{ ,  } \zeta^{\mathfrak{m}} (\cdot) \mapsto \zeta ( \cdot ).
\end{equation}
The period map will induce for each result for a basis for MMZV$_{\mu_{N}}$ a corresponding result for a generating family for MZV relative to $\mu_{N}$.\\  
  
  Furthermore, $\mathcal{H}^{N}$ is an Hopf comodule with an explicit coaction $\Delta$ given by Goncharov ($\cite{Go}$) and extended by F. Brown ($\cite{Br2}$). And for each $N, N'$ with $N' | N$ there are Galois groups $\mathcal{G}_{N} $ acting on $\mathcal{H}^{N} $ and Galois descents determined by this coaction:

   \begin{figure}[H]
\centering
\begin{equation}\label{eq:descent} \xymatrix{
\mathcal{H}^{N}  \\
\mathcal{F}_{i}\mathcal{H}^{N} \ar[u]^{\mathcal{G}_{i}}\\
\mathcal{F}_{0}\mathcal{H}^{N} =\mathcal{H}^{N'} \ar@/^2pc/[uu]^{\mathcal{G}_{0}=\mathcal{G}^{N/N'}}  \ar[u] \\
\mathbb{Q} \ar[u]^{\mathcal{G}^{N'}} \ar@/_2pc/[uuu]_{\mathcal{G}^{N}} }  \quad \quad 
\begin{array}{l}
 (\mathcal{H}^{N})^{\mathcal{G}_{i}}=\mathcal{F}_{i}\mathcal{H}^{N} \\
\\
 \begin{array}{llll}
 \mathcal{G}^{N/N'}=\mathcal{G}_{0} & \supset \mathcal{G}_{1} & \supset \cdots  &\supset \mathcal{G}_{i} \cdots\\
 & & & \\
 \mathcal{H}^{N'}= \mathcal{F}_{0}\mathcal{H}^{N} &  \subset  \mathcal{F}_{1 }\mathcal{H}^{N} & \subset \cdots  & \subset \mathcal{F}_{i}\mathcal{H}^{N}  \cdots.
 \end{array}
\end{array}
\end{equation}
\caption{Representation of a Galois descent.}\label{fig:descent}
\end{figure}

\paragraph{Outlines of the Results.}

Consider the Tannakian category of mixed Tate motives over $\mathcal{O}_{N}[1/N]$ (cf. $\cite{Go2}$, $\cite{DG}$):
\begin{equation}\label{eq:mt} \mathcal{MT}_{N}\mathrel{\mathop:}=\left\lbrace \begin{array}{ll} \mathcal{MT}(\mathcal{O}_{N}[1/N]) & \text{ for } N \text{ fixed in } \lbrace 2,3,4,8\rbrace.\\
\mathcal{MT}(\mathcal{O}_{6}) & \text{ for } N=6. 
\end{array}
\right. 
\end{equation}
Denote by $\mathcal{G}^{\mathcal{MT}}= \mathbb{G}_{m} \ltimes \mathcal{U}^{\mathcal{MT}}$ its Tannaka group (cf. $\cite{Mi}$) with respect to the canonical fiber functor which is defined over $\mathbb{Q}$ and by $\mathcal{A}^{\mathcal{MT}}=\mathcal{O}(\mathcal{U}^{\mathcal{MT}}) $ its fundamental Hopf algebra and by $\mathcal{H}^{\mathcal{MT}}$ the free $\mathcal{A}^{\mathcal{MT}}$-comodule:
\begin{equation}\label{eq:hmt} \mathcal{H}^{\mathcal{MT}} = 
\left\{
\begin{array}{l}
  \mathcal{A}^{\mathcal{MT}} \otimes_{\mathbb{Q}} \mathbb{Q}[ t] \text{ for } N>2 .\\
   \mathcal{A}^{\mathcal{MT}}\otimes_{\mathbb{Q}} \mathbb{Q} [t^{2}] \text{ for } N=1,2
  \end{array} \subset \mathcal{O}(\mathcal{G}^{\mathcal{MT}})=\mathcal{A}^{\mathcal{MT}}\otimes_{\mathbb{Q}} \mathbb{Q}[ t,t^{-1}].
  \right .
  \end{equation}
%  where $(2i\pi)^{j,\mathfrak{m}}$ has a trivial coaction, degree $j$, and $(2i\pi)^{j}$ for period
There is a notion of \textbf{\textit{motivic multiple zeta values}} which form an algebra $\mathcal{H}^{N}$ which embeds non canonically into $\mathcal{H}^{\mathcal{MT}_{N}}$ with $(2i\pi)^{\mathfrak{m}}\rightarrow t$. For those specific values of $N$, it is an isomorphism (by F. Brown for $N=1$ and for $N=2,3,4,\textquoteleft 6 \textquoteright,8$ by Deligne $\cite{De}$ or by the Corollary $1.2$ which follow). For the sake of the introduction, we will sometimes write $\mathcal{H}$ and forget the distinction.\\

We define recursively on $i$ increasing motivic filtrations on $\mathcal{H}^{N}$, one for each descent $(\mathcal{d})=(k_{N}/k_{N'}, M/M')$, called \textbf{\textit{motivic levels}}, $\mathcal{F}^{\mathcal{d}}_{i}$, stable under the action of $\mathcal{G}^{\mathcal{MT}_{N}}$, by sub-$\mathbb{Q}$-vector spaces. The exponent $k_{N}/k_{N'}$ indicates the change of cyclotomic field and $M/M'$ the change of ramification. The $0^{\text{th}}$ level $\mathcal{F}^{\mathcal{d}}_{0}$, corresponds to invariants under the group $\mathcal{G}_{N/N'}$ as above and the $i^{\text{th}}$ level $\mathcal{F}^{\mathcal{d}}_{i}$, can be seen as the $i^{\text{th}}$ ramification space corresponding to generalised Galois descents. The associated quotients are denoted: 
\begin{equation}\label{eq:quotient} \boldsymbol{\mathcal{H}^{\geq i}} \mathrel{\mathop:}=  \mathcal{H}/ \mathcal{F}_{i-1}\mathcal{H}\text{  ,  } \mathcal{H}^{\geq 0}=\mathcal{H}.
\end{equation}
The exponent $k_{N}/k_{N'}, M/M'$ if not ambiguous will be omitted when we look at a specific descent.\\
\\
\texttt{Example, for $N=2$:} The action of the Lie algebra of the Galois group factors through certain operators $D_{2r+1}$ which are obtained from the formula for the coaction $\Delta$ by restricting the left hand side to weight $2r+1$. The Galois descent between $\mathcal{H}^{2}$ and $\mathcal{H}^{1}$ is then measured by $D_{1}$, that is to say:
$\mathcal{F}_{-1}\mathcal{H}^{2}=0$ and $\mathcal{F}_{i}\mathcal{H}^{2}$ is the largest sub-module such that $\mathcal{F}_{i}/ \mathcal{F}_{i-1}$ is killed by $D_{1}$.\\
Motivic Euler sums belonging to the $0^{\text{th}}$-level of this filtration are sometimes called \textit{unramified}  or \textit{honorary} motivic multiple zeta values and are in $\mathcal{H}^{1}$. Some of these periods have been studied notably by D. Broadhurst (cf. $\cite{BBV}$) among others.\\

Part of our results (if we restrict to the $0^{\text{th}}$ level of the filtrations) are illustrated by the following diagrams:\\

 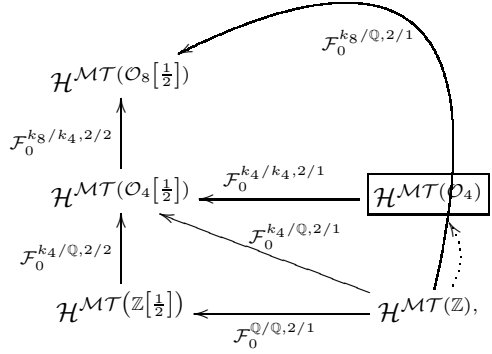
\begin{figure}[H]
\centering
$$\xymatrixcolsep{5pc}\xymatrix{
\mathcal{H}^{\mathcal{MT}(\mathcal{O}_{8}\left[ \frac{1}{2}\right] )} &   \\
\mathcal{H}^{\mathcal{MT}(\mathcal{O}_{4}\left[ \frac{1}{2}\right] )} \ar[u]^{\mathcal{F}^{k_{8}/k_{4},2/2}_{0}}   &  \text{\framebox[1.1\width]{$\mathcal{H}^{\mathcal{MT}(\mathcal{O}_{4})}$}} \ar[l]_{\mathcal{F}^{k_{4}/k_{4},2/1}_{0}} \\
\mathcal{H}^{\mathcal{MT}\left( \mathbb{Z}\left[ \frac{1}{2}\right] \right) } \ar[u]^{\mathcal{F}^{k_{4}/\mathbb{Q},2/2}_{0}} &  \mathcal{H}^{\mathcal{MT}(\mathbb{Z}),} \ar[l]^{\mathcal{F}^{\mathbb{Q}/\mathbb{Q},2/1}_{0}} \ar[lu]_{\mathcal{F}^{k_{4}/\mathbb{Q},2/1}_{0}} \ar@{.>}@/_1pc/[u] \ar@/_7pc/[uul] ^{\mathcal{F}^{k_{8}/\mathbb{Q},2/1}_{0}}
}
$$
\caption{\textsc{The cases $N=1,2,4,8$}. }\label{fig:d248}
\end{figure}

\begin{figure}[H]
$$\xymatrixcolsep{5pc}\xymatrix{
& \mathcal{H}^{\mathcal{MT}(\mathcal{O}_{6})}   \\
\mathcal{H}^{\mathcal{MT}\left( \mathcal{O}_{3}\left[ \frac{1}{3}\right] \right) } & \text{\framebox[1.1\width]{$\mathcal{H}^{\mathcal{MT}(\mathcal{O}_{3})}$}} \ar[l]_{\mathcal{F}^{k_{3}/k_{3},3/1}_{0}} \\
\text{\framebox[1.1\width]{$\mathcal{H}^{\mathcal{MT}\left( \mathbb{Z}\left[ \frac{1}{3}\right] \right) }$}}  \ar[u]^{\mathcal{F}^{k_{3}/\mathbb{Q},3/3}_{0}} &  \mathcal{H}^{\mathcal{MT}(\mathbb{Z})}  \ar[lu]_{\mathcal{F}^{k_{3}/\mathbb{Q},3/1}_{0}} \ar@{.>}@/_1pc/[u] \ar@/_3pc/[uu]_{\mathcal{F}^{k_{6}/\mathbb{Q},1/1}_{0}}
}
$$
\caption{\textsc{The cases $N=1,3,\textquoteleft 6 \textquoteright$}. }\label{fig:d36}
\end{figure}

\textsc{Remarks:}
\begin{enumerate}
\item[$\cdot$] The vertical arrows represent the change of field and the horizontal arrows the change of ramification. The full arrows are the descents made explicit in this paper.\\
More precisely, for each arrow $A \stackrel{\mathcal{F}_{0}}{\leftarrow}B$ in the above diagrams, we give a basis $\mathcal{B}^{A}_{n}$ of $\mathcal{H}_{n}^{A}$, and a basis of $\mathcal{H}_{n}^{B}= \mathcal{F}_{0} \mathcal{H}_{n}^{A}$ in terms of the elements of $\mathcal{B}_{n}^{A}$.
 \item[$\cdot$] The framed spaces $\mathcal{H}^{\cdots}$ appearing in these diagrams are not known to be associated to a fundamental group and there is presently no other known way to construct those periods. For instance, we obtain by descent, a basis for $\mathcal{H}_{n}^{\mathcal{MT}(\mathbb{Z}[\frac{1}{3}])}$ in terms of the basis of $\mathcal{H}_{n}^{\mathcal{MT}(\mathcal{O}_{3}[\frac{1}{3}])}$.
% In those cases, the $0^{\text{th}}$-level of the filtrations gives us a free family of motivic multiple zeta values and $\mathcal{H}(k_{N'},M)$ is simply defined as the vector space spanned by this family. The notation suggests only that the corresponding field is $k_{N'}$, and the ramification $M'$...
\end{enumerate}

More precisely, for $N\in \left\{2, 3, 4,\textquoteleft 6 \textquoteright, 8\right\}$, we define a particular family $\mathcal{B}^{N}$ of motivic multiple zeta values relative to $\mu_{N}$ with different notions of \textbf{\textit{level}} on the basis elements, one for each Galois descent considered above:
\begin{equation}\label{eq:base} \mathcal{B}^{N}\mathrel{\mathop:}=\left\{ \zeta^{\mathfrak{m}}\left(x_{1}, \cdots x_{p-1}, x_{p} \atop \epsilon_{1}, \cdots, \epsilon_{p-1}, \epsilon_{p}\xi_{N}\right) (2\pi i)^{s ,\mathfrak{m}} \text{ , }  x_{i}\in\mathbb{N}^{\ast} , s\geq 0 ,  \left\{
\begin{array}{l}
 x_{i} \geq 1 \text{ odd ,  } \epsilon_{i}=1 \text{ and } s \text{ even if } N=2 \\
 x_{i} \geq 1  \text{ ,  } \epsilon_{i}=1 \text{  if } N=3,4\\
   x_{i} >1 \text{ ,  } \epsilon_{i}=1 \text{  if } N=\textquoteleft 6 \textquoteright\\
   x_{i}\geq 1  \text{ ,  } \epsilon_{i}=\pm 1\text{  if } N=8
   \end{array}
\right. . \right\}
\end{equation}
  Denote by $\mathcal{B}_{n,p,i}$ the subset of elements with weight $n$, depth $p$ and level $i$. \\
  \texttt{Examples}: 
\begin{itemize}
  \item[$\cdot$]  For $N=2$, the basis for the motivic Euler sums:
  $$\mathcal{B}^{2}\mathrel{\mathop:}=\left\{ \zeta^{\mathfrak{m}}\left(2y_{1}+1, \cdots , 2 y_{p}+1 \atop 1, 1, \cdots, 1, -1\right) \zeta^{\mathfrak{m}} (2)^{s}, y_{i} \geq 0, s\geq 0 \right\}. $$
  The level is defined to be the number of $y_{i}'s$ equal to $0$.
  \item[$\cdot$]  For $N=4$,
  $\mathcal{B}^{4}\mathrel{\mathop:}= \left\{   \zeta^{\mathfrak{m}}\left(x_{1}, \cdots , x_{p} \atop 1,1, \cdots, 1, \sqrt{-1}\right)  (2\pi i)^{s ,\mathfrak{m}}, s\geq 0, x_{i} >0 \right\} $. Here, the level is the number of even $x_{i}'s$ if we focus on the descent from $\mathcal{H}^{4}$ to $\mathcal{H}^{2}$, or the number of even $x_{i}'s$ plus the number of $x_{i}'s$ equal to 1 if we focus on the Galois descent from $\mathcal{H}^{4}$ to $\mathcal{H}^{1}$.
  \item[$\cdot$]   For $N=8$ the level includes the number of $\epsilon_{i}'s$ equal to $-1$, etc.\\
  \end{itemize}  
Fix a descent $(\mathcal{d})=(k_{N}/k_{N'}, M/M')$ among those considered above and let:
\begin{equation}\label{eq:z1p} \mathbb{Z}_{1[P]} \mathrel{\mathop:}= \frac{\mathbb{Z}}{1+ P\mathbb{Z}}=\left\{ \frac{a}{1+b P}, a,b\in\mathbb{Z} \right\} \text{ with  } P=2 \text{ for } N=4,8 \text{ and }P=3 \text{ for } N=3,6.
\end{equation}
The previous quotients $\mathcal{H}^{\geq i}$, respectively filtrations $\mathcal{F}_{i}$ associated to the descent $\mathcal{d}$, will match with the sub-families restricted to the level (associated to $\mathcal{d}$) $\mathcal{B}_{n,p, \geq i}$, respectively $\mathcal{B}_{n,p, \leq i}$. Indeed, we prove the following (cf. Theorem $4.3$ slightly more precise):
\begin{theom}
\begin{itemize}
\item[$(i)$] $\mathcal{B}_{n,\leq p, \geq i}$ is a basis of $\mathcal{F}_{p}^{D} \mathcal{H}_{n}^{\geq i}$.
\item[$(ii)$] $\mathcal{B}_{n,\cdot, \geq i} $ is a basis of $\mathcal{H}^{\geq i}_{n}$ and $\mathcal{B}_{n, p, \geq i}$ is a basis of $gr_{p}^{D} \mathcal{H}_{n}^{\geq i}$ on which it defines a $\mathbb{Z}_{1[P]}$-structure:\\
Each $\zeta^{\mathfrak{m}}\left( z_{1}, \cdots , z_{p} \atop \epsilon_{1}, \cdots, \epsilon_{p}\right)$ decomposes in $gr_{p}^{D} \mathcal{H}_{n}^{\geq i}$ as a $\mathbb{Z}_{1[P]}$-linear combination of $\mathcal{B}_{n, p, \geq i}$ elements.\\
	\item[$(iii)$] We have the two split exact sequences in bijection:
$$ 0\longrightarrow \mathcal{F}_{i}\mathcal{H}_{n} \longrightarrow \mathcal{H}_{n} \stackrel{\pi_{0,i+1}} {\rightarrow}\mathcal{H}_{n}^{\geq i+1} \longrightarrow 0$$
$$ 0 \rightarrow \langle \mathcal{B}_{n, \cdot, \leq i} \rangle_{\mathbb{Q}} \rightarrow \langle\mathcal{B}_{n} \rangle_{\mathbb{Q}} \rightarrow \langle (\mathcal{B}_{n, \cdot, \geq i+1} \rangle_{\mathbb{Q}} \rightarrow 0 .$$
		\item[$(iv)$] A basis for the filtration spaces $\mathcal{F}_{i} \mathcal{H}_{n}$:
$$\cup_{p} \left\{ x+ cl_{n, \leq p, \geq i+1}(x), x\in \mathcal{B}_{n, p, \leq i} \right\},$$
	$$\text{ where } cl_{n,\leq p,\geq i}: \langle\mathcal{B}_{n, p, \leq i-1}\rangle_{\mathbb{Q}} \rightarrow \langle\mathcal{B}_{n, \leq p, \geq i}\rangle_{\mathbb{Q}} \text{ such as } x+cl_{n,\leq p,\geq i}(x)\in \mathcal{F}_{i-1}\mathcal{H}_{n}.$$
\item[$(v)$] A basis for the graded space $gr_{i} \mathcal{H}_{n}$:
$$\cup_{p} \left\{ x+ cl_{n, \leq p, \geq i+1}(x), x\in \mathcal{B}_{n, p, i} \right\}.$$
\end{itemize}
\end{theom}
The linear independence is obtained first in the depth graded, and the proof relies on the bijectivity of the following map $\partial^{i, \mathcal{d}}_{n,p}$ by an argument involving $2$ or $3$ adic properties, where:
\begin{equation}\label{eq:derivintro}
\partial^{i, \mathcal{d}}_{n,p}: gr_{p}^{D} \mathcal{H}_{n}^{\geq i} \rightarrow  \oplus_{r<n} \left( gr_{p-1}^{D} \mathcal{H}_{n-r}^{\geq i-1}\right) ^{\oplus c ^{\mathcal{d}}_{r}} \oplus_{r<n} \left( gr_{p-1}^{D} \mathcal{H}_{n-r}^{\geq i}\right) ^{\oplus c^{\backslash\mathcal{d}}_{r}} \text{, } c ^{\mathcal{d}}_{r}, c^{\backslash\mathcal{d}}_{r}\in\mathbb{N}.
\end{equation}
is obtained from the depth and weight graded part of the coaction, followed by a projection for the left side (by depth $1$ results of Deligne and Goncharov, cf. $§ 3.1$), and by passing to the level quotients. Once the freedom obtained, the generating property is obtained from counting dimensions, since the K-theory would give an upper bound for the dimensions.\\

This main theorem generalizes in particular (with $i=0$) a result of P. Deligne ($\cite{De}$) \footnote{The basis $\mathcal{B}$, in the case $\left\{3, 4,8\right\}$ is identical to P. Deligne's in $\cite{De}$, and for $N=2$ is a linear basis analogous to his algebraic basis which is formed by Lyndon words in the odd positive integers (with $\ldots 5 \leq3 \leq 1$).}:
\begin{coro} The map $\mathcal{G}^{\mathcal{MT}} \rightarrow \mathcal{G}^{\mathcal{MT}'}$ is an isomorphism.\\
Elements of $\mathcal{B}_{n}$ form a basis of $ \mathcal{H}_{n}$, the space of motivic multiple zeta values relative to $\mu_{N}$.
\end{coro}

The period map, $per: \mathcal{H} \rightarrow \mathbb{C}$, induces the following result for $N=2,3,4,8$:
\begin{center}
Each multiple zeta value relative to $N^{\text{th}}$ roots of unity is a $\mathbb{Q}$-linear combination of multiple zeta values of the same weight of type $\mathcal{B}^{N}$.
\end{center}
\textsc{Remark: } For $N=6$ the result remains true if we restrict to iterated integrals relative not to all $6^{\text{th}}$ roots of unity but only to those relative to primitive roots.\\
The previous theorem (with $i=0$), will also give us the Galois descent from $\mathcal{H}^{\mathcal{MT}_{N}}$ to $\mathcal{H}^{\mathcal{MT}_{N'}}$, according to the level filtration considered, with $N'| N$:
\begin{coro}
A basis for the space $\mathcal{H}^{N'}_{n}$, of motivic multiple zeta values relative to $\mu_{N'}$ is formed by motivic multiple zeta values relative to $\mu_{N}$ $\in \mathcal{B}^{N}$ of level $0$ each corrected by a $\mathbb{Q}$-linear combination of motivic multiple zeta values relative to $\mu_{N}$ of level greater than or equal to $1$:  
$$\left\{ x+ cl_{n,\cdot, \geq 1}(x), x\in \mathcal{B}^{N}_{n, \cdot, 0} \right\}.$$
\end{coro}

\texttt{Example, $N=2$:} A basis for motivic multiple zeta values is formed by:
$$ \left\lbrace  \zeta^{\mathfrak{m}}(2x_{1}+1, \cdots, \overline{2x_{p}+1})\zeta^{\mathfrak{m}}(2)^{s} + \sum_{\exists i, y_{i}=0 \atop q\leq p} \alpha_{\textbf{y}}^{\textbf{x}} \zeta^{\mathfrak{m}}(2y_{1}+1, \cdots, \overline{2y_{q}+1})\zeta^{\mathfrak{m}}(2)^{s} \text{  ,  } x_{i}>0, \alpha^{\textbf{x}}_{\textbf{y}}\in\mathbb{Q} \right\rbrace .$$
\textsc{Remarks}:
\begin{itemize}
		\item[$\cdot$] Recall that each basis for motivic multiple zeta values at roots of unity gives a generating family for (simple) multiple zeta values at roots of unity, by the period map.
	\item[$\cdot$] Descent can be calculated explicitly in small depth, less than or equal to $3$, as we will explain in the appendice.
For instance, for $N=2$, the following linear combination is a motivic MZV:
$$\zeta^{\mathfrak{m}}(3,3,\overline{3})+ \frac{774}{191} \zeta^{\mathfrak{m}}(1,5, \overline{3})  - \frac{804}{191} \zeta^{\mathfrak{m}}(1,3, \overline{5})  + \frac{450}{191}\zeta^{\mathfrak{m}}(1,1, \overline{7})  -6 \zeta^{\mathfrak{m}}(3,1,\overline{5}).$$
	
In the general case, we could make the part of maximal depth of $cl(x)$ explicit (by inverting a matrix with binomial coefficients) but the motivic methods do not enable us to describe the other coefficients for terms of lower depth.
\end{itemize}

\paragraph{\textsc{Contents}}
The second section points out some generalities and definitions about motivic multiple zeta values at roots of unity, and motivic iterated integrals to set up the background of this paper.
The third section deals with both the coaction and the filtration by depth, which are essentials to the results here and presents general results on Galois descents, and useful criteria.
The fourth section states and proves the main result announced in the introduction, for descents considered in Figures $\ref{fig:d248}, \ref{fig:descent}$, with specifications for each case.
The appendice provides some explicit examples in small depths ($2$ and $3$).

\paragraph{\textsc{Aknowledgements}}
The author thanks Francis Brown for many discussions and corrections on this work, and Pierre Cartier for a careful reading and helpful comments. \\
This work was supported by ERC Grant 257638. 

\section{\textsc{Motivic multiple zeta values at roots of unity}}

\paragraph{Tannakian category of Mixed Tate Motives.}
Recall that $\mathcal{MT}_{N}$ is a Tannakian category equipped with a weight filtration $W_{r}$ indexed by even integers such that $gr_{-2r}^{W}(M)$ is a sum of copies of $\mathbb{Q}(r)$ for $M\in \mathcal{MT}_{N}$. This defines a canonical fiber functor:
\begin{equation} \omega: \mathcal{MT}_{N} \rightarrow Vec_{\mathbb{Q}},  M \mapsto \oplus \omega_{r}(M)
\end{equation}
$$\omega_{r}(M)\mathrel{\mathop:}= Hom_{\mathcal{MT}(k)}(\mathbb{Q}(r), gr_{-2r}^{W}(M)) \text{ , } gr_{-2r}^{W}(M)= \mathbb{Q}(r)\otimes \omega_{r}(M) .$$
The de Rham functor $\omega_{dR}$ here is not defined over $\mathbb{Q}$ but on $k_{N}$ and $\omega_{dR}=\omega \otimes_{\mathbb{Q}} k_{N}$, so the de Rham realisation of an object $M$ is $M_{dR}=\omega(M)\otimes_{\mathbb{Q}} k_{N}$.\\
The Betti fiber functor depends on the embedding $\sigma: k_{N}\hookrightarrow \mathbb{C}$ (fixed here):
\begin{equation}\omega_{B,\sigma}: \mathcal{MT}_{N} \rightarrow Vec_{\mathbb{Q}}.\end{equation}
There are canonical comparison isomorphisms between those functors:
$$comp_{B,dR}:\omega_{dR}(M)\otimes_{k_{N}} \mathbb{C}  \rightarrow \omega_{B}(M)\otimes_{\mathbb{Q}} \mathbb{C} \text{ and } comp_{dR,B} : \omega_{B}(M)\otimes_{\mathbb{Q}} \mathbb{C}  \rightarrow \omega_{dR}(M)\otimes_{k_{N}} \mathbb{C} $$

The \textit{\textbf{motivic Galois groups}} are defined by $\mathcal{G}_{B}\mathrel{\mathop:}=Aut^{\otimes}(\omega_{B})$ and $\mathcal{G}^{\mathcal{MT}}\mathrel{\mathop:}= Aut^{\otimes} \omega $ and $P_{\omega, B}\mathrel{\mathop:}=Isom(\omega_{B},\omega)$ resp. $P_{B,\omega}\mathrel{\mathop:}=Isom(\omega,\omega_{B})$ are $(\mathcal{G}, \mathcal{G}_{B})$ resp. $(\mathcal{G}_{B}, \mathcal{G})$ bitorsors.\\
\\
By the Tannakian dictionnary, $\mathcal{MT}_{N}$ is equivalent to the category of representations of $\mathcal{G}^{\mathcal{MT}} $, which decomposes as $\mathcal{G}^{\mathcal{MT}}= \mathbb{G}_{m} \ltimes \mathcal{U}^{\mathcal{MT}}$, where $\mathcal{U}^{\mathcal{MT}}$ is a pro-unipotent group scheme defined over $\mathbb{Q}$.
\paragraph{Dimensions.}
The algebraic $K$-theory will provide an upper bound for the dimensions of motivic periods. Indeed, it is proved (with the results of Beilinson and Borel, cf. $\cite{DG}$, and Levine, $\cite{Le}$) that: 
\begin{equation}
\begin{array}{ll}
  Ext_{\mathcal{MT}_{N}}^{1} (\mathbb{Q}(0), \mathbb{Q}(1)) = K_{1}(\mathcal{O}_{k_{N}}[\frac{1}{M}]) \otimes \mathbb{Q} =  (\mathcal{O}_{k_{N}}[\frac{1}{M}])^{\ast} \otimes \mathbb{Q}& \\
  Ext_{\mathcal{MT}_{N}}^{1} (\mathbb{Q}(0), \mathbb{Q}(n)) = K_{2n-1}(\mathcal{O}_{k_{N}}[\frac{1}{M}]) \otimes \mathbb{Q} = K_{2n-1}(k_{N}) \otimes \mathbb{Q}  & \text{ for } n >1 .\\
  Ext_{\mathcal{MT}_{N}}^{i} (\mathbb{Q}(0), \mathbb{Q}(n)) =0 & \text{ for } i>1 \text{ or } n\leq 0 . 
  \end{array}
\end{equation}  
In the cases studied here, $M$ is equal to $N$ if $N=2,3,4,8$ or equal to 1, for $N=6$ since we consider the unramified category.\\
Let $\mathfrak{u}$ denote the completion of the pro-nilpotent graded Lie algebra of the pro-unipotent group $\mathcal{U}^{\mathcal{MT}}$, defined by a limit; $\mathfrak{u}$ is free since $Ext^{2}=0$ and graded with positive degrees from the $\mathbb{G}_{m}$-action. Furthermore:
\begin{equation}
\label{LieCyclo}
\mathfrak{u}^{ab}= \bigoplus (Ext^{1}_{\mathcal{MT}_{N}} (\mathbb{Q}(0), \mathbb{Q}(n))^{\vee} \text{ in degree } n)= \bigoplus (K_{2n-1}(\mathcal{O}_{k_{N}}[1/M])^{\vee} \text{ in degree } n).
\end{equation}
Hence the fundamental Hopf algebra is:
\begin{equation}
\label{IsomCyclo}
\mathcal{A}^{\mathcal{MT}}\mathrel{\mathop:}=\mathcal{O}(\mathcal{U}^{\mathcal{MT}})\cong (U^{\wedge} (\mathfrak{u}))^{\vee} \cong T(\oplus_{n\geq 1} K_{2n-1}(\mathcal{O}_{k_{N}}[1/M]) \otimes \mathbb{Q}).
\end{equation}

 Let $p(N)$ denote the number of prime factors of $N$ and $\varphi$ Euler's indicator function. For $M|N$ (cf.$\cite{Bo}$) such as all prime dividing $M $ are inert, using Dirichlet $S$-unit theorem when $n=1$:
\begin{equation}\label{dimensionk}\dim K_{2n-1}(\mathcal{O}_{k_{N}} [1/M]) \otimes \mathbb{Q} = \left\{
\begin{array}{ll}
  1 & \text{ if } N =1 \text{ or } 2 , \text{ and } n \text{ odd }, (n,N) \neq (1,1) .\\
  0 & \text{ if } N =1 \text{ or } 2 , \text{ and } n \text{ even } .\\
  a_{N}\mathrel{\mathop:}=\frac{\varphi(N)}{2}+ p(M)-1& \text{ if } N >2, n=1 . \\
  b_{N}\mathrel{\mathop:}=\frac{\varphi(N)}{2} & \text{ if } N >2 , n>1  .
\end{array}
\right.
\end{equation}
Hence the dimensions for the fundamental Hopf algebra and the free $\mathcal{A}^{\mathcal{MT}}$-comodule $\mathcal{H}^{\mathcal{MT}}$:
\begin{lemm}
For $N>2$, $M \mid N$ such as all primes dividing $M$ are inert:
\begin{flushleft}
$\mathcal{A}^{\mathcal{MT}}$ is a cofree commutative graded Hopf algebra cogenerated by $a_{N}$ elements $f^{\bullet}_{1}$ in degree 1, and $b_{N}$ elements $f^{\bullet}_{r}$ in degree $r>1$, and we have a non canonical isomorphism of comodules:
\end{flushleft}
$$\mathcal{H}^{\mathcal{MT}}=\mathcal{A}^{\mathcal{MT}} \otimes \mathbb{Q}\left[ t \right] \xrightarrow[\sim]{\phi} H \mathrel{\mathop:}= \mathbb{Q} \left\langle \left(f^{j}_{1}\right) _{1\leq j \leq a_{N}}, \left( f^{j}_{r}\right)_{r>1\atop 1\leq j \leq b_{N}} \right\rangle  \otimes \mathbb{Q}\left[ f'_{1} \right]. $$
There is a recursive formula for $d^{N}_{n}$, the dimension of $\mathcal{H}^{\mathcal{MT}_{N}}_{n}$:
$$d^{N}_{n}=1 + a_{N} d_{n-1}+ b_{N}\sum_{i=2}^{n} d_{n -i}= (a_{N}+1)d_{n-1}+ (b_{N}-a_{N})d_{n -2} \text{  ,  } d_{0}=1 \text{  ,  } d_{1}=a_{N}+1.$$
Hence the Hilbert series for the dimensions of $\mathcal{H}^{\mathcal{MT}_{N}}$ is:
$$h_{N}(t)\mathrel{\mathop:}=\sum_{k} d_{k}^{N} t^{k}=\frac{1}{1-(a_{N}+1)t+ (a_{N}-b_{N})t^{2}}. $$
\end{lemm}
\textsc{Remarks:}
\begin{itemize}
\item[$\cdot$] In the general case of the category $\mathcal{MT}(\mathcal{O}_{k_{N}} \left[ \frac{1}{M}\right] )$ where all prime dividing $M$ are not inert, the formula above remains true, with $a_{N,M}\mathrel{\mathop:}= \frac{\varphi(N)}{2} + n_{\mathfrak{p}_{M}}-1$, where $ n_{\mathfrak{p}_{M}}$ is the number of different primes above the primes dividing $M$. However, in the cases considered here, $ n_{\mathfrak{p}_{M}}$ is simply $p(N)$.
\item[$\cdot$] In particular, those dimensions (for $\mathcal{H}^{\mathcal{MT}_{\Gamma_{N}}}$ with $a_{N,N}$) are an upper bound for the dimensions of motivic MZV$_{\mu_{N}}$ (i.e. of $\mathcal{H}^{N}$), and hence of MZV$_{\mu_{N}}$ by the period map.
\item[$\cdot$] The generators $\sigma_{r}=f_{r}^{\vee}$ of the graded Lie algebra $\mathfrak{u}$ are indeed non canonical, only their classes in the abelianization are.\\
\end{itemize}
\texttt{Examples:}
\begin{itemize}
\item[$\cdot$] For the unramified category $\mathcal{MT}(\mathcal{O}_{N})$:
$$d_{n}= \frac{\varphi(N)}{2}d_{n-1}+ d_{n-2}.$$
This suggests to look for a basis with $1$ (with $\frac{\varphi(N)}{2}$ choices of roots of unity) and $2$ (with 1 choice of roots of unity), in the \textit{Hoffman way}.\footnote{As for Hoffman basis, with $\zeta(\lbrace 2,3\rbrace^{\times})$ for $N=1$ where dimensions satisfy $d_{n}=d_{n-2}+d_{n-3}$, cf. $\cite{Br2}$.}
\item[$\cdot$]  For $M \mid N$ such that all primes dividing $M$ are inert, $ n_{\mathfrak{p}_{M}}=\nu(N)$. In particular, it is the case if $N=p^{r}$:
$$\text{ For } \mathcal{MT}\left( \mathcal{O}_{p^{r}}\left[  \frac{1}{p} \right] \right)  \text{  ,  } \quad d_{n}= \left( \frac{\varphi(N)}{2}+1\right) ^{n}.$$
Let us detail the cases $N=2,3,4,\textquoteleft 6 \textquoteright,8$:\\
\end{itemize}

\begin{tabular}{|c|c|c|c|}
    \hline
    & & & \\
   $N \backslash$ $d_{n}^{N}$& $A$ & Dimension relation $d_{n}^{N}$ & Hilbert series \\
  \hline
  $N=1$ & \twolines{$1$ generator in each odd degree $>1$\\
  $\mathbb{Q} \langle f_{3}, f_{5}, f_{7}, \cdots \rangle$ }  &\twolines{$d_{n}=d_{n-3} +d_{n-2}$,\\ $d_{2}=1$, $d_{1}=0$} & $ \frac{1}{1-t^{2}-t^{3}}$  \\
      \hline
   $N=2$\footnotemark[2] & \twolines{$1$ generator in each odd degree $\geq 1$\\
   $\mathbb{Q} \langle f_{1}, f_{3}, f_{5}, \cdots \rangle$ } &  \twolines{$d_{n}=d_{n-1} +d_{n-2}$\\ $d_{0}=d_{1}=1$}  & $ \frac{1}{1-t-t^{2}}$  \\
       \hline
  $N=3,4$ & \twolines{$1$ generator in each degree $\geq 1$\\
  $\mathbb{Q} \langle f_{1}, f_{2}, f_{3}, \cdots \rangle$ } & $d_{k}=2d_{k-1} = 2^{k}$ & $\frac{1}{1-2t}$  \\
      \hline
  $N=8$ & \twolines{$2$ generators in each degree $\geq 1$ \\ $\mathbb{Q} \langle f^{1}_{1}, f^{2}_{1}, f^{1}_{2}, f^{2}_{2}, \cdots \rangle$ } & $d_{k}= 3 d_{k -1}=3^{k}$ & $\frac{1}{1-3t}$  \\
      \hline
   \twolines{$N=6$ \\ $\mathcal{MT}(\mathcal{O}_{6}\left[\frac{1}{6}\right])$} & \twolines{$1$ in each degree $> 1$, $2$ in degree $1$\\
   $\mathbb{Q} \langle f^{1}_{1}, f^{2}_{1}, f_{2}, f_{3}, \cdots \rangle$ } & \twolines{$d_{k}= 3 d_{k -1} -d_{k-2}$, \\$d_{1}=3$} & $\frac{1}{1-3t+t^{2}}$ \\
       \hline
   \twolines{$N=6$ \\ $\mathcal{MT}(\mathcal{O}_{6})$} & \twolines{$1$ generator in each degree $>1$\\
   $\mathbb{Q} \langle f_{2}, f_{3}, f_{4}, \cdots \rangle$} & \twolines{$d_{k}= 1+ \sum_{i\geq 2} d_{k-i}$\\$=d_{k -1} +d_{k-2}$} & $ \frac{1}{1-t-t^{2}}$ \\
   \hline
\end{tabular}
\footnotetext[2]{For $N=2$, the dimensions are Fibonacci numbers.}

\paragraph{Fundamental groupoid.}
Let $\Pi_{0,1}\mathrel{\mathop:}= \pi^{dR}_{1}(\mathbb{P}^{1} \backslash \lbrace 0,\mu_{N},\infty \rbrace, \overrightarrow{1}_{0}, \overrightarrow{-1}_{1} )$ denote the de Rham realisation of the motivic torsor of paths from $0$ to $1$ on $\mathbb{P}^{1} -\left\{0,\mu_{N},\infty\right\}$, with tangential basepoints given by the tangent vectors $1$ at $0$ and $-1$ at $1$. It is the following functor:
\begin{equation}\label{eq:pi}\Pi_{0,1}: R \text{ a } \mathbb{Q}-\text{algebra } \mapsto \left\{S\in R<<e_{0}, (e_{\eta})_{\eta\in\mu_{N}}>>^{\times} | \Delta S= S\otimes S \right\} ,\end{equation}
i.e. the set of non-commutative formal series with $N+1$ generators which are group-like for the completed coproduct for which $e_{i}$ are primitive. It is dual to the shuffle $\shuffle$ relation between the coefficients of the series $S$. Its affine ring of regular functions is the graded algebra for the shuffle product:
\begin{equation}\label{eq:opi}
\mathcal{O}(\Pi_{0,1})\cong \mathbb{Q} \left\langle e^{0}, (e^{\eta})_{\eta\in\mu_{N}} \right\rangle .
\end{equation}

Denote by $\mathcal{MT}'_{N}$ the full Tannakian subcategory of $\mathcal{MT}_{N}$ generated by the motivic fundamental groupoid of $\mathbb{P}^{1}\setminus \left\{0,\mu_{N}, \infty\right\}$. Denote also by $\mathcal{G}=\mathbb{G}_{m} \ltimes \mathcal{U} $ its Galois group defined over $\mathbb{Q}$, $\mathcal{A}=\mathcal{O}(\mathcal{U})$ its fundamental Hopf algebra and $\mathcal{L}\mathrel{\mathop:}= \mathcal{A}_{>0} / \mathcal{A}_{>0} \cdot\mathcal{A}_{>0}$ the Lie coalgebra of indecomposable elements.

\paragraph{Motivic periods.}
A \textbf{\textit{motivic period}} in a tannakian category of mixed Tate motives $\mathcal{M}$ is (cf. $\cite{Del}$) a triplet $\left[M,v,\sigma \right]$, where $M\in Ind(\mathcal{M})$, $v\in\omega(M)$, $\sigma\in\omega_{B}(M)^{\vee}$.\\
Such a motivic period $p^{\mathfrak{m},M}_{v,\sigma}$ can be seen as a function on $P_{B,\omega}$:
\begin{equation}\label{eq:mper} p^{\mathfrak{m},M}_{v,\sigma}: P_{B,\omega} (\mathbb{Q}) \rightarrow \mathbb{Q}\text{  ,  }  p \mapsto <v, p(\sigma)>  \text{   }\in \mathcal{O}(P_{B,\omega}).
\end{equation}
Its period is obtained by evaluation $p^{\mathfrak{m},M}_{v,\sigma}$ on the complex point $comp_{B, dR}$:
\begin{equation}\label{eq:perm} per (p^{\mathfrak{m},M}_{v,\sigma})\mathrel{\mathop:}= p_{v,\sigma}= < comp_{B,dR} (v\otimes 1), \sigma> \in\mathbb{C}.
\end{equation}

From now, $M=\mathcal{O}(\pi^{\mathfrak{m}}_{1}(\mathbb{P}^{1}-\lbrace 0,\mu_{N},\infty\rbrace),\overrightarrow{1_{0}},\overrightarrow{-1_{1}} )$. A \textit{\textbf{motivic iterated integral}} is $I^{\mathfrak{m}}(w)=\left[M,w,dch^{B}\right]$ where $dch^{B}$ is the image of the straight path (droit chemin) from $0$ to $1$ in $\omega_{B}(M)^{\vee}$. Its period is :
\begin{equation}\label{eq:peri} per(I^{\mathfrak{m}}(w))= I(w) =\int_{0}^{1}w= <comp_{B,dR}(w\otimes 1),dch^{B}>\in\mathbb{C}.
\end{equation}
To a word $w$ in $\left\{0, \eta_{\in\mu_{N}}\right\}$, we associate its image $I^{\mathfrak{m}}(0; w ; 1)$, with the correspondance:
$$(a_{1}, \cdots a_{n})\in \left\{0, \eta_{\in\mu_{N}}\right\}^{n} \leftrightarrow\omega_{a_{1}} \cdots  \omega_{a_{n}} \text{ where } \omega_{\alpha}(t)= \frac{dt}{t-\alpha}.$$
\begin{defi}
The \textbf{motivic multiple zeta values} relative to $\mu_{N}$ are defined by:
$$\zeta_{k}^{\mathfrak{m}} \left({ x_{1}, \cdots , x_{p} \atop \epsilon_{1} , \cdots ,\epsilon_{p} }\right) \mathrel{\mathop:}= (-1)^{p} I^{\mathfrak{m}} \left(0;0^{k} (\epsilon_{1}\cdots \epsilon_{p})^{-1}, 0^{x_{1}-1} ,\cdots, (\epsilon_{i}\cdots \epsilon_{p})^{-1}, 0^{x_{i}-1} ,\cdots, \epsilon_{p}^{-1}, 0^{x_{p}-1} ;1 \right) , $$
for $\epsilon_{i}\in\mu_{N}, k\geq 0, x_{i}>0$ and $(x_{p},\epsilon_{p})\neq (1,1)$.
\end{defi}
 We denote by $\mathcal{H}^{N}$ the $\mathbb{Q}$-vector space of motivic multiple zeta values relative to $\mu_{N}$, which is a quotient of $\mathcal{O}(\Pi_{0,1})$. Moreover, $\mathcal{H}^{N}$ is a comodule of $\mathcal{A}^{N}$:
 \begin{equation}\label{eq:hn}\mathcal{H}^{N} = \mathcal{A}^{N} \otimes \left\{
\begin{array}{ll} 
\mathbb{Q}\left[  (2i\pi)^{2,\mathfrak{m}} \right] & \text{ for } N=1,2  \\
\mathbb{Q}\left[ (2i\pi)^{\mathfrak{m}} \right] & \text{ for } N>2 .
\end{array}
\right.
\end{equation}
Let $I^{\mathfrak{a}}, \zeta^{\mathfrak{a}}$ respectively $I^{\mathfrak{l}},\zeta^{\mathfrak{a}}$ denote the image in $\mathcal{A}$, resp. in the coalgebra of indecomposables $\mathcal{L}$. \\

There is a surjective homomorphism called the \textbf{\textit{period map}}, conjectured to be isomorphism:
 \begin{equation}\label{eq:period}per:\mathcal{H} \rightarrow \mathcal{Z} \text{ ,  } \zeta^{\mathfrak{m}} \left(x_{1}, \cdots , x_{p} \atop \epsilon_{1} , \cdots ,\epsilon_{p} \right)\mapsto  \zeta\left(x_{1}, \cdots , x_{p} \atop \epsilon_{1} , \cdots ,\epsilon_{p} \right).
 \end{equation}
Each identity between motivic multiple zeta values at roots of unity is then true for multiple zeta values at roots of unity and in particular each result about a basis with motivic MZVs implies the corresponding result about a generating family of MZVs by application of the period map.\\
Conversely, we can almost lift an identity between MZVs at roots of unity to an identity between motivic ones up to one rational coefficient at each step thanks to the coaction (via an analogue of $\cite{Br1}$, Theorem $3.3$ for roots of unity).\\

The comodule $\mathcal{H}^{N}\subseteq \mathcal{O}(\Pi_{0,1})$ embeds, non canonically, into $\mathcal{H}^{\mathcal{MT}_{N}}$.\\
We will work in the subcategories $\mathcal{MT}'_{N}$, which are equivalent to $\mathcal{MT}_{N}$ since $\mathcal{H}=\mathcal{H}^{\mathcal{MT}}$ for $N=1,2,3,4,6,8$ (by F. Brown $\cite{Br1}$ for $N=1$, by Deligne $\cite{De}$ for the other cases, or Corollary $1.2$). For each $N, N'$ with $N' | N$, the motivic Galois descent has a parallel for the motivic fundamental group:
\begin{figure}[H]
$$\xymatrixcolsep{5pc}\xymatrix{
\mathcal{H}^{N} \ar@{^{(}->}[r]
^{\sim}_{n.c} & \mathcal{H}^{\mathcal{MT}_{N}}  \\
\mathcal{H}^{N'}\ar[u]_{\mathcal{G}^{N/N'}} \ar@{^{(}->}[r]
_{n.c}^{\sim} &\mathcal{H}^{\mathcal{MT}_{N'}} \ar[u]^{\mathcal{G}^{\mathcal{MT}}_{N/N'}}    \\
\mathbb{Q}[\pi^{\mathfrak{m}}] \ar[u]_{\mathcal{U}^{N'}} \ar@{^{(}->}[r]^{\sim} & \mathbb{Q}[\pi^{\mathfrak{m}}] \ar[u]^{\mathcal{U}^{\mathcal{MT}_{N'}}} \\
\mathbb{Q}  \ar[u]_{\mathbb{G}_{m}} \ar@/^2pc/[uuu]^{\mathcal{G}^{N}} & \mathbb{Q}  \ar[u]^{\mathbb{G}_{m}} \ar@/_2pc/[uuu]_{\mathcal{G}^{\mathcal{MT}_{N}}}
}$$
\caption{Galois descents (level $0$).\protect\footnotemark }\label{fig:paralleldescent}
\end{figure}
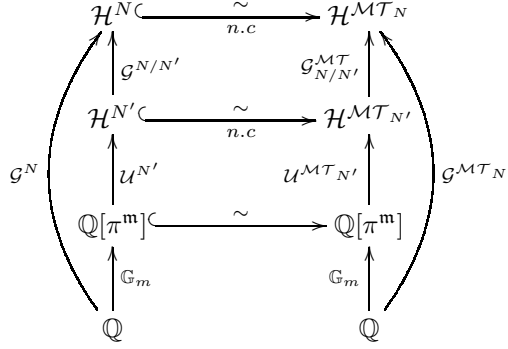
\footnotetext{\texttt{Nota Bene:}  For $N'=1$  or $2$, $\pi^{\mathfrak{m}}$ has to be replaced by $(\pi^{\mathfrak{m}})^{2}$ or $\zeta^{\mathfrak{m}}(2)$.} 

\paragraph{Coaction.}
The group $\mathcal{G}^{\mathcal{MT}_{N}}$ acts on the de Rham realisation $\Pi_{0,1}$ of the motivic fundamental groupoid (cf. $\cite{DG}, \S 5.12$). Since  $\mathcal{A}^{\mathcal{MT}}= \mathcal{O}(\mathcal{U}^{\mathcal{MT}})$, the action of $\mathcal{U}^{\mathcal{MT}}$ on $\Pi_{0,1}$ gives rise by duality to a coaction: $\Delta^{\mathcal{MT}}$. It factorizes through $\mathcal{A}$ since $\mathcal{U}$ is the quotient of $\mathcal{U}^{\mathcal{MT}}$ by the kernel of its action on $\Pi_{0,1}$ ($\cite{DG}$).\\
Then the combinatorial coaction (on words on $0, \eta\in\mu_{N}$) induces a coaction $\Delta$ on $\mathcal{H}$, which is explicit (by Goncharov $\cite{Go}$ and extended by Brown); the formula being given in $ § 3.2$.

$$ \label{Coaction} \xymatrix{
\mathcal{O}(\Pi_{0,1}) \ar[r]^{\Delta^{\mathcal{MT}}} & \mathcal{A}^{\mathcal{MT}} \otimes_{\mathbb{Q}} \mathcal{O} (\Pi_{0,1})\\
\mathcal{O}(\Pi_{0,1}) \ar[r]^{\Delta^{c}}\ar[u]^{\sim} \ar[d] & \mathcal{A} \otimes_{\mathbb{Q}} \mathcal{O} (\Pi_{0,1}) \ar[u] \ar[d]\\
\mathcal{H} \ar[r]^{\Delta}  & \mathcal{A} \otimes \mathcal{H}. \\
}$$

\paragraph{Motivic iterated integrals.}
Extend the previous definition of $I^{\mathfrak{m}}(a_{0}; a_{1}, \cdots a_{n}; a_{n+1}) \in \mathcal{H}_{n}$, with $a_{i}\in\mu_{N} \cup \left\{0\right\}$ defined as above if $a_{0}=0$ and $a_{n+1}=1$, and extend -in an unique way- by the following properties:
\\
\begin{itemize}
	\item[(i)] $I^{\mathfrak{m}}(a_{0}; a_{1})=1$.
	\item[(ii)] $I^{\mathfrak{m}}(a_{0}; a_{1}, \cdots a_{n}; a_{n+1})=0$ if $a_{0}=a_{n+1}$.
	\item[(iii)] Shuffle product: 
	\begin{multline}
\zeta_{k}^{\mathfrak{m}} \left( {x_{1}, \cdots , x_{p} \atop \epsilon_{1}, \cdots ,\epsilon_{p} }\right)= \\
(-1)^{k}\sum_{i_{1}+ \cdots + i_{p}=k} \binom {x_{1}+i_{1}-1} {i_{1}} \cdots \binom {x_{p}+i_{p}-1} {i_{p}} \zeta^{\mathfrak{m}} \left( {x_{1}+i_{1}, \cdots , x_{p}+i_{p} \atop \epsilon_{1}, \cdots ,\epsilon_{p} }\right).
 \end{multline}
	\item[(iv)] Path composition: 
	$$ \forall x\in \mu_{N} \cup \left\{0\right\} ,  I^{\mathfrak{m}}(a_{0}; a_{1}, \cdots, a_{n}; a_{n+1})=\sum_{i=1}^{n} I^{\mathfrak{m}}(a_{0}; a_{1}, \cdots, a_{i}; x) I^{\mathfrak{m}}(x; a_{i+1}, \cdots, a_{n}; a_{n+1}) .$$
	\item[(v)] Path reversal: $I^{\mathfrak{m}}(a_{0}; a_{1}, \cdots, a_{n}; a_{n+1})= (-1)^n I^{\mathfrak{m}}(a_{n+1}; a_{n}, \cdots, a_{1}; a_{0}).$
	\item[(vi)] Homothety: $\forall \alpha \in \mu_{N}, I^{\mathfrak{m}}(0; \alpha a_{1}, \cdots, \alpha a_{n}; \alpha a_{n+1})  = I^{\mathfrak{m}}(0; a_{1}, \cdots, a_{n}; a_{n+1})$.
\end{itemize}
\vspace{0,5cm}
\textsc{Remark}: These relations, for the multiple zeta values relative to $\mu_{N}$, and for the iterated integrals  $I(a_{0}; a_{1}, \cdots a_{n}; a_{n+1})\mathrel{\mathop:}= \int_{\gamma} \omega_{a_{1}} \cdots  \omega_{a_{n+1}}$, where $\omega_{\alpha}(t)= \frac{dt}{t-\alpha}$ and $\gamma$ the straight path from $a_{0}$ to $a_{n+1}$, are obviously all easily checked, following from the properties of iterated integrals.\\
\\
\textsc{Notation}: An overline at the end means that the corresponding $\epsilon_{i}$ are $1$, except the last one which is $\exp(\frac{2i\pi}{N})$. For instance, for $N=3$:
$$\zeta(3,6, \bar{2})= \zeta\left(3,6,2 \atop  1,1, \exp(\frac{2i\pi}{3})  \right).$$ 
In the case $N=2$, $\epsilon_{i}\in \left\{\pm 1\right\}$, we simplify the notation:
 \begin{equation}\label{eq:notation2} \zeta\left(z_{1},  \ldots, z_{p} \right) \text{ where } z_{i}\in \mathbb{Z}^{\ast} \text{ corresponds to }   \zeta\left(x_{1}, \cdots , x_{p} \atop  \epsilon_{1} , \cdots ,\epsilon_{p} \right)\text{ with  } \left( \mid z_{i} \mid \atop sign(z_{i} ) \right)= \left( x_{i}  \atop \epsilon_{i} \right) .
 \end{equation}

\section{\textsc{Overview}}

\subsection{\textsc{Depth filtration}}

The inclusion of $\mathbb{P}^{1}\diagdown \lbrace 0, \mu_{N},\infty\rbrace \subset \mathbb{P}^{1}\diagdown \lbrace 0,\infty\rbrace$ implies the surjection for the de Rham realisations of fundamental groupoid:
$$  _{0}\Pi_{1} \rightarrow  \pi_{1}^{dR}(\mathbb{G}_{m}, \overrightarrow{01}).$$
Looking at the dual, it corresponds to the inclusion of:
$$ \mathcal{O}  \left( \pi_{1}^{dR}(\mathbb{G}_{m}, \overrightarrow{01} ) \right)  \cong \mathbb{Q} \left\langle e^{0} \right\rangle \hookrightarrow \mathcal{O} \left(  _{0}\Pi_{1} \right) \cong \mathbb{Q} \left\langle e^{0}, (e^{\eta})_{\eta} \right\rangle  .$$
This leads to the definition of  an increasing \textit{depth filtration} $\mathcal{F}^{\mathfrak{D}}$ on $\mathcal{O}(_{0}\Pi_{1})$ such as:
 \begin{equation}\label{eq:filtprofw} \boldsymbol{\mathcal{F}_{p}^{\mathfrak{D}}\mathcal{O}(_{0}\Pi_{1})} \mathrel{\mathop:}= \left\langle  \text{ words } w \text{ in }e^{0},e^{\eta}, \eta\in\mu_{N} \text{ such as } \sum_{\eta\in\mu_{N}} deg _{e{\eta}}w \leq p \right\rangle _{\mathbb{Q}}.
 \end{equation}
This filtration is preserved by the coaction and thus descends to $\mathcal{H}$ (cf. $\cite{Br3}$), on which:
 \begin{equation}\label{eq:filtprofh}  \mathcal{F}_{i}^{\mathfrak{D}}\mathcal{H}\mathrel{\mathop:}= \left\langle  \zeta^{\mathfrak{m}}\left( n_{1}, \cdots, n_{r} \atop \epsilon_{1}, \cdots, \epsilon_{r} \right) , r\leq p \right\rangle _{\mathbb{Q}}.
 \end{equation}
In the same way, we define $ \mathcal{F}_{i}^{\mathfrak{D}}\mathcal{A}$ and  $\mathcal{F}_{i}^{\mathfrak{D}}\mathcal{L}$. Beware, the corresponding grading on $\mathcal{O}(_{0}\Pi_{1})$ is not motivic and the depth is not a graduation on $\mathcal{H}$\footnote{ For instance: $\zeta^{\mathfrak{m}}(3)=\zeta^{\mathfrak{m}}(1,2)$.}. The graded spaces $gr^{\mathfrak{D}}_{p}$ are defined as the quotient $\mathcal{F}_{p}^{\mathfrak{D}}/\mathcal{F}_{p-1}^{\mathfrak{D}}$. This $p$ is sometimes called the \textit{motivic depth}, as it may not coincide with -being smaller or equal- the usual depth defined in the introduction.\\

\paragraph{Depth $1$.}
In depth $1$, it is known for $\mathcal{A}$ (cf. $\cite{DG}$ Theorem $6.8$):
\begin{lemm}[Deligne, Goncharov]
The elements $\zeta^{\mathfrak{a}} \left( r; \eta \right)$ are subject only to the following relations in $\mathcal{A}$:
\begin{description}
\item[Distribution]
$$\forall d|N \text{  ,  }  \forall \eta\in\mu_{\frac{N}{d}} \text{  ,  } (\eta,r)\neq(1,1)\text{    ,    } \zeta^{\mathfrak{a}} \left({r \atop \eta}\right)= d^{r-1} \sum_{\epsilon^{d}=\eta} \zeta^{\mathfrak{a}} \left({r \atop \epsilon}\right).$$ 
\item[Conjugation]
$$\zeta^{\mathfrak{a}} \left({r \atop \eta}\right)= (-1)^{r-1} \zeta^{\mathfrak{a}} \left({r \atop \eta^{-1}}\right).$$
\end{description}
\end{lemm}
\textsc{Remark}: More generally, distribution relations for MZV relative to $\mu_{N}$ are:
$$\forall d| N, \quad \forall \epsilon_{i}\in\mu_{\frac{N}{d}} \text{  ,  } \quad \zeta\left( { x_{1}, \cdots , x_{p} \atop  \epsilon_{1} , \cdots ,\epsilon_{p} } \right) = d^{\sum x_{i} - p} \sum_{\eta_{1}^{d}=\epsilon_{1}} \cdots  \sum_{\eta_{p}^{d}=\epsilon_{p}} \zeta \left( {x_{1}, \cdots , x_{p} \atop  \eta_{1} , \cdots ,\eta_{p} } \right) .$$
They are deduced from the following identity:  
$$\text{ For } d|N , \epsilon\in\mu_{\frac{N}{d}} \text {   ,  } \sum_{\eta^{d}=\epsilon} \eta^{n}=  \left\{
\begin{array}{ll}
  d \epsilon ^{\frac{n}{d}}& \text{ if } d|n \\
  0 & \text{ else }.\\
\end{array}
\right. $$
Those relations are obviously the analogues of those satisfied by cyclotomic units modulo torsion.

\paragraph{For $N=2,3,4,\textquoteleft 6 \textquoteright,8$.}
Let start with depth $1$ results, deduced from the Lemma above, fundamental to initiate the recursion in the proof of Lemma $4.2$.
\begin{lemm} The basis for $gr^{\mathfrak{D}}_{1} \mathcal{A}$ is:
$$\left\{ \zeta^{\mathfrak{a}}\left(r;  \xi \right) \text{ such as } \left\{
\begin{array}{ll}  
r>1 \text{ odd  }  & \text{ if }N=1 \\
r \text{ odd  } & \text{ if }N=2 \\
r>0 & \text{ if } N=3,4 \\
r>1 & \text{  if } N=6 \\
\end{array} \right.   \right\rbrace  $$
For $N=8$, the basis for $gr^{\mathfrak{D}}_{1} \mathcal{A}_{r}$ is two dimensional, for all $r>0$:
$$\left\{ \zeta^{\mathfrak{a}}\left(r;  \xi \right), \zeta^{\mathfrak{a}}\left(r;  -\xi \right)\right\rbrace.$$
\end{lemm} 

Let explicit those relations in depth $1$ for $N=2,3,4,\textquoteleft 6 \textquoteright,8$, since we would use some $p$-adic properties of the basis elements in our proof:

\begin{description}
\item[\textsc{For $N=2$:}]
The distribution relation in depth 1 is:
$$\zeta^{\mathfrak{a}}\left( {2 r + 1 \atop  1}\right) = (2^{-2r}-1)\zeta^{\mathfrak{a}}\left( {2r+1 \atop -1}\right) .$$
\item[\textsc{For $N=3$:}]
$$  \zeta^{\mathfrak{l}} \left( {2r+1 \atop  1} \right)\left(1-3^{2r}\right)= 2\cdot 3^{2r}\zeta^{\mathfrak{l}}\left({2r+1 \atop \xi}\right) \quad \quad \zeta^{\mathfrak{l}}\left({2r \atop  1}\right)=0 \quad \quad  \zeta^{\mathfrak{l}}\left({r \atop  \xi}\right) =\left(-1\right)^{r-1}  \zeta^{\mathfrak{l}}\left({r \atop  \xi^{-1}}\right). $$
\item[\textsc{For $N=4$:}]
$$\begin{array}{lllllll}
\zeta^{\mathfrak{l}}\left({ r \atop 1} \right) (1-2^{r-1}) & = & 2^{r-1}\cdot \zeta^{\mathfrak{l}}\left( {r\atop -1} \right) \text{  for } r\neq 1 & \quad & \zeta^{\mathfrak{l}}\left({1\atop   1}\right) & = & \zeta^{\mathfrak{l}}\left( {2r\atop -1} \right)=0 \\
\zeta^{\mathfrak{l}}\left({2r+1\atop -1}\right) & = & 2^{2r+1}  \zeta^{\mathfrak{l}}\left( {2r+1\atop \xi} \right) & \quad & \zeta^{\mathfrak{l}}\left( {r \atop \xi} \right) & = & \left(-1\right)^{r-1}  \zeta^{\mathfrak{l}}\left({r\atop  \xi^{-1}}\right).
\end{array}$$
\item[\textsc{For $N=6$:}]
$$\begin{array}{lllllll}
\zeta^{\mathfrak{l}}\left({r\atop  1}\right)\left(1-2^{r-1}\right) & = & 2^{r-1}\zeta^{\mathfrak{l}}\left({ r \atop -1} \right) \text{  for } r\neq 1 & \quad & \zeta^{\mathfrak{l}}\left( {1 \atop 1} \right) & = & \zeta^{\mathfrak{l}}\left({2r\atop  -1}\right)=0\\
\zeta^{\mathfrak{l}}\left(  {2r+1 \atop -1} \right)  & = & \frac{2\cdot 3^{2r}}{1-3^{2r}}  \zeta^{\mathfrak{l}}\left( {2r+1 \atop  \xi} \right)& \quad & \zeta^{\mathfrak{l}}\left( {r \atop \xi^{2}} \right)  & = & \frac{2^{r-1}}{1-(-2)^{r-1}}  \zeta^{\mathfrak{l}}\left( {r \atop  \xi} \right).\\
\zeta^{\mathfrak{l}}\left({r\atop  \xi} \right) &=&\left(-1\right)^{r-1}  \zeta^{\mathfrak{l}}\left( {r \atop \xi^{-1}} \right) &\quad & \zeta^{\mathfrak{l}}\left({r \atop  -\xi} \right)  & = & \left(-1\right)^{r-1}  \zeta^{\mathfrak{l}}\left( {r \atop -\xi^{-1}} \right).
\end{array}$$
\item[\textsc{For $N=8$:}]
$$\begin{array}{lllllll}
  \zeta^{\mathfrak{l}}\left({ r \atop  1} \right)& =& \frac{ 2^{r-1}}{\left(1-2^{r-1}\right)}\zeta^{\mathfrak{l}}\left({r\atop  -1}\right) \text{  for } r\neq 1  &  \quad &  \zeta^{\mathfrak{l}}\left( {1 \atop  1} \right) &=&\zeta^{\mathfrak{l}}\left({2r\atop  -1}\right)=0  \\
 \zeta^{\mathfrak{l}}\left({ r \atop  -i }\right) &=& 2^{r-1}  \left( \zeta^{\mathfrak{l}}\left({r\atop  \xi}\right) + \zeta^{\mathfrak{l}}\left({r\atop  -\xi}\right) \right) &  \quad & \zeta^{\mathfrak{l}}\left( {2r+1 \atop  -1} \right) &=& 2^{2r+1}  \zeta^{\mathfrak{l}}\left({2r+1\atop  i}\right)   \\
\zeta^{\mathfrak{l}}\left({ r\atop  \pm \xi} \right) &=&\left(-1\right)^{r-1}  \zeta^{\mathfrak{l}}\left( {r \atop \pm \xi^{-1} }\right)  &  \quad & \zeta^{\mathfrak{l}}\left( {r \atop  i} \right) &=&\left(-1\right)^{r-1}  \zeta^{\mathfrak{l}}\left( {r \atop  -i}\right) \\
\end{array}$$

\end{description}

\subsection{\textsc{Coaction}}
The group $\mathcal{G}^{\mathcal{MT}_{N}}$ acts on the de Rham realisation $\Pi_{0,1}$ of the motivic fundamental groupoid (cf. $\cite{DG}, \S 5.12$). Since  $\mathcal{A}^{\mathcal{MT}}= \mathcal{O}(\mathcal{U}^{\mathcal{MT}})$, the action of $\mathcal{U}^{\mathcal{MT}}$ on $\Pi_{0,1}$ gives rise by duality to a coaction: $\Delta^{\mathcal{MT}}$. It factorizes through $\mathcal{A}$ since $\mathcal{U}$ is the quotient of $\mathcal{U}^{\mathcal{MT}}$ by the kernel of its action on $\Pi_{0,1}$ ($\cite{DG}$).\\
Then the combinatorial coaction (on words on $0, \eta\in\mu_{N}$) induces a coaction $\Delta$ on $\mathcal{H}$, which is explicit (by Goncharov $\cite{Go}$ and extended by Brown); the formula being given below.

$$ \label{Coaction} \xymatrix{
\mathcal{O}(\Pi_{0,1}) \ar[r]^{\Delta^{\mathcal{MT}}} & \mathcal{A}^{\mathcal{MT}} \otimes_{\mathbb{Q}} \mathcal{O} (\Pi_{0,1})\\
\mathcal{O}(\Pi_{0,1}) \ar[r]^{\Delta^{c}}\ar[u]^{\sim} \ar[d] & \mathcal{A} \otimes_{\mathbb{Q}} \mathcal{O} (\Pi_{0,1}) \ar[u] \ar[d]\\
\mathcal{H} \ar[r]^{\Delta}  & \mathcal{A} \otimes \mathcal{H}. \\
}$$
\\
The coaction for motivic iterated integrals is given by the following formula, due to A. B. Goncharov (cf. $\cite{Go}$) for $\mathcal{A}$ and extended by F. Brown to $\mathcal{H}$ (cf. $\cite{Br2}$):
\begin{theo} \label{eq:coaction}
The coaction $\Delta: \mathcal{H} \rightarrow \mathcal{A} \otimes_{\mathbb{Q}} \mathcal{H}$ is given by the combinatorial coaction $\Delta^{c}$:
$$\Delta I^{\mathfrak{m}}(a_{0}; a_{1}, \cdots a_{n}; a_{n+1}) =$$
$$\sum_{k ;i_{0}= 0<i_{1}< \cdots < i_{k}<i_{k+1}=n+1} \left( \prod_{p=0}^{k} I^{\mathfrak{a}}(a_{i_{p}}; a_{i_{p}+1}, \cdots a_{i_{p+1}-1}; a_{i_{p+1}}) \right) \otimes I^{\mathfrak{m}}(a_{0}; a_{i_{1}}, \cdots a_{i_{k}}; a_{n+1}) .$$
\end{theo}
\textsc{Remark:} It has a nice geometric formulation, considering the $a_{i}$ as vertices on a half-circle: 
$$\Delta^{c} I^{\mathfrak{m}}=\sum_{\text{ polygons on circle  } \atop \text{ with vertices } (a_{i_{p}})}     \prod_{p}  I^{\mathfrak{a}}\left(  \text{ arc between consecutive vertices  } \atop \text{ from } a_{i_{p}} \text{ to } a_{i_{p+1}} \right)  \otimes I^{\mathfrak{m}}(\text{  vertices } ).$$
Define for $r\geq 1$, the \textit{derivation operators}:
\begin{equation}\label{eq:dr}
\boldsymbol{D_{r}}: \mathcal{H} \rightarrow \mathcal{L}_{r} \otimes_{\mathbb{Q}} \mathcal{H},
\end{equation}
 composite of $\Delta'= \Delta^{c}- 1\otimes id$ with $\pi_{r} \otimes id$, where $\pi_{r}$ is the projection $\mathcal{A} \rightarrow \mathcal{L} \rightarrow \mathcal{L}_{r}$.\\
 \\
\texttt{Nota Bene:} It is sufficient to consider those weight-graded derivative operators to keep track of all the information of the coaction.\\
\\
According to the previous theorem, the action of $D_{r}$ on $I^{\mathfrak{m}}(a_{0}; a_{1}, \cdots a_{n}; a_{n+1})$ is:
\begin{framed}
\begin{equation}
\label{eq:Der}
D_{r}I^{\mathfrak{m}}(a_{0}; a_{1}, \cdots a_{n}; a_{n+1})= 
\end{equation}
$$\sum_{p=0}^{n-1} I^{\mathfrak{l}}(a_{p}; a_{p+1}, \cdots a_{p+r}; a_{p+r+1}) \otimes I^{\mathfrak{m}}(a_{0}; a_{1}, \cdots a_{p}, a_{p+r+1} \cdots a_{n}; a_{n+1}) .$$
\end{framed}
\textsc{Remarks}
\begin{itemize}
\item[$\cdot$] Geometrically, it is equivalent to keep in the previous coaction only the polygons corresponding to an unique cut of (interior) length $r$ between two elements of the iterated integral.
\item[$\cdot$] These maps $D_{r}$ are derivations:
$$D_{r} (XY)= (1\otimes X) D_{r}(Y) + (1\otimes Y) D_{r}(X).$$
\item[$\cdot$] This formula is linked with the equation differential satisfied by the iterated integral $I(a_{0}; \cdots; a_{n+1})$ when the $a_{i} 's$ vary (cf. \cite{Go})\footnote{Since $I(a_{i-1};a_{i};a_{i+1})= \log(a_{i+1}-a_{i})-\log(a_{i-1}-a_{i})$.}:
$$dI(a_{0}; \cdots; a_{n+1})= \sum dI(a_{i-1};a_{i};a_{i+1})  I(a_{0}; \cdots \widehat{a_{i}} \cdots; a_{n+1}).$$
\end{itemize}

Translating ($\ref{eq:dr}$) for cyclotomic MZV:
\begin{lemm}
\label{drz}
$$D_{r}: \mathcal{H}_{n} \rightarrow  \mathcal{L}_{r} \otimes  \mathcal{H}_{n-r} $$
$$ D_{r} \left(\zeta^{\mathfrak{m}} \left({n_{1}, \cdots , n_{p} \atop \epsilon_{1}, \cdots ,\epsilon_{p}} \right)\right) =   \delta_{r=n_{1}+ \cdots +n_{i}} \zeta^{\mathfrak{l}} \left({n_{1}, \cdots n_{i} \atop  \epsilon_{1}, \cdots, \epsilon_{i}}\right) \otimes \zeta^{\mathfrak{m}} \left(  { n_{i+1},\cdots, n_{p} \atop \epsilon_{i+1}, \cdots, \epsilon_{p} }\right) $$
$$ \sum_{1 \leq i<j\leq p \atop \lbrace r \leq \sum_{k=i}^{j} n_{k} -1\rbrace}  \left[ \delta_{ \sum_{k=i+1}^{j} n_{k} \leq r }  \zeta^{\mathfrak{l}}_{0^{r-  \sum_{k=i+1}^{j}n_{k}}} \left({ n_{i+1}, \cdots , n_{j} \atop  \epsilon_{i+1}, \cdots, \epsilon_{j}}\right) +(-1)^{r} \delta_{ \sum_{k=i}^{j-1} n_{k} \leq r}  \zeta^{\mathfrak{l}}_{0^{r-  \sum_{k=i}^{j-1}n_{k}}} \left({ n_{j-1}, \cdots n_{i} ,\atop  \epsilon_{j-1}^{-1}, \cdots, \epsilon_{i}^{-1}}\right) \right]  $$
$$ \otimes \zeta^{\mathfrak{m}} \left( {\cdots, \sum_{k=i}^{j} n_{k}-r,\cdots \atop  \cdots , \prod_{k=i}^{j}\epsilon_{k}, \cdots}\right)  $$
\end{lemm}
\begin{proof}
Straight-forward from $\ref{eq:dr}$, passing to MZV$_{\mu_{N}}$ writing.
\end{proof}
A key point is that the Galois action and hence the coaction respects the weight grading and the depth filtration:
$$D_{r} (\mathcal{H}_{n}) \subset \mathcal{L}_{r} \otimes_{\mathbb{Q}} \mathcal{H}_{n-r}.$$
$$ D_{r} (\mathcal{F}_{p}^{\mathfrak{D}}  \mathcal{H}_{n}) \subset \mathcal{L}_{r} \otimes_{\mathbb{Q}} \mathcal{F}_{p-1}^{\mathfrak{D}} \mathcal{H}_{n-r}.$$
Passing to the depth-graded, let define:
$$gr^{\mathfrak{D}}_{p} D_{r}: gr_{p}^{\mathfrak{D}} \mathcal{H} \rightarrow \mathcal{L}_{r} \otimes gr_{p-1}^{\mathfrak{D}} \mathcal{H} \text{, as the composition } (id\otimes gr_{p-1}^{\mathfrak{D}}) \circ D_{r |gr_{p}^{\mathfrak{D}}\mathcal{H}}.$$
By Lemma $\autoref{drz}$, all the terms appearing in the left side of $gr^{\mathfrak{D}}_{p} D_{2r+1}$ have depth $1$. Hence, let consider from now the derivations $D_{r,p}$:
\begin{lemm}
\label{Drp}
$$\boldsymbol{D_{r,p}}: gr_{p}^{\mathfrak{D}} \mathcal{H} \rightarrow gr_{1}^{\mathfrak{D}} \mathcal{L}_{r} \otimes gr_{p-1}^{\mathfrak{D}} \mathcal{H} $$
$$ D_{r,p} \left(\zeta^{\mathfrak{m}} \left({x_{1}, \cdots , x_{p} \atop \epsilon_{1}, \cdots ,\epsilon_{p}} \right)\right) = \textsc{(a0)      }  \delta_{r=x_{1}}\ \zeta^{\mathfrak{l}} \left({r \atop  \epsilon_{1}}\right) \otimes \zeta^{\mathfrak{m}} \left(  { x_{2},\cdots \atop \epsilon_{2}, \cdots }\right) $$
$$\textsc{(a)      }  + \sum_{i=2}^{p-1} \delta_{x_{i}\leq r < x_{i}+ x_{i-1}-1} (-1)^{r-x_{i}} \binom {r-1}{r-x_{i}} \zeta^{\mathfrak{l}} \left({ r \atop  \epsilon_{i}}\right) \otimes \zeta^{\mathfrak{m}} \left( {\cdots, x_{i}+x_{i-1}-r,\cdots \atop  \cdots , \epsilon_{i-1}\epsilon_{i}, \cdots}\right)  $$
$$ \textsc{(b)  } -\sum_{i=1}^{p-1} \delta_{x_{i}\leq r < x_{i}+ x_{i+1}-1} (-1)^{x_{i}} \binom{r-1}{r-x_{i}} \zeta^{\mathfrak{l}}  \left( {r \atop   \epsilon_{i}^{-1}}\right) \otimes \zeta^{\mathfrak{m}} \left( {\cdots, x_{i}+x_{i+1}-r, \cdots \atop  \cdots , \epsilon_{i+1}\epsilon_{i}, \cdots}\right)  $$
$$\textsc{(c)  } +\sum_{i=2}^{p-1} \delta_{ r = x_{i}+ x_{i-1}-1 \atop \epsilon_{i-1}\epsilon_{i}\neq 1}  \left( (-1)^{x_{i}} \binom{r-1}{x_{i}-1} \zeta^{\mathfrak{l}}  \left( {r \atop \epsilon_{i-1}^{-1}}  \right) + (-1)^{x_{i-1}-1} \binom{r-1}{x_{i-1}-1} \zeta^{\mathfrak{l}}  \left( {r \atop \epsilon_{i}} \right)  \right)$$
$$\otimes \zeta^{\mathfrak{m}} \left( {\cdots, 1, \cdots \atop \cdots, \epsilon_{i-1} \epsilon_{i}, \cdots}\right) $$

$$ \textsc{(d)  } +\delta_{ x_{p} \leq r < x_{p}+ x_{p-1}-1} (-1)^{r-x_{p}} \binom{r-1}{r-x_{p}} \zeta^{\mathfrak{l}}  \left({r \atop \epsilon_{p}} \right) \otimes \zeta^{\mathfrak{m}} \left( {\cdots, x_{p-1}+x_{p}-r\atop  \cdots, \epsilon_{p-1}\epsilon_{p}}\right)  $$

$$\textsc{(d') } +\delta_{ r = x_{p}+ x_{p-1}-1 \atop \epsilon_{p-1}\epsilon_{p}\neq 1}   (-1)^{x_{p-1}}\left( \binom{r-1}{x_{p}-1} \zeta^{\mathfrak{l}}   \left( {r \atop  \epsilon_{p-1}^{-1}}  \right) - \binom{r-1}{x_{p-1}-1} \zeta^{\mathfrak{l}}  \left( {r \atop \epsilon_{p}}  \right) \right)  \otimes \zeta^{\mathfrak{m}} \left( { \cdots,  1 \atop  \cdots \epsilon_{p-1}\epsilon_{p}}\right) .$$ 
\end{lemm}
\textsc{Remarks}:
\begin{itemize}
\item[$\cdot$]
The terms of type \textsc{(d, d')}, corresponding to a \textit{deconcatenation}, play a particular role since modulo some congruences (using depth $1$ result for the left side of the coaction), we will get rid of the other terms in the cases $N=2,3,4,\textquoteleft 6 \textquoteright,8$ for the elements in the basis. In the dual point of view of Lie algebra, like in Deligne article \cite{De} or Wojtkowiak \cite{Wo}, this corresponds to show that the Ihara bracket $\lbrace,\rbrace$ on those elements modulo some vector space reduces to the usual bracket $[,]$. More generally, in other case of basis, as Hoffman one's for $N=1$, the idea is still to find an appropriate filtration on the conjectural basis, such as the coaction in the graded space acts on this family, modulo some space, as the deconcatenation, as for the $f_{i}$ alphabet. Indeed, on $H$ (Lemma $2.1$), the weight graded part of the coaction, $D_{r}$ is defined by:
\begin{equation}\label{eq:derf}
D_{r} : \quad H_{n} \quad \longrightarrow \quad L_{r} \otimes H_{n-r} \quad \quad\text{ such as :}
\end{equation}
$$ f^{j_{1}}_{i_{1}} \cdots f^{j_{k}}_{i_{k}} \mapsto \left\{
\begin{array}{ll}
  f^{j_{1}}_{i_{1}} \otimes f^{j_{2}}_{i_{2}}, \cdots, f^{j_{k}}_{i_{k}} & \text{ if } i_{1}=r .\\
  0 & \text{ else }.\\
\end{array}
\right.$$ 

\item[$\cdot$] One fundamental feature for a family of motivic multiple zeta values (which makes it 'natural' and simple) is the \textit{stability} under the coaction. The basis considered in the section $4$ are stable under those derivations.
\end{itemize}

\begin{proof}
Straight-forward from $\autoref{drz}$, using the properties of motivic iterated integrals previously listed, in Section $2$. Terms of type \textsc{(a)} correspond to cuts from a $0$ (possibly the very first one) to a root of unity, $\textsc{(b)}$ terms from a root of unity to a $0$, $\textsc{(c)}$ terms between two roots of unity and $\textsc{(d,d')}$ terms are the cuts ending by the last $1$ - called \textit{deconcatenation terms}.
\end{proof}

\paragraph{Derivation space.}
By depth $1$ results in $\S 3.1$, once we have chosen a basis for $gr_{1}^{\mathfrak{D}} \mathcal{L}_{r}$, composed by some $\zeta^{\mathfrak{a}}(r_{i};\eta_{i})$,  we can well define: \footnote{Without passing to the depth-graded, we could also define $D^{\eta}_{r}$ as $D_{r}: \mathcal{H}\rightarrow gr^{\mathfrak{D}}_{1}\mathcal{L}_{r} \otimes \mathcal{H}$ followed by $\pi^{\eta}_{r}\otimes id$ where $\pi^{\eta}:gr^{\mathfrak{D}}_{1}\mathcal{L}_{r} \rightarrow \zeta^{\mathfrak{m}}$ is the projection on $\zeta^{\mathfrak{m}}$, once we have fixed a basis for $gr^{\mathfrak{D}}_{1}\mathcal{L}_{r}$; and define as above $\mathscr{D}_{r}$ as the set of the $D^{\eta}_{r,p}$, for $\zeta^{\mathfrak{m}}(r,\eta)$ in the basis of $gr_{1}^{\mathfrak{D}} \mathcal{A}_{r}$.}
\begin{itemize}
\item[$(i)$] For each $(r_{i}, \eta_{i})$:
\begin{equation}
\label{eq:derivnp}
\boldsymbol{D^{\eta_{i}}_{r_{i},p}}: gr_{p}^{\mathfrak{D}}\mathcal{H} \rightarrow  gr_{p-1}^{\mathfrak{D}} \mathcal{H},
\end{equation}
 as the composition of $D_{r_{i},p}$ followed by the projection:
$$\pi^{\eta}: gr_{1}^{\mathfrak{D}} \mathcal{L}_{r}\otimes gr_{p-1}^{\mathfrak{D}} \mathcal{H}\rightarrow  gr_{p-1}^{\mathfrak{D}} \mathcal{H},  \quad \quad\zeta^{\mathfrak{m}}(r;  \epsilon) \otimes X \mapsto c_{\eta, \epsilon, r} X , $$ 
with $c_{\eta, \epsilon, r}\in \mathbb{Q}$ the coefficient of $\zeta^{\mathfrak{m}}(r;  \eta)$ in the decomposition of $\zeta^{\mathfrak{m}}(r;  \epsilon)$ in the basis.
\item[$(ii)$] \begin{equation}\label{eq:setdrp}
 \boldsymbol{\mathscr{D}_{r,p}} \text{ as the set of }  D^{\eta_{i}}_{r_{i},p}   \text{ for } \zeta^{\mathfrak{m}}(r_{i},\eta_{i}) \text{ in the chosen basis of } gr_{1}^{\mathfrak{D}} \mathcal{A}_{r}.
 \end{equation}
\item[$(iii)$] The \textit{derivation set} $\boldsymbol{\mathscr{D}}$ as the (disjoint) union:  $\boldsymbol{\mathscr{D}}  \mathrel{\mathop:}= \sqcup_{r>0} \left\lbrace \mathscr{D}_{r} \right\rbrace $.
\end{itemize}

\textsc{Remarks}: 
\begin{itemize}
\item[$\cdot$]  In the case $N=2,3,4,\textquoteleft 6 \textquoteright$, the cardinal of $\mathscr{D}_{r,p}$ is one (or $0$ if $r$ even and $N=2$, or if $(r,N)=(1,6)$), whereas for $N=8$ the space generated by those derivations is $2$-dimensional, generated by  $D^{\xi_{8}}_{r}$ and $D^{-\xi_{8}}_{r}$ for instance.
\item[$\cdot$] Doing the same procedure for the n.c. Hopf comodule $H$ defined in Lemma $2.1$, isomorphic to $\mathcal{H}^{\mathcal{MT}_{N}}$, since the coproduct on $H$ is the deconcatenation $\ref{eq:derf}$, leads to the following derivations operators:
$$\begin{array}{llll}
D^{j}_{r} : &  H_{n} &  \rightarrow & H_{n-r} \\
& f^{j_{1}}_{i_{1}} \cdots f^{j_{k}}_{i_{k}} & \mapsto & \left\{
\begin{array}{ll}
  f^{j_{2}}_{i_{2}}, \cdots, f^{j_{k}}_{i_{k}} & \text{ if } j_{1}=j \text{ and } i_{1}=r .\\
  0 & \text{ else }.\\
\end{array}
\right.
\end{array}.$$
\end{itemize}
Now, consider the following application, depth graded version of the derivations above, fundamental for linear independence results in $\S 4 $:
\begin{equation}
\label{eq:pderivnp}
\boldsymbol{\partial _{n,p}} \mathrel{\mathop:}=\oplus_{r<n\atop D\in \mathscr{D}_{r,p}} D : gr_{p}^{\mathfrak{D}}\mathcal{H}_{n}  \rightarrow \oplus_{r<n } \left( gr_{p-1}^{\mathfrak{D}}\mathcal{H}_{n-r}\right) ^{\oplus \text{ card } \mathscr{D}_{r,p}}
\end{equation}

\paragraph{Kernel of $\boldsymbol{D_{<n}}$. }

A key point for the use of those derivations is the ability to prove some relations (and possibly lift some from MZV to motivic MZV) up to rational coefficients, and comes from the following theorem, looking at primitive elements:
\begin{theo}
Let $D_{<n}\mathrel{\mathop:}= \oplus_{r<n} D_{r}$, and fix a basis $\lbrace \zeta^{\mathfrak{a}}\left( n \atop \eta_{j} \right) \rbrace$ of $gr_{1}^{\mathfrak{D}} \mathcal{A}_{n}$. Then:
$$\ker D_{<n}\cap \mathcal{H}^{N}_{n} =  \oplus  \mathbb{Q} \pi^{n} \bigoplus_{j} \mathbb{Q}  \zeta^{\mathfrak{m}}\left( n \atop \eta_{j} \right). $$
\end{theo}
\begin{proof} It comes from the injective morphism of graded Hopf comodule $\phi$ (Lemma $2.1$), isomorphism for $N=1,2,3,4,\textquoteleft 6 \textquoteright,8$:
$$\phi: \mathcal{H}^{N}  \xrightarrow[\sim]{n.c} H^{N} \mathrel{\mathop:}= \mathbb{Q}\left\langle \left( f^{j}_{i}\right)  \right\rangle  \otimes_{\mathbb{Q}} \mathbb{Q} \left[  g_{1}^{s}\right] .$$
Indeed, for $H^{N}$, the analogue statement is obviously true, for $\Delta'=1\otimes \Delta+ \Delta\otimes 1$:
$$\ker \Delta' \cap H_{n} = \oplus_{j} f^{j}_{n} \oplus g_{1}^{n} .$$
\end{proof}
\begin{coro}\label{kerdn}
Let $D_{<n}\mathrel{\mathop:}= \oplus_{r<n} D_{r}$.\footnote{For $N=1$, we restrict to $r$ odd $>1$; for $N=2$ we restrict to r odd; for $N=6$ we restrict to $r>1$.} Then:
$$\ker D_{<n}\cap \mathcal{H}^{N}_{n} = \left\lbrace  \begin{array}{ll}
\mathbb{Q}\zeta^{\mathfrak{m}}\left( n \atop 1 \right)   & \text{ for } N=1,2.\\
\mathbb{Q} \pi^{n} \oplus \mathbb{Q}  \zeta^{\mathfrak{m}}\left( n \atop \xi_{N} \right)  & \text{ for } N=3,4,\textquoteleft 6 \textquoteright.\\
\mathbb{Q} \pi^{n} \oplus \mathbb{Q}  \zeta^{\mathfrak{m}}\left( n \atop \xi_{8} \right) \oplus \mathbb{Q}  \zeta^{\mathfrak{m}}\left( n \atop -\xi_{8} \right) & \text{ for } N=8.\\
\end{array}\right. .$$
\end{coro}
In particular, by this result (for $N=1,2$), proving an identity between motivic MZV (resp. motivic Euler sums), amounts to:
\begin{enumerate}
\item Prove that the coaction is identical on both sides, computing $D_{r}$ for $r>0$ smaller than the weight. If the families are not stable under the coaction, this step would requires other identities.
\item Use the analytic corresponding result for MZV (resp. Euler sums) to deduce the remaining rational coefficient; if the analytic equivalent is unknown, we can at least evaluate numerically this rational coefficient.
\end{enumerate}
Another important use of this corollary, is the decomposition of (motivic) multiple zeta values into a conjectured basis, which has been explained by F. Brown in \cite{Br1}.\footnote{He gave an 'exact numerical algorithm' for this decomposition, where, at each step, a rational coefficient has to be evaluated; hence, for other roots of unity, the generalization, albeit easily stated, is harder for numerical experiments.}\\ 
However, for greater $N$, several rational coefficients appear at each step, and we would need linear independence results before concluding.\\

\subsection{\textsc{General Galois descent}}

\paragraph{Change of field. }
For each $N, N'$ with $N' | N$, the Galois action on $\mathcal{H}_{N}$ and $\mathcal{H}_{N'}$ is determined by the coaction $\Delta$. More precisely, let consider the following descent\footnote{More generally, there are Galois descents $(\mathcal{d})=(k_{N}/k_{N'}, M/M')$ from $\mathcal{H}^{\mathcal{MT}\left( \mathcal{O}_{k_{N}} \left[ \frac{1}{M}\right] \right)  }$, to $\mathcal{H}^{\mathcal{MT}\left( \mathcal{O}_{k_{N'}} \left[ \frac{1}{M'}\right]\right)  }$, with $N'\mid N$, $M'\mid M$, with a set of derivations $\mathscr{D}^{\mathcal{d}} \subset \mathscr{D}^{N}$ associated.}, assuming $\phi_{N'}$ is an isomorphism of graded Hopf comodule: \footnote{Conjecturally as soon as $N'\neq p^{r}$, $p\geq 5$. Proven for  $N'=1,2,3,4,\textquoteleft 6 \textquoteright,8$.}
$$\xymatrixcolsep{5pc}\xymatrix{
\mathcal{H}^{N} \ar@{^{(}->}[r]
^{\phi_{N}}_{n.c} & \mathcal{H}^{\mathcal{MT}_{\Gamma_{N}}}  \\
\mathcal{H}^{N'}\ar[u]_{\mathcal{G}^{N/N'}} \ar@{^{(}->}[r]
_{n.c}^{\phi_{N'}\atop \sim} &\mathcal{H}^{\mathcal{MT}_{\Gamma_{N'}}} \ar[u]^{\mathcal{G}^{\mathcal{MT}}_{N/N'}} }$$
Let choose basis for $gr_{1}\mathcal{L}_{r}^{\mathcal{MT}_{N'}}$, and extend it into a basis of $gr_{1}\mathcal{L}_{r}^{\mathcal{MT}_{N}}$:
$$ \left\lbrace \zeta^{\mathfrak{m}}(r; \eta'_{i,r}) \right\rbrace_{i} \subset \left\lbrace \zeta^{\mathfrak{m}}(r; \eta'_{i,r}) \right\rbrace \cup \left\lbrace \zeta^{\mathfrak{m}}(r; \eta_{i}) \right\rbrace_{1\leq i \leq c_{r}}, $$
$$\text{where } \quad c_{r} =\left\{
\begin{array}{ll} 
a_{N}-a_{N'}= \frac{\varphi(N)-\varphi(N')}{2}+p(N)-p(N') & \text{ if } r=1\\
b_{N}-b_{N'}= \frac{\varphi(N)-\varphi(N')}{2}   & \text{ if } r>1\\
\end{array} \right. .$$
Then, once chosen this basis, let split the set of derivations $\mathscr{D}^{N}$ into two parts (cf. $\S 3.2$), one corresponding to $\mathcal{H}^{N'}$:
\begin{equation}
\label{eq:derivdescent}
\mathscr{D}^{N} =  \mathscr{D}^{\backslash \mathcal{d}} \uplus \mathscr{D}^{\mathcal{d}}  \quad  \text{ where }\quad  \left\lbrace \begin{array}{l}
\mathscr{D}^{\backslash \mathcal{d}} =\mathscr{D}^{N'}\mathrel{\mathop:}=  \cup_{r}  \left\lbrace  D_{r}^{\eta'_{i,r}} \right\rbrace_{1\leq i \leq c_{r}} \\
\mathscr{D}^{\mathcal{d}}\mathrel{\mathop:}= \cup_{r} \left\lbrace D^{\eta_{i,r}}_{r} \right\rbrace_{1\leq i \leq c_{r}} 
\end{array} \right.  .
\end{equation}

\texttt{Examples:}
\begin{itemize}
\item[$\cdot$] For the descent from $\mathcal{MT}_{3}$ to $\mathcal{MT}_{1}$: $\mathscr{D}^{(k_{3}/\mathbb{Q}, 3/1)}=\left\lbrace  D^{-1}_{1}, D^{\xi_{3}}_{2r}, r>0 \right\rbrace $.
\item[$\cdot$] For the descent from $\mathcal{MT}_{8}$ to $\mathcal{MT}_{4}$: $\mathscr{D}^{(k_{8}/k_{4}, 2/2)}=\left\lbrace  D^{\xi_{8}}_{r}-D^{-\xi_{8}}_{r}, r>0 \right\rbrace $.
\item[$\cdot$] For the descent from $\mathcal{MT}_{9}$ to $\mathcal{MT}_{3}$:  $\mathscr{D}^{(k_{9}/k_{3}, 3/3)}=\left\lbrace  D^{\xi_{9}}_{r}-D^{-\xi^{4}_{9}}_{r}, D^{\xi_{9}}_{r}-D^{-\xi^{7}_{9}}_{r} r>0 \right\rbrace $.\footnote{By the relations in depth $1$, since:
$$\zeta^{\mathfrak{a}} \left( r\atop \xi^{3}_{9}\right)= 3^{r-1} \left( \zeta^{\mathfrak{a}} \left( r\atop \xi^{1}_{9}\right) + \zeta^{\mathfrak{a}} \left( r\atop \xi^{4}_{9}\right)+ \zeta^{\mathfrak{a}} \left( r\atop \xi^{7}_{9}\right)  \right)  \quad \text{etc.}$$}
\end{itemize}

\begin{theo}
Let $N'\mid N$ and $\mathfrak{Z}\in gr^{\mathfrak{D}}_{p}\mathcal{H}_{n}^{N}$, depth graded MMZV relative to $\mu_{N}$.\\
Then $\mathfrak{Z}\in gr^{\mathfrak{D}}_{p}\mathcal{H}^{N'}$, i.e. $\mathfrak{Z}$ is a depth graded MMZV relative to $\mu_{N'}$ modulo smaller depth if and only if:
$$ \left( \forall r<n, \forall D_{r,p}\in\mathscr{D}_{r}^{\mathcal{d}},\quad   D_{r,p}(\mathfrak{Z})=0\right)  \quad \textrm{  and  }  \quad  \left( \forall r<n, \forall D_{r,p}\in \in\mathscr{D}^{\diagdown\mathcal{d}}, \quad D_{r,p}(\mathfrak{Z})\in gr^{\mathfrak{D}}_{p-1}\mathcal{H}^{N'}\right) .$$ 
\end{theo}
\begin{proof}
In the $(f_{i})$ side, the analogue of this theorem is pretty obvious, and the result can be transported via $\phi$, and back since $\phi_{N'}$ isomorphism by assumption.
\end{proof}
This is a very useful recursive criteria (derivation strictly decreasing weight and depth) to determine if a (motivic) multiple zeta value at $\mu_{N}$ is in fact a (motivic) multiple zeta value at $\mu_{N'}$, modulo smaller depth terms; applying it recursively, it could also take care of smaller depth terms. This criteria applies for motivic MZV$_{\mu_{N}}$, and by period morphism is deduced for MZV$_{\mu_{N}}$.\\

\paragraph{Change of Ramification.} If the descent has just a ramified part, the criteria can be stated in a non depth graded version. Indeed, there, since only weight $1$ matters, to define the derivation space $\mathcal{D}^{\mathcal{d}}$ as above ($\ref{eq:derivdescent}$), we need to choose a basis for $\mathcal{O}_{N}^{\ast}\otimes \mathbb{Q}$, which we complete with $\left\lbrace \xi^{\frac{N}{q_{i}}}_{N}\right\rbrace_{i\in I}$ into a basis for $\Gamma_{N}$. Then, with $N=\prod p_{i}^{\alpha_{i}}=\prod q_{i}$:
\begin{theo}\label{ramificationchange}
Let $\mathfrak{Z}\in \mathcal{H}_{n}^{N}\subset \mathcal{H}^{\mathcal{MT}_{\Gamma_{N}}}$, MMZV relative to $\mu_{N}$.\\
Then $\mathfrak{Z}\in \mathcal{H}^{\mathcal{MT}(\mathcal{O}_{N})}$ unramified if and only if:
$$ \left( \forall i\in I, D^{\xi^{\frac{N}{q_{i}}}}_{1}(\mathfrak{Z})=0\right)  \quad \textrm{  and  }  \quad  \left( \forall r<n, \forall D_{r}\in\mathscr{D}^{\diagdown\mathcal{d}}, \quad D_{r}(\mathfrak{Z})\in \mathcal{H}^{\mathcal{MT}(\mathcal{O}_{N})}\right) .$$ 
\end{theo}
\texttt{Examples}:
\begin{description}
\item[$\boldsymbol{N=2}$:]  As claimed in the introduction, the descent between $\mathcal{H}^{2}$ and $\mathcal{H}^{1}$ is precisely measured by $D_{1}$:\footnote{$\mathscr{D}^{(\mathbb{Q}/\mathbb{Q}, 2/1)}=\left\lbrace  D^{-1}_{1} \right\rbrace $ with the above notations.}
\begin{coro}\label{criterehonoraire}
Let $\mathfrak{Z}\in\mathcal{H}^{2}=\mathcal{H}^{\mathcal{MT}_{2}}$, a motivic Euler sum.\\
Then $\mathfrak{Z}\in\mathcal{H}^{1}=\mathcal{H}^{\mathcal{MT}_{1}}$, i.e. $\mathfrak{Z}$ is a motivic multiple zeta value if and only if:
$$D_{1}(\mathfrak{Z})=0 \quad \textrm{  and  } \quad D_{2r+1}(\mathfrak{Z})\in\mathcal{H}^{1}.$$ 
\end{coro}
\item[$\boldsymbol{N=3,4,\textquoteleft 6 \textquoteright}$:]
\begin{coro}\label{ramif346}
Let $N\in \lbrace 3,4,6\rbrace$ and $\mathfrak{Z}\in\mathcal{H}^{\mathcal{MT}(\mathcal{O}_{N} \left[ \frac{1}{N}\right] )}$, a motivic MZV$_{\mu_{N}}$.\\
Then $\mathfrak{Z}$ is unramified, $\mathfrak{Z}\in\mathcal{H}^{\mathcal{MT} (\mathcal{O}_{N})}$ if and only if:
$$D_{1}(\mathfrak{Z})=0 \textrm{  and  } \quad D_{r}(\mathfrak{Z})\in\mathcal{H}^{\mathcal{MT} (\mathcal{O}_{N})}.$$ 
\end{coro}
\item[$\boldsymbol{N=p^{r}}$:] A basis for $\mathcal{O}^{N}\otimes \mathbb{Q}$ is formed by: $\left\lbrace  \frac{1-\xi^{k}}{1-\xi}  \right\rbrace_{k\wedge p=1 \atop 0<k\leq\frac{N}{2}} $,  which corresponds to 
$$\text{ The basis for }  \mathcal{A}_{1}^{\mathcal{MT}(\mathcal{O}_{N})} \quad : \left\lbrace  \zeta^{\mathfrak{m}}\left( 1 \atop \xi^{k} \right)- \zeta^{\mathfrak{m}}\left( 1 \atop \xi \right) \right\rbrace  _{ k\wedge p=1 \atop 0<k\leq\frac{N}{2}} .$$
It can be completed in a basis of $\mathcal{A}_{1}^{N}$ with $\zeta^{\mathfrak{m}}\left( 1 \atop \xi^{1} \right)$. \footnote{With the previous theorem notations, $\mathcal{D}^{\mathcal{d}}=\lbrace D^{\xi}_{1}\rbrace$ whereas $\mathcal{D}^{\diagdown \mathcal{d}}= \lbrace D^{\xi^{k}}_{1}-D^{\xi}_{1} \rbrace_{k\wedge p=1 \atop 1<k\leq\frac{N}{2} } \cup_{r>1} \lbrace D^{\xi^{k}}_{r}\rbrace_{k\wedge p=1 \atop 0<k\leq\frac{N}{2}}$; where $D^{\xi}_{1}$ has to be understood as the projection of the left side over $\zeta^{\mathfrak{a}}\left( 1 \atop \xi \right)$ in respect to the basis above of $\mathcal{H}_{1}^{\mathcal{MT}(\mathcal{O}_{N})}$ more $\zeta^{\mathfrak{a}}\left( 1 \atop \xi \right)$. This leads to a criteria equivalent to $(\ref{ramifpr})$.} However, if we consider the basis of $\mathcal{A}_{1}^{N}$ formed by primitive roots of unity up to conjugates, the criteria for the descent could also be stated as follows: 
\begin{coro}\label{ramifpr}
Let $N=p^{r}$ and $\mathfrak{Z}\in\mathcal{H}^{\mathcal{MT}_{\Gamma_{N}}}=\mathcal{H}^{\mathcal{MT}(\mathcal{O}_{N} \left[ \frac{1}{p}\right] )}$, relative to $\mu_{N}$\footnote{For instance a MMZV relative to $\mu_{N}$. Beware, for $p>5$, there could be other periods.}.\\
Then $\mathfrak{Z}$ is unramified, $\mathfrak{Z}\in\mathcal{H}^{\mathcal{MT} (\mathcal{O}_{N})}$ if and only if:
$$\sum_{k\wedge p=1 \atop 0<k\leq\frac{N}{2}} D^{\xi^{k}_{N}}_{1}(\mathfrak{Z})=0 \quad \textrm{  and  } \quad \forall \left\lbrace  \begin{array}{l}
 r>1 \\
1<k\leq\frac{N}{2}\\
k\wedge p=1 \\
\end{array} \right.  , \quad D^{\xi^{k}_{N}}_{r}(\mathfrak{Z})\in\mathcal{H}^{\mathcal{MT} (\mathcal{O}_{N})}.$$ 
\end{coro}

\end{description}

\section{\textsc{Galois Descents} for $N=2,3,4,\textquoteleft 6 \textquoteright,8$}

\subsection{\textsc{Motivic Level filtration}}
Let fix a descent $(\mathcal{d})=(k_{N}/k_{N'}, M/M')$ from $\mathcal{H}^{\mathcal{MT}\left( \mathcal{O}_{k_{N}} \left[ \frac{1}{M}\right] \right)  }$, to $\mathcal{H}^{\mathcal{MT}\left( \mathcal{O}_{k_{N'}} \left[ \frac{1}{M'}\right]\right)  }$, with $N'\mid N$, $M'\mid M$, among those considered in this section, represented in Figures $\ref{fig:d248}, \ref{fig:d36}$.\\
Let define a motivic level increasing filtration $\mathcal{F}^{\mathcal{d}}$ associated, from the set of derivations associated to this descent, $\mathscr{D}^{\mathcal{d}} \subset \mathscr{D}^{N}$, defined in $(\ref{eq:derivdescent})$.

\begin{defi}
The filtration by the \textbf{motivic level} associated to a descent $(\mathcal{d})$ is defined recursively on $\mathcal{H}^{N}$ by:
\begin{itemize}
\item[$\cdot$] $\mathcal{F}^{\mathcal{d}} _{-1} \mathcal{H}^{N}=0$.
\item[$\cdot$] $\mathcal{F}^{\mathcal{d}} _{i} \mathcal{H}^{N}$ is the largest submodule of $\mathcal{H}^{N}$ such that $\mathcal{F}^{\mathcal{d}}_{i}\mathcal{H}^{N}/\mathcal{F}^{\mathcal{d}} _{i-1}\mathcal{H}^{N}$ is killed by $\mathscr{D}^{\mathcal{d}}$, $\quad$ i.e. is in the kernel of $\oplus_{D\in \mathscr{D}^{\mathcal{d}}} D$.
\end{itemize}
\end{defi}
It's a graded Hopf algebra's filtration:
$$\mathcal{F} _{i}\mathcal{H}. \mathcal{F}_{j}\mathcal{H} \subset \mathcal{F}_{i+j}\mathcal{H} \text{  ,  } \quad \Delta (\mathcal{F}_{n}\mathcal{H})\subset \sum_{i+j=n} \mathcal{F}_{i}\mathcal{A} \otimes \mathcal{F}_{j}\mathcal{H}.$$
The associated graded is denoted: $gr^{\mathcal{d}}  _{i}$ and the quotients, coalgebras compatible with $\Delta$: 
\begin{equation}
\label{eq:quotienth} \mathcal{H}^{\geq 0}  \mathrel{\mathop:}= \mathcal{H} \text{  ,  } \mathcal{H}^{\geq i}\mathrel{\mathop:}= \mathcal{H}/ \mathcal{F}_{i-1}\mathcal{H} \text{ with the projections :}\quad \quad \forall j\geq i \text{  , } \pi_{i,j}:  \mathcal{H}^{\geq i} \rightarrow  \mathcal{H}^{\geq j}.
\end{equation}
Note that, via the isomorphism $\phi$, the motivic filtration on $\mathcal{H}^{\mathcal{MT}_{N}}$ corresponds to:
\begin{equation}
\label{eq:isomfiltration}\mathcal{F}^{\mathcal{d}} _{i} \mathcal{H}^{\mathcal{MT}_{N}} \longleftrightarrow \left\langle  x\in H^{N} \mid Deg^{\mathcal{d}} (x) \leq i \right\rangle _{\mathbb{Q}} ,
\end{equation}
where $Deg^{\mathcal{d}}$ is the degree in $\left\lbrace  \lbrace f^{j}_{r} \rbrace_{b_{N'}<j\leq b_{N} \atop r>1} , \lbrace f^{j}_{1} \rbrace_{a_{N'}<j\leq a_{N}} \right\rbrace $, which are the images of the complementary part of $ gr_{1}\mathcal{L}^{\mathcal{MT}_{N'}}$ in the basis of $gr_{1}\mathcal{L}^{\mathcal{MT}_{N}}$.\\
In particular, $\dim \mathcal{F}^{\mathcal{d}} _{i} \mathcal{H}_{n}^{\mathcal{MT}_{N}}$ are known.\\
\\
\texttt{Example}: For the descent between $\mathcal{H}^{\mathcal{MT}_{2}}$ and $\mathcal{H}^{\mathcal{MT}_{1}}$, since $gr_{1}\mathcal{L}^{\mathcal{MT}_{2}}= \left\langle \zeta^{\mathfrak{m}}(-1), \left\lbrace  \zeta^{\mathfrak{m}}(2r+1)\right\rbrace  _{r>0}\right\rangle$:
$$\mathcal{F} _{i} \mathcal{H}^{\mathcal{MT}_{2}} \quad \xrightarrow[\sim]{\phi}\quad  \left\langle  x\in \mathbb{Q}\langle f_{1}, f_{3}, \cdots \rangle\otimes \mathbb{Q}[f_{2}]  \mid Deg_{f_{1}} (x) \leq i \right\rangle _{\mathbb{Q}} \text{ , where } Deg_{f_{1}}= \text{ degree in } f_{1}.$$
\\
\\
By definition of those filtrations:
\begin{equation}D_{r,p}^{\eta} \left( \mathcal{F}_{i}\mathcal{H}_{n} \right) \subset \left\lbrace  \begin{array}{ll} 
\mathcal{F}_{i-1}\mathcal{H}_{n-r} & \text{ if }D_{r,p}^{\eta}\in\mathscr{D}^{\mathcal{d}}_{r} \\
\mathcal{F}_{i}\mathcal{H}_{n-r} & \text{ if } D_{r,p}^{\eta}\in\mathscr{D}^{\backslash\mathcal{d}}_{r}
\end{array} \right. . 
\end{equation}
Similarly, looking at $\partial_{n,p}$ (cf. $\ref{eq:pderivnp}$):
\begin{equation} \partial_{n,p}(\mathcal{F}_{i-1}\mathcal{H}_{n}) \subset \oplus_{r<n} \left( gr_{p-1}^{\mathfrak{D}}  \mathcal{F}_{i-2}\mathcal{H}_{n-r}\right)^{\text{ card } \mathscr{D}^{\mathcal{d}}_{r}} \oplus_{r<n} \left( gr_{p-1}^{\mathfrak{D}} \mathcal{F}_{i-1}\mathcal{H}_{n-r}\right) ^{\text{ card } \mathscr{D}^{\backslash\mathcal{d}}_{r}}.
\end{equation}
This allows us to pass to quotients, and define $D^{\eta,i,\mathcal{d}}_{n,p}$ and $\partial^{i,\mathcal{d}}_{n,p}$:
\begin{equation}
\label{eq:derivinp}
\boldsymbol{D^{\eta,i,\mathcal{d}}_{n,p}}: gr_{p}^{\mathfrak{D}} \mathcal{H}_{n}^{\geq i} \rightarrow \left\lbrace  \begin{array}{ll} 
 gr_{p-1}^{\mathfrak{D}} \mathcal{H}_{n-r}^{\geq i-1} & \text{ if }D_{r,p}^{\eta}\in\mathscr{D}^{\mathcal{d}}_{r} \\
gr_{p-1}^{\mathfrak{D}} \mathcal{H}_{n-r}^{\geq i} & \text{ if } D_{r,p}^{\eta}\in\mathscr{D}^{\backslash\mathcal{d}}_{r}
\end{array} \right. 
\end{equation}
\begin{framed}
\begin{equation}
\label{eq:pderivinp}
\boldsymbol{\partial^{i,\mathcal{d}}_{n,p}}: gr_{p}^{\mathfrak{D}} \mathcal{H}_{n}^{\geq i} \rightarrow  \oplus_{r<n} \left( gr_{p-1}^{\mathfrak{D}} \mathcal{H}_{n-r}^{\geq i-1}\right)^{\text{ card } \mathscr{D}^{\mathcal{d}}_{r}}  \oplus_{r<n} \left( gr_{p-1}^{\mathfrak{D}} \mathcal{H}_{n-r}^{\geq i}\right)^{\text{ card } \mathscr{D}^{\backslash\mathcal{d}}_{r}} . 
\end{equation}
\end{framed}
The bijectivity of this map is essential to the results stated below.

\subsection{\textsc{Results}}

In the following results, the filtration considered $\mathcal{F}_{i}$ is the filtration by the motivic level associated to the (fixed) descent $\mathcal{d}$ (Definition $4.1$) while the index $i$, in $\mathcal{B}_{n, p, i}$ refers to the level notion associated to the descent $\mathcal{d}$.\footnote{Precisely defined, for each descent in $\S 4.4 $.}\\
We first obtain the following result on the depth graded quotients, for all $i\geq 0$, with:
$$\mathbb{Z}_{1[P]} \mathrel{\mathop:}= \frac{\mathbb{Z}}{1+ P\mathbb{Z}}=\left\{ \frac{a}{1+b P}, a,b\in\mathbb{Z} \right\} \text{ with }  \begin{array}{ll}
P=2 & \text{ for } N=2,4,8 \\
P=3 & \text{ for } N=3,6
\end{array} .$$

\begin{lemm}
\begin{itemize}
\item[$\cdot$] $$\mathcal{B}_{n, p, \geq i} \text{  is a linearly free family of  } gr_{p}^{\mathfrak{D}} \mathcal{H}_{n}^{\geq i} \text{ and defines a } \mathbb{Z}_{1[P]} \text{ structure :}$$
Each element $Z= \zeta^{\mathfrak{m}}\left( z_{1}, \cdots , z_{p} \atop \epsilon_{1}, \cdots, \epsilon_{p}\right)\in \mathcal{B}_{n,p} $ decomposes in a $\mathbb{Z}_{1[P]}$-linear combination of $\mathcal{B}_{n, p, \geq i}$ elements, denoted $cl_{n,p,\geq i}(Z)$ in $gr_{p}^{\mathfrak{D}} \mathcal{H}_{n}^{\geq i}$, which defines, in an unique way:
$$cl_{n,p,\geq i}: \langle\mathcal{B}_{n, p, \leq i-1}\rangle_{\mathbb{Q}} \rightarrow \langle\mathcal{B}_{n, p, \geq i}\rangle_{\mathbb{Q}}.$$
\item[$\cdot$]
The following map $\partial^{i,\mathcal{d}}_{n,p}$ is bijective:
$$\partial^{i,\mathcal{d}}_{n,p}: gr_{p}^{\mathfrak{D}} \langle \mathcal{B}_{n, \geq i} \rangle_{\mathbb{Q}} \rightarrow \oplus_{r<n} \left( gr_{p-1}^{\mathfrak{D}} \langle \mathcal{B}_{n-1, \geq i-1} \rangle_{\mathbb{Q}} \right) ^{\oplus \text{ card } \mathcal{D}^{\mathcal{d}}_{r}} \oplus_{r<n} \left( gr_{p-1}^{\mathfrak{D}} \langle \mathcal{B}_{n-2r-1, \geq i} \rangle_{\mathbb{Q}} \right) ^{\oplus \text{ card } \mathcal{D}^{\backslash\mathcal{d}}_{r}}.$$
\end{itemize}
\end{lemm}
Before giving the proof, in the next section, let present its consequences: basis for the quotient, the filtration and the graded spaces for each descent considered. More precisely:
\begin{theo}
\begin{itemize}
\item[$(i)$] $\mathcal{B}_{n,\leq p, \geq i}$ is a basis of $\mathcal{F}_{p}^{\mathfrak{D}} \mathcal{H}_{n}^{\geq i}=\mathcal{F}_{p}^{\mathfrak{D}} \mathcal{H}_{n}^{\geq i, \mathcal{MT}}$.
\item[$(ii)$] 
\begin{itemize}
\item[$\cdot$] $\mathcal{B}_{n, p, \geq i}$ is a basis of $gr_{p}^{\mathfrak{D}} \mathcal{H}_{n}^{\geq i}=gr_{p}^{\mathfrak{D}} \mathcal{H}_{n}^{\geq i, \mathcal{MT}}$ on which it defines a $\mathbb{Z}_{1[P]}$-structure:\\
Each element $Z= \zeta^{\mathfrak{m}}\left( z_{1}, \cdots , z_{p} \atop \epsilon_{1}, \cdots, \epsilon_{p}\right)$ decomposes in a $\mathbb{Z}_{1[P]}$-linear combination of $\mathcal{B}_{n, p, \geq i}$ elements, denoted $cl_{n,p,\geq i}(Z)$ in $gr_{p}^{\mathfrak{D}} \mathcal{H}_{n}^{\geq i}$, which defines in an unique way: 
$$cl_{n,p,\geq i}: \langle\mathcal{B}_{n, p, \leq i-1}\rangle_{\mathbb{Q}} \rightarrow \langle\mathcal{B}_{n, p, \geq i}\rangle_{\mathbb{Q}} \text{ such as } x+cl_{n,p,\geq i}(x)\in \mathcal{F}_{i-1}\mathcal{H}_{n}+ \mathcal{F}^{\mathfrak{D}}_{p-1}\mathcal{H}_{n}.$$
\item[$\cdot$] The following map is bijective:
$$\partial^{i, \mathcal{d}}_{n,p}: gr_{p}^{\mathfrak{D}} \mathcal{H}_{n}^{\geq i} \rightarrow  \oplus_{r<n} \left( gr_{p-1}^{\mathfrak{D}} \mathcal{H}_{n-1}^{\geq i-1}\right) ^{\oplus \text{ card } \mathcal{D}^{\mathcal{d}}_{r}} \oplus_{r<n} \left( gr_{p-1}^{\mathfrak{D}} \mathcal{H}_{n-r}^{\geq i}\right) ^{\oplus \text{ card } \mathcal{D}^{\backslash\mathcal{d}}_{r}}. $$
\item[$\cdot$] $\mathcal{B}_{n,\cdot, \geq i} $ is a basis of $\mathcal{H}^{\geq i}_{n} =\mathcal{H}^{\geq i, \mathcal{MT}}_{n}$.
\end{itemize}
	\item[$(iii)$] We have the two split exact sequences in bijection:
$$ 0\longrightarrow \mathcal{F}_{i}\mathcal{H}_{n} \longrightarrow \mathcal{H}_{n} \stackrel{\pi_{0,i+1}} {\rightarrow}\mathcal{H}_{n}^{\geq i+1} \longrightarrow 0$$
$$ 0 \rightarrow \langle \mathcal{B}_{n, \cdot, \leq i} \rangle_{\mathbb{Q}} \rightarrow \langle\mathcal{B}_{n} \rangle_{\mathbb{Q}} \rightarrow \langle \mathcal{B}_{n, \cdot, \geq i+1} \rangle_{\mathbb{Q}} \rightarrow 0 .$$
The following map, defined in a unique way:
	$$cl_{n,\leq p,\geq i}: \langle\mathcal{B}_{n, p, \leq i-1}\rangle_{\mathbb{Q}} \rightarrow \langle\mathcal{B}_{n, \leq p, \geq i}\rangle_{\mathbb{Q}} \text{ such as } x+cl_{n,\leq p,\geq i}(x)\in \mathcal{F}_{i-1}\mathcal{H}_{n}.$$
\item[$(iv)$] A basis for the filtration spaces $\mathcal{F}_{i} \mathcal{H}^{\mathcal{MT}}_{n}=\mathcal{F}_{i} \mathcal{H}_{n}$:
$$\cup_{p} \left\{ x+ cl_{n, \leq p, \geq i+1}(x), x\in \mathcal{B}_{n, p, \leq i} \right\}.$$
\item[$(v)$] A basis for the graded space $gr_{i} \mathcal{H}^{\mathcal{MT}}_{n}=gr_{i} \mathcal{H}_{n}$:
$$\cup_{p} \left\{ x+ cl_{n, \leq p, \geq i+1}(x), x\in \mathcal{B}_{n, p, i} \right\}.$$
\end{itemize}
\end{theo}
The proof is given in $\S 4.3$, and the notion of level resp. motivic level, some consequences and specifications for $N=2,3,4,\textquoteleft 6 \textquoteright,8$ individually are provided in  $\S 4.4$. Some examples in small depth are displayed in Appendice $A$.\\
\paragraph{Level $0$:} \footnote{Consequences of previous theorem for $i=0$}
\begin{itemize}
\item[$\cdot$] The level $0$ of the basis elements $\mathcal{B}^{N}$ forms a basis of $\mathcal{H}^{N} = \mathcal{H}^{\mathcal{MT}_{N}}$, for $N=2,3,4,\textquoteleft 6 \textquoteright, 8$. This gives a new proof (dual) of Deligne's result (cf. $\cite{De}$).\\
The level $0$ of this filtration is hence isomorphic to the following algebras:\footnote{The equalities of the kind $\mathcal{H}^{\mathcal{MT}_{N}}= \mathcal{H}^{N}$ are consequences of the previous theorem for $N=2,3,4,\textquoteleft 6 \textquoteright,8$, and by F. Brown for $N=1$ (cf. $\cite{Br2}$). Moreover, we have inclusions of the kind $\mathcal{H}^{\mathcal{MT}_{N'}} \subseteq \mathcal{F}_{0}^{k_{N}/k_{N'},M/M'}\mathcal{H}^{\mathcal{MT}_{N}}$ and we deduce the equality from dimensions at fixed weight.}
$$ \mathcal{F}_{0}^{k_{N}/k_{N'},M/M'}\mathcal{H}^{\mathcal{MT}_{N}}=\mathcal{F}_{0}^{k_{N}/k_{N'},M/M'}\mathcal{H}^{N}=\mathcal{H}^{\mathcal{MT}_{N',M'}}="\mathcal{H}^{N',M'}" .$$
Hence the inclusions in the following diagram are here isomorphisms: 
$$\xymatrix{
\mathcal{F}_{0}^{k_{N}/k_{N'},M/M'}\mathcal{H}^{\mathcal{MT}_{N}}  &  \mathcal{H}^{\mathcal{MT}_{N'}}  \ar@{^{(}->}[l]\\
\mathcal{F}_{0}^{k_{N}/k_{N'},M/M'}\mathcal{H}^{N} \ar@{^{(}->}[u]  &  \mathcal{H}^{N'}  \ar@{^{(}->}[l] \ar@{^{(}->}[u]}.$$
\item[$\cdot$] It gives, considering such a descent $(k_{N}/k_{N'},M/M')$, a basis for $\mathcal{F}^{0}\mathcal{H}^{N}= \mathcal{H}^{\mathcal{MT}_{N',M'}}$ in terms of the basis of $\mathcal{H}^{N}$. For instance, it leads to a new basis for motivic multiple zeta values in terms of motivic Euler sums, or motivic MZV$_{\mu_{3}}$.\\
Some others $0$-level such as  $\mathcal{F}_{0}^{k_{N}/k_{N},P/1}$, $N=3,4$ which should reflect the descent from $\mathcal{MT}(\mathcal{O}_{N}\left[ \frac{1}{P}\right] )$ to $\mathcal{MT}(\mathcal{O}_{N})$ are not known to be associated to a fundamental group, but the previous result enables to reach them. We obtain a basis for:
\begin{itemize}
\item[$\cdot$] $\mathcal{H}^{\mathcal{MT}(\mathbb{Z}\left[\frac{1}{3}\right])}$ in terms of the basis of $\mathcal{H}^{\mathcal{MT}(\mathcal{O}_{3}[\frac{1}{3}])}$.
\item[$\cdot$] $\mathcal{H}^{\mathcal{MT}(\mathcal{O}_{3})}$ in terms of the basis of $\mathcal{H}^{\mathcal{MT}(\mathcal{O}_{3}[\frac{1}{3}])}$.
\item[$\cdot$] $\mathcal{H}^{\mathcal{MT}(\mathcal{O}_{4})}$ in terms of the basis of $\mathcal{H}^{\mathcal{MT}(\mathcal{O}_{4}[\frac{1}{4}])}$.
\end{itemize}
\end{itemize}

\subsection{\textsc{Proofs}}

As proved below, Theorem $4.3$ boils down to the Lemma $4.2$. Remind the map $\partial^{i,\mathcal{d}}_{n,p}$:
$$\partial^{i,\mathcal{d}}_{n,p}: gr_{p}^{\mathfrak{D}} \mathcal{H}_{n}^{\geq i} \rightarrow  \oplus_{r<n} \left( gr_{p-1}^{\mathfrak{D}} \mathcal{H}_{n-r}^{\geq i-1}\right)^{\text{ card } \mathscr{D}^{\mathcal{d}}_{r}}  \oplus_{r<n} \left( gr_{p-1}^{\mathfrak{D}} \mathcal{H}_{n-r}^{\geq i}\right)^{\text{ card } \mathscr{D}^{\backslash\mathcal{d}}_{r}}.$$
We will look at its image on $\mathcal{B}_{n,p,\geq i} $  and prove both the injectivity of $\partial^{i,\mathcal{d}}_{n,p}$ as considered in Lemma $4.2$, and the linear independence of those elements $\mathcal{B}_{n,p,\geq i}$.

\paragraph{\texttt{Proof of Lemma $4.2$ for $N=2$}.}

The formula $(\autoref{Drp})$ for $D^{-1}_{2r+1,p}$ on $\mathcal{B}$ elements:\footnote{Using identity: $\zeta^{\mathfrak{a}}(\overline{2 r + 1 })= (2^{-2r}-1)\zeta^{\mathfrak{a}}(2r+1 )$. Projection on $\zeta^{\mathfrak{l}}(\overline{2r+1})$ for the left side.}
\begin{align} \label{Deriv2} D^{-1}_{2r+1,p} \left(\zeta^{\mathfrak{m}} (2x_{1}+1, \cdots , \overline{2x_{p}+1}) \right) = & \\
 \frac{2^{2r}}{1-2^{2r}}\delta_{r =x_{1}} & \cdot \zeta^{\mathfrak{m}} (2 x_{2}+1, \cdots, \overline{2x_{p}+1})  \nonumber \\
\frac{2^{2r}}{1-2^{2r}} \left\lbrace \begin{array}{l}
\sum_{i=1}^{p-2} \delta_{x_{i+1}\leq r < x_{i}+ x_{i+1} } \binom{2r}{2x_{i+1}}   \nonumber\\
-\sum_{i=1}^{p-1}  \delta_{x_{i}\leq r < x_{i}+ x_{i+1}} \binom{2r}{2x_{i}} 
\end{array} \right. & \cdot \zeta^{\mathfrak{m}} \left( \cdots ,2x_{i-1}+1, 2 (x_{i}+x_{i+1}-r) +1, 2 x_{i+2}+1, \cdots \right)   \nonumber\\
\textrm{\textsc{(d) }} +\delta_{x_{p} \leq r \leq x_{p}+ x_{p-1}} \binom{2r}{2x_{p}}  &\cdot\zeta^{\mathfrak{m}} \left( \cdots ,2x_{p-2}+1, \overline{2 (x_{p-1}+x_{p}-r) +1}\right)   \nonumber
\end{align}

Terms of type \textsc{(d)} play a particular role since they correspond to deconcatenation for the coaction, and will be the terms of minimal $2$-adic valuation.\\
$D^{-1}_{1,p}$ acts as a deconcatenation on this family:
\begin{equation} \label{Deriv21} D^{-1}_{1,p} \left(\zeta^{\mathfrak{m}} (2x_{1}+1, \cdots , \overline{2x_{p}+1}) \right) = \left\{
\begin{array}{ll}
  0 & \text{ if } x_{p}\neq 0 \\
  \zeta^{\mathfrak{m}} (2x_{1}+1, \cdots , \overline{2x_{p-1}+1}) & \text{ if } x_{p}=0 .\\
\end{array}
\right. \end{equation}
For $N=2$, $\partial^{i}_{n,p}$ ($\ref{eq:pderivinp}$) is simply:
\begin{equation}\label{eq:pderivinp2}
\partial^{i}_{n,p}: gr_{p}^{\mathfrak{D}} \mathcal{H}_{n}^{\geq i} \rightarrow  gr_{p-1}^{\mathfrak{D}} \mathcal{H}_{n-1}^{\geq i-1} \oplus_{1<2r+1\leq n-p+1} gr_{p-1}^{\mathfrak{D}} \mathcal{H}_{n-2r-1}^{\geq i}. 
\end{equation}
Let prove all statements of Lemma $4.2$, recursively on the weight, and then recursively on depth and on the level -from $i=0$.
\begin{proof}
By recursion hypothesis, weight being strictly smaller, we assume that:	
	$$\mathcal{B}_{n-1,p-1,\geq i-1} \oplus_{1<2r+1\leq n-p+1} \mathcal{B}_{n-2r-1,p-1,\geq i} \text{ is a  basis of  } $$
	$$gr_{p-1}^{\mathfrak{D}} \mathcal{H}_{n-1}^{\geq i-1,\mathcal{B}} \oplus_{1<2r+1\leq n-p+1} gr_{p-1}^{\mathfrak{D}} \mathcal{H}_{n-2r-1}^{\geq i,\mathcal{B}}. $$
	\begin{center}
\textsc{Claim:}	The matrix $M^{i}_{n,p}$ of $\left(\partial^{i,\mathcal{d}}_{n,p} (z) \right)_{z\in \mathcal{B}_{n, p, \geq i}}$ on those spaces is invertible.
	\end{center}
\texttt{Nota Bene:} Here $D^{-1}_{1}(z)$, resp. $D^{-1}_{2r+1,p}(z)$ are expressed in terms of $\mathcal{B}_{n-1,p-1,\geq i-1} $ resp. $\mathcal{B}_{n-2r-1,p-1,\geq i}$.\\
It will prove both the bijectivity of $\partial^{i,\mathcal{d}}_{n,p}$ as considered in the lemma and the linear independence of $\mathcal{B}_{n, p, \geq i}$. Let divide $M^{i}_{n,p}$ into four blocks, with the first column corresponding to elements of $\mathcal{B}_{n, p, \geq i}$ ending by $1$:
\begin{center}
  \begin{tabular}{ l || c | c ||}
     & $x_{p}=0$ &  $x_{p}>0$ \\ \hline
     $D_{1,p}$ & M$1$ & M$2$ \\
    $D_{>1,p}$ & M$3$ & M$4$ \\
    \hline
  \end{tabular}
\end{center}
According to ($\ref{Deriv21}$), $D^{-1}_{1,p}$ is zero on the elements not ending by 1, and acts as a deconcatenation on the others. Therefore, M$3=0$, so $M^{i}_{n,p}$ is lower triangular by blocks, and the left-upper-block M$1$ is diagonal invertible. It remains to prove the invertibility of the right-lower-block $\tilde{M}\mathrel{\mathop:}=M4$, corresponding to $D^{-1}_{>1,p}$ and elements of $\mathcal{B}_{n, p, \geq i}$ not ending by 1.\\
\\
Notice that in the formula $(\ref{Deriv2})$ of $D_{2r+1,p}$, applied to an element of $\mathcal{B}_{n, p, \geq i}$, most of terms appearing have a number of $1$ greater than $i$ but there are also terms in $\mathcal{B}_{n-2r-1,p-1,i-1}$, with exactly $(i-1)$ "$1$" for type $\textsc{a,b,c}$ only. We will make disappear the latter modulo $2$, since they are $2$-adically greater. \\
More precisely, using recursion hypothesis (in weight strictly smaller), we can replace them in $gr_{p-1} \mathcal{H}^{\geq i}_{n-2r-1}$ by a $\mathbb{Z}_{\text{odd}}$-linear combination of elements in $\mathcal{B}_{n-2r-1, p-1, \geq i}$, which does not lower the $2$-adic valuation. It is worth noticing that the type \textsc{d} elements considered are now always in $\mathcal{B}_{n-2r-1,p-1,\geq i}$, since we removed the case $x_{p}= 0$.\\

Once done, we can construct the matrix $\tilde{M}$ and examine its entries.\\
Order elements of $\mathcal{B}$ on both sides by lexicographic order of its \textit{reversed} elements:
\begin{center}
$(x_{p},x_{p-1},\cdots, x_{1})$ for the colums, $(r,y_{p-1},\cdots, y_{1})$ for the rows.
\end{center}
Remark that, with such an order, the diagonal corresponds to the deconcatenation terms: $r=x_{p}$ and $x_{i}=y_{i}$.\\
Referring to $(\ref{Deriv2})$, and by the previous remark, we see that $\tilde{M}$ has all its entries of 2-adic valuation positive or equal to zero, since the coefficients in $(\ref{Deriv2})$ are in $2^{2r}\mathbb{Z}_{\text{odd}}$ (for types \textsc{a,b,c}) or of the form $\mathbb{Z}_{\text{odd}}$ for types \textsc{d,d'}. If we look only at the terms with $2$-adic valuation zero, (which comes to consider $\tilde{M}$ modulo $2$), it only remains in $(\ref{Deriv2})$ the terms of type \textsc{(d,d')}, that is:
$$ D_{2r+1,p} (\zeta^{\mathfrak{m}}(2x_{1}+1, \cdots, \overline{2x_{p}+1}))  \equiv \delta_{ r = x_{p}+ x_{p-1}} \binom{2r}{2x_{p}}  \zeta^{\mathfrak{m}} (2x_{1}+1, \cdots ,2x_{p-2}+1, \overline{1})  $$
$$+ \delta_{x_{p} \leq r < x_{p}+ x_{p-1}} \binom{2r}{2x_{p}} \zeta^{\mathfrak{m}} (2x_{1}+1, \cdots ,2x_{p-2}+1, \overline{2 (x_{p-1}+x_{p}-r) +1})  \pmod{ 2}.$$ 
Therefore, modulo 2, with the order previously defined, it remains only an upper triangular matrix ($\delta_{x_{p}\leq r}$), with 1 on the diagonal ($\delta_{x_{p}= r}$, deconcatenation terms). Thus, $\det\tilde{M}$ has a 2-adic valuation equal to zero, and in particular can not be zero, that's why $\tilde{M}$ is invertible.\\

The $\mathbb{Z}_{odd}$ structure is easily deduced from the fact that the determinant of $\tilde{M}$ is odd, and the observation that if we consider $D_{2r+1,p} (\zeta^{\mathfrak{m}} (z_{1}, \cdots, z_{p}))$, all the coefficients are integers.
\end{proof}

\paragraph{\texttt{ Proof of Lemma } $4.2$ \texttt{ for other } $N$.}

Those cases can be handled in a rather similar way than the case $N=2$, except that the number of generators is different and that several descents are possible, hence there will be several notions of level and filtrations by the motivic level, one for each descent. Let fix a descent $\mathcal{d}$ and underline the differences in the proof:

\begin{proof}
In the same way, we prove by recursion on weight, depth and level, that the following map is bijective:
$$\partial^{i,\mathcal{d}}_{n,p}: gr_{p}^{\mathfrak{D}} \langle \mathcal{B}_{n, \geq i} \rangle_{\mathbb{Q}} \rightarrow \oplus_{r<n} \left( gr_{p-1}^{\mathfrak{D}} \langle \mathcal{B}_{n-1, \geq i-1} \rangle_{\mathbb{Q}} \right) ^{\oplus \text{ card } \mathscr{D}^{\mathcal{d}}_{r}} \oplus_{r<n} \left( gr_{p-1}^{\mathfrak{D}} \langle \mathcal{B}_{n-2r-1, \geq i} \rangle_{\mathbb{Q}} \right) ^{\oplus \text{ card } \mathscr{D}^{\backslash\mathcal{d}}_{r}}.$$
\begin{center}
I.e the matrix $M^{i}_{n,p}$ of $\left(\partial^{i}_{n,p} (z) \right)_{z\in \mathcal{B}_{n, p, \geq i}}$ on $\oplus_{r<n} \mathcal{B}_{n-r,p-1,\geq i-1}^{\text{ card } \mathscr{D}^{\mathcal{d}}_{r}} \oplus_{r<n} \mathcal{B}_{n-r,p-1,\geq i}^{\text{ card } \mathscr{D}^{\backslash\mathcal{d}}_{r}}$ \footnote{Elements in arrival space are linearly independent by recursion hypothesis.} is invertible.
\end{center}
 As before, by recursive hypothesis, we replace elements of level $\leq i$ appearing in $D^{i}_{r,p}$, $r\geq 1$ by $\mathbb{Z}_{1[P]}$-linear combinations of elements of level $\geq i$ in the quotient $gr_{p-1}^{\mathfrak{D}} \mathcal{H}_{n-r}^{\geq i}$, which does not decrease the $P$-adic valuation.\\
Now looking at the expression for $D_{r,p}$ in Lemma $3.5$, we see that on the elements considered, \footnote{i.e. of the form $\zeta^{\mathfrak{m}} \left({x_{1}, \cdots , x_{p} \atop \epsilon_{1}, \cdots ,\epsilon_{p-1}, \epsilon_{p}\xi_{N} }\right)$, with $\epsilon_{i}\in \pm 1$ for $N=8$, $\epsilon_{i}=1$ else.} the left side is:
\begin{center}
Either $\zeta^{\mathfrak(l)}\left(  r\atop 1 \right) $ for type $\textsc{a,b,c} \qquad    $      Or $\zeta^{\mathfrak(l)}\left(  r\atop \xi \right) $ for Deconcatenation terms.
\end{center}
Using results in depth $1$ of Deligne and Goncharov (cf. $\S 3.1$), the deconcatenation terms are $P$-adically smaller. \\
\texttt{For instance}, for $N=6$, $r$ odd:
$$\zeta^{\mathfrak{l}}(r;  1)=\frac{2\cdot 6^{r-1}}{(1-2^{r-1})(1-3^{r-1})} \zeta^{\mathfrak{l}}(r;  \xi) , \quad \text{ and } v_{3} (\frac{2\cdot 6^{r-1}}{(1-2^{r-1})(1-3^{r-1})}) >0 .$$
\texttt{Nota Bene:} For $N=8$, $D_{r}$ has two independent components, $D_{r}^{\xi}$ and $D_{r}^{-\xi}$. We have to distinguish them, but the statement remains similar since the terms appearing in the left side are either $\zeta^{\mathfrak(l)}\left( r\atop \pm 1 \right)$, or deconcatenation terms, $\zeta^{\mathfrak(l)}\left(  r\atop \pm \xi \right)$, $2$-adically smaller by $\S 4.1$.\\
Thanks to congruences modulo $P$, only the deconcatenation terms remain:\\
$$D_{r,p} \left(\zeta^{\mathfrak{m}} \left({x_{1}, \cdots , x_{p} \atop \epsilon_{1}, \cdots ,\epsilon_{p-1},\epsilon_{p} \xi }\right)\right) = $$
$$  \delta_{ x_{p} \leq r \leq x_{p}+ x_{p-1}-1} (-1)^{r-x_{p}} \binom{r-1}{x_{p}-1} \zeta ^{\mathfrak{l}} \left( r\atop  \epsilon_{p}\xi \right) \otimes \zeta^{\mathfrak{m}} \left({ x_{1}, \cdots, x_{p-2}, x_{p-1}+x_{p}-r\atop  \epsilon_{1},  \cdots, \epsilon_{p-2}, \epsilon_{p-1}\epsilon_{p}\xi} \right) \pmod{P}.$$
As in the previous case, the matrix being modulo $P$ triangular with $1$ on the diagonal, has a determinant congruent at $1$ modulo $P$, and then, in particular, is invertible.
\end{proof}

%For $N=6$ (and the descent towards $N=1$), the corresponding bijective map is:
%$$ \partial^{i, (k_{6}/\mathbb{Q},1/1)}_{n,p}: gr_{p}^{\mathfrak{D}} \mathcal{H}_{n}^{\geq i,\mathcal{B}} \rightarrow  \oplus_{r \text{ even} } gr_{p-1}^{\mathfrak{D}}\mathcal{B}_{n-r, p-1, \geq i-1}\rangle \oplus_{r \text{ odd } >1}  gr_{p-1}^{\mathfrak{D}}  \langle \mathcal{B}_{n-r, p-1, \geq i}\rangle.$$

\paragraph{\texttt{Example for $N=2$}:} Let us illustrate the previous proof by an example, for weight $n=9$, depth $p=3$, level $i=0$, with the previous notations.\\
Instead of $\mathcal{B}_{9, 3, \geq 0}$, we will restrict to the subfamily (corresponding to $\mathcal{A}$):
$$\mathcal{B}_{9, 3, \geq 0}^{0}\mathrel{\mathop:}= \left\{ \zeta^{\mathfrak{m}}(2a+1,2b+1,\overline{2c+1}) \text{ of weight } 9 \right\} \subset$$
$$ \mathcal{B}_{9, 3, \geq 0}\mathrel{\mathop:}= \left\{ \zeta^{\mathfrak{m}}(2a+1,2b+1,\overline{2c+1})\zeta^{\mathfrak{m}}(2)^{s}\text{ of weight } 9 \right\}$$
Note that $\zeta^{\mathfrak{m}}(2)$ being trivial under the coaction, the matrix $M_{9,3}$ is diagonal by blocks following the different values of $s$ and we can prove the invertibility of each block separately; here we restrict to the block $s=0$. The matrix $\tilde{M}$ considered represents the coefficients of:
$$\zeta^{\mathfrak{m}}(\overline{2r+1})\otimes \zeta^{\mathfrak{m}}(2x+1,\overline{2y+1})\quad \text{ in }\quad  D_{2r+1,3}(\zeta^{\mathfrak{m}}(2a+1,2b+1,\overline{2c+1})).$$
The chosen order for the columns, resp. for the rows \footnote{I.e. for $\zeta^{\mathfrak{m}}(2a+1,2b+1,2c+1)$ resp. for $(D_{2r+1,3}, \zeta^{\mathfrak{m}}(2x+1,\overline{2y+1}))$.} is the lexicographic order applied to $(c,b,a)$ resp. to $(r,y,x)$. Modulo $2$, it only remains the terms of type \textsc{d,d'}, that is:
$$  D_{2r+1,3} (\zeta^{\mathfrak{m}}(2a+1, 2b+1, \overline{2c+1}))  \equiv \delta_{c \leq r \leq  b+c} \binom{2r}{2c} \zeta^{\mathfrak{m}} (2a+1, \overline{2 (b+c-r) +1}) \text{  }  \pmod{ 2}.$$
With the previous order, $\tilde{M}_{9,3}$ is then, modulo $2$:\footnote{Notice that the first four rows are exact: no need of congruences modulo $2$ for $D_{1}$ because it acts as a deconcatenation on the base.}\\
\\
\begin{tabular}{c|c|c|c|c|c|c|c|c|c|c}
   $D_{r}, \zeta\backslash$ $\zeta$& $7,1,\overline{1}$ & $5,3,\overline{1}$ & $3,5,\overline{1}$& $1,7,\overline{1}$& $5,1,\overline{3}$&$3,3,\overline{3}$&$1,5,\overline{3}$&$3,1,\overline{5}$&$1,3,\overline{5}$ & $1,1,\overline{7}$ \\
  \hline
  $D_{1},\zeta^{\mathfrak{m}}(7,\overline{1})$ & $1$ & $0$ &$0$ &$0$ &$0$ &$0$ &$0$ &$0$ &$0$ &$0$ \\
   $D_{1},\zeta^{\mathfrak{m}}(5,\overline{3})$ & $0$ & $1$ &$0$ &$0$ &$0$ &$0$ &$0$ &$0$ &$0$ &$0$ \\
  $D_{1},\zeta^{\mathfrak{m}}(3,\overline{5})$ & $0$ & $0$ &$1$ &$0$ &$0$ &$0$ &$0$ &$0$ &$0$ &$0$ \\
  $D_{1},\zeta^{\mathfrak{m}}(1,\overline{7})$ & $0$ & $0$ &$0$ &$1$ &$0$ &$0$ &$0$ &$0$ &$0$ &$0$ \\
  $D_{3},\zeta^{\mathfrak{m}}(5,\overline{7})$ & $0$ & $0$ &$0$ &$0$ &$1$ &$0$ &$0$ &$0$ &$0$ &$0$ \\
  $D_{3},\zeta^{\mathfrak{m}}(3,\overline{3})$ & $0$ & $0$ &$0$ &$0$ &$0$ &$1$ &$0$ &$0$ &$0$ &$0$ \\
  $D_{3},\zeta^{\mathfrak{m}}(1,\overline{5})$ & $0$ & $0$ &$0$ &$0$ &$0$ &$0$ &$1$ &$0$ &$0$ &$0$ \\
  $D_{5},\zeta^{\mathfrak{m}}(3,\overline{1})$ & $0$ & $0$ &$0$ &$0$ &$0$ &$\binom{4}{2}$ &$0$ &$1$ &$0$ &$0$ \\
  $D_{5},\zeta^{\mathfrak{m}}(1,\overline{3})$ & $0$ & $0$ &$0$ &$0$ &$0$ &$0$ &$\binom{4}{2}$ &$0$ &$1$ &$0$ \\
  $D_{7},\zeta^{\mathfrak{m}}(1,\overline{1})$ & $0$ & $0$ &$0$ &$0$ &$0$ &$0$ &$\binom{6}{2}$ &$0$ &$\binom{6}{4}$ &$1$ \\
  \\
\end{tabular}.
As announced, $\tilde{M}$ modulo $2$ is triangular with $1$ on the diagonal, thus obviously invertible.

\paragraph{\texttt{ Proof of the Theorem } $4.3$.}
\begin{proof}
This Theorem comes down to the Lemma $4.2$ proving the freedom of $\mathcal{B}_{n, p, \geq i}$ in $gr_{p}^{\mathfrak{D}} \mathcal{H}_{n}^{\geq i}$ defining a $\mathbb{Z}_{odd}$-structure:
\begin{itemize}
 \item[$(i)$] By this Lemma, $\mathcal{B}_{n, p, \geq i}$ is linearly free in the depth graded, and $\partial^{i,\mathcal{d}}_{n,p}$, which decreases strictly the depth, is bijective on $\mathcal{B}_{n, p, \geq i}$. The family $\mathcal{B}_{n, \leq p, \geq i}$, all depth mixed is then linearly independent on $\mathcal{F}_{p}^{\mathfrak{D}} \mathcal{H}_{n}^{\geq i}\subset \mathcal{F}_{p}^{\mathfrak{D}} \mathcal{H}_{n}^{\geq i, \mathcal{MT}}$: easily proved by application of $\partial^{i,\mathcal{d}}_{n,p}$.\\
 By a dimension argument, since $\dim \mathcal{F}_{p}^{\mathfrak{D}} \mathcal{H}_{n}^{\geq i, \mathcal{MT}}= \text{ card } \mathcal{B}_{n, \leq p, \geq i}$, we deduce the generating property.
	\item[$(ii)$] By the lemma, this family is linearly independent, and by $(i)$ applied to depth $p-1$, 
	$$gr_{p}^{\mathfrak{D}} \mathcal{H}_{n}^{\geq i}\subset gr_{p}^{\mathfrak{D}} \mathcal{H}_{n}^{\geq i, \mathcal{MT}}.$$
	Then, by a dimension argument, since $\dim gr_{p}^{\mathfrak{D}} \mathcal{H}_{n}^{\geq i, \mathcal{MT}} = \text{ card } \mathcal{B}_{n, p, \geq i}$ we conclude on the generating property. The $\mathbb{Z}_{odd}$ structure has been proven in the previous lemma.\\
	By the bijectivity of $\partial_{n,p}^{i,\mathcal{d}}$ (still previous lemma), which decreases the depth, and using the freedom of the elements of a same depth in the depth graded, there is no linear relation between elements of  $\mathcal{B}_{n,\cdot, \geq i}$ of different depths in $\mathcal{H}_{n}^{\geq i} \subset \mathcal{H}^{\geq i \mathcal{MT}}_{n}$. The family considered is then linearly independent in $\mathcal{H}_{n}^{\geq i}$. Since $\text{card } \mathcal{B}_{n,\cdot, \geq i} =\dim \mathcal{H}^{\geq i, \mathcal{MT}}_{n}$, we conclude on the equality of the previous inclusions.
	\item[$(iii)$] The second exact sequence is obviously split since $ \mathcal{B}_{n, \cdot,\geq i+1}$ is a subset of $\mathcal{B}_{n}$. We already know that $\mathcal{B}_{n}$ is a basis of $\mathcal{H}_{n}$ and $\mathcal{B}_{n, \cdot, \geq i+1}$ is a basis of $\mathcal{H}_{n}^{\geq i+1}$. Therefore, it gives a map $\mathcal{H}_{n} \leftarrow\mathcal{H}_{n}^{\geq i+1}$ and split the first exact sequence. \\
	The construction of $cl_{n,\leq p, \geq i}(x)$, obtained from $cl_{n,p, \geq i}(x)$ applied repeatedly, is the following: 
\begin{center}
	$x\in\mathcal{B}_{n, \cdot, \leq i-1} $ is sent on $\bar{x}\in \mathcal{H}_{n}^{\geq i} \cong \langle\mathcal{B}_{n, \leq p, \geq i}\rangle_{\mathbb{Q}} $ by the projection $\pi_{0,i}$ and so $x -\bar{x} \in \mathcal{F}_{i-1}\mathcal{H}$.
\end{center}
Notice that the problem of making $cl(x)$ explicit boils down to the problem of describing the map $\pi_{0,i}$ in the bases $\mathcal{B}$. 
	\item[$(iv)$] By the previous statements, those elements are linearly independent in $\mathcal{F}_{i} \mathcal{H}^{MT}_{n}$. Moreover, their cardinal is equal to the dimension of $\mathcal{F}_{i} \mathcal{H}^{MT}_{n}$. It gives the basis announced, composed of elements $x\in \mathcal{B}_{n, \cdot, \leq i}$, each corrected by an element denoted $cl(x)$ of $ \langle\mathcal{B}_{n, \cdot, \geq i+1}\rangle_{\mathbb{Q}}$.
	\item[$(v)$] By the previous statements, those elements are linearly independent in $gr_{i} \mathcal{H}_{n}$, and by a dimension argument, we can conclude.
\end{itemize}
\end{proof}

\subsection{\textsc{Specifications for each case}}
\subsubsection{\textsc{The case } $N=2$.}
Here, since there is only one Galois descent from $\mathcal{H}^{2}$ to $\mathcal{H}^{1}$, the previous exponents for level filtrations can be omitted, as the exponent $2$ for $\mathcal{H}$ the space of motivic Euler sums. Set $\mathbb{Z}_{\text{odd}}= \left\{ \frac{a}{b} \text{ , } a\in\mathbb{Z}, b\in 2 \mathbb{Z}+1  \right\}$, rationals having a $2$-adic valuation positive or infinite. Let define particular families of motivic Euler sums, a notion of level and of motivic level. 
\begin{defi}
\begin{itemize}
	\item[$\cdot$] $\mathcal{B}^{2}\mathrel{\mathop:}=\left\{\zeta^{\mathfrak{m}}(2x_{1}+1, \cdots, 2 x_{p-1}+1,\overline{2 x_{p}+1}) \zeta(2)^{\mathfrak{m},k}, x_{i} \geq 0, k \in \mathbb{N} \right\}.$\\
Here, the level is defined as the number of $x_{i}$ equal to zero.
	\item[$\cdot$] The filtration by the motivic ($\mathbb{Q}/\mathbb{Q},2/1$)-level, %("`le nb of f1"'):
$$\mathcal{F}_{i}\mathcal{H}\mathrel{\mathop:}=\left\{ \xi \in \mathcal{H}, \textrm{ such that } D^{-1}_{1}\xi \in \mathcal{F}_{i-1} \mathcal{H} \text{ , } \forall r>0, D^{1}_{2r+1}\xi\in \mathcal{F}_{i}\mathcal{H} \right\}.$$
\begin{center}
I.e. $\mathcal{F}_{i}$ is the largest submodule such that $\mathcal{F}_{i} / \mathcal{F}_{i-1}$ is killed by $D_{1}$.
\end{center}
\end{itemize}
\end{defi}
This level filtration commutes with the increasing depth filtration.\\
\\
\textsc{Remarks}: 
\begin{itemize}
 \item[$\cdot$] For $N=2$, the recursion relation for dimensions $d_{n}=d_{n-1}+d_{n-2}$ of $\mathcal{H}_{n}^{2}$ suggests, \textit{in the Hoffman's way}, \footnote{The Hoffman basis: $\left\lbrace \zeta^{\mathfrak{m}} \left( \lbrace 2, 3\rbrace^{\times}\right)\right\rbrace_{\text{ weight } n}  $ is a basis of $\mathcal{H}^{1}_{n}$, whose dimensions verify $d_{n}=d_{n-2}+d_{n-3}$.} a basis composed of Euler sums with only $1$ and $2$. Indeed, for instance $\left\lbrace  \zeta^{\mathfrak{m}} \left( \boldsymbol{n}  \atop 1, \cdots, 1, -1 \right) , \boldsymbol{n}\in \lbrace 2, 1\rbrace^{\times}\right\rbrace $ is conjectured to be a basis. However, there is not a nice \textit{suitable} filtration which would correspond to the motivic depth, and would allow a recursive proof \footnote{A suitable filtration, whose level $0$ would be the Galois trivial elements, level $1$ would be linear combinations of $\zeta(odd)\cdot\zeta(2)^{\bullet}$, etc.}. Similarly for the family $\left\lbrace  \zeta^{\mathfrak{m}} \left( 1, \cdots 1, \atop \boldsymbol{s}, -1 \right)\zeta^{\mathfrak{m}} (2)^{\bullet} ,\boldsymbol{s}\in \left\lbrace \lbrace 1\rbrace, \lbrace -1,-1\rbrace\right\rbrace   ^{\ast }\right\rbrace $, conjectured to be a basis for $\mathcal{H}^{2}$. 
	\item[$\cdot$] The increasing or decreasing filtration defined from the number of 1 appearing in the motivic multiple zeta values is not preserved by the coproduct, since the number of 1 can either decrease or increase (by at the most 1) and is therefore not \textit{motivic}.
\end{itemize}

Let list some consequences of the results in $\S 4.2$, which generalizes in particular a result similar to P. Deligne's one (cf. $\cite{De}$):
\begin{coro} The map $\mathcal{G}^{\mathcal{MT}} \rightarrow \mathcal{G}^{\mathcal{MT}'}$ is an isomorphism.\\
The elements of $\mathcal{B}_{n}$, $\zeta^{\mathfrak{m}}(2x_{1}+1, \cdots, \overline{2 x_{p}+1}) \zeta(2)^{k}$ of weight $n$, form a basis of motivic Euler sums of weight $n$, $\mathcal{H}^{2}_{n}=\mathcal{H}^{\mathcal{MT}_{2}}_{n}$, and define a $\mathbb{Z}_{odd}$-structure on the motivic Euler sums.
\end{coro}
The period map, $per: \mathcal{H} \rightarrow \mathbb{C}$, induces the following result for the Euler sums:
\begin{center}
Each Euler sum is a $\mathbb{Z}_{odd}$-linear combination of Euler sums \\
$\zeta(2x_{1}+1, \cdots, \overline{2 x_{p}+1}) \zeta(2)^{k}, k\geq 0, x_{i} \geq 0$ of the same weight.
\end{center}
Here is the result on the $0^{th}$ level of the Galois descent from $\mathcal{H}^{1}$ to $\mathcal{H}^{2}$:
\begin{coro}
$$\mathcal{F}_{0}\mathcal{H}^{\mathcal{MT}_{2}}=\mathcal{F}_{0}\mathcal{H}^{2}=\mathcal{H}^{\mathcal{MT}_{1}}=\mathcal{H}^{1} .$$
Then, a basis of motivic multiple zeta values in weight $n$, is formed by terms of $\mathcal{B}_{n}$ with $0$-level each corrected by linear combinations of elements of $\mathcal{B}_{n}$ of level $1$: 
$$\mathcal{B}_{n}^{1}\mathrel{\mathop:}=\left\{   \zeta^{\mathfrak{m}}(2x_{1}+1, \cdots, \overline{2x_{p}+1})\zeta^{\mathfrak{m}}(2)^{s} + \sum_{y_{i} \geq 0 \atop \text{at least one } y_{i} =0} \alpha_{\textbf{x} , \textbf{y}} \zeta^{\mathfrak{m}}(2y_{1}+1, \cdots, \overline{2y_{p}+1})\zeta^{\mathfrak{m}}(2)^{s} + \right.$$
$$\left. \sum_{\text{lower depth } q<p, z_{i}\geq 0 \atop \text{ at least one } z_{i} =0} \beta_{\textbf{x}, \textbf{z}} \zeta^{\mathfrak{m}}(2 z_{1}+1, \cdots, \overline{2 z_{q}+1})\zeta^{\mathfrak{m}}(2)^{s}, x_{i}>0 , \alpha_{\textbf{x} , \textbf{y}} , \beta_{\textbf{x} , \textbf{z}} \in\mathbb{Q},\right\}_{\sum x_{i}= \sum y_{i}=\sum z_{i}= \frac{n-p}{2} -s}.$$
\end{coro}

\paragraph{Honorary.}
About the first condition in $\autoref{criterehonoraire}$ to be honorary:
\begin{lemm}
Let $\zeta^{\mathfrak{m}}(n_{1},\cdots,n_{p}) \in\mathcal{H}^{2}$, a motivic Euler sum, with $n_{i}\in\mathbb{Z}^{\ast}$, $ n_{p}\neq 1$. Then:
$$\forall i \text{ ,  } n_{i}\neq -1 \Rightarrow D_{1}(\zeta^{\mathfrak{m}}(n_{1},\cdots,n_{p}))=0 $$
\end{lemm}
\begin{proof}
Looking at all iterated integrals of length $1$ in $\mathcal{L}$, $I^{\mathfrak{l}}(a;b;c)$, $a,b,c\in \lbrace 0,\pm 1\rbrace$: the only non zero ones are those with a consecutive $\lbrace 1,-1\rbrace$ or $\lbrace -1,1\rbrace$ sequence in the iterated integral, with the condition that extremities are different, that is:
$$I(0;1;-1), I(0;-1;1), I(1;-1;0), I(-1;+1;0), I(-1;\pm 1;1), I(1;\pm 1;-1).$$
Moreover, they are all equal to $\pm \log^{\mathfrak{a}} (2)$ in the Hopf algebra $\mathcal{A}$. Consequently, if there is no $-1$ in the Euler sums writing, it implies that $D_{1}$ would be zero.
\end{proof}
\paragraph{Comparison with Hoffman's basis. } Let compare:
\begin{itemize}
\item[$(i)$] The Hoffman basis of $\mathcal{H}^{1}$ formed by motivic MZV with only $2$ and $3$ ($\cite{Br2}$)
$$\mathcal{B}^{H}\mathrel{\mathop:}= \left\{\zeta^{\mathfrak{m}} (x_{1}, \cdots, x_{k}), \text{  where } x_{i}\in\left\{2,3\right\} \right\},$$
\item[$(ii)$] 
$\mathcal{B}^{1}$, the base of $\mathcal{H}^{1}$ previously obtained (Corollary $4.6$).
\end{itemize}
Beware, the index $p$ for $\mathcal{B}^{H}$ indicates the number of "`3"' among the $x_{i}$, whereas for $\mathcal{B}^{1}$, it still indicates the depth; in both case, it can be seen as the \textit{motivic depth} (cf. $\S 3.1$):
\begin{coro}
$\mathcal{B}^{1}_{n,p}$ is a basis of $gr_{p}^{\mathfrak{D}} \langle\mathcal{B}^{H}_{n,p}\rangle_{\mathbb{Q}}$ and defines a $\mathbb{Z}_{\text{odd}}$-structure.\\
I.e. each element of the Hoffman basis of weight $n$ and with $p$ three, $p>0$, decomposes into a $\mathbb{Z}_{\text{odd}}$-linear combination of $\mathcal{B}^{1}_{n,p}$ elements plus terms of depth strictly less than $p$.
\end{coro}
\begin{proof}
Deduced from previous results, from the $\mathcal{Z}_{odd}$ structure of Euler sums basis.
\end{proof}

\subsubsection{\textsc{The cases } $N=3,4$.}

For $N=3,4$ there is a generator in each degree $\geq 1$ and two Galois descents. \\
\begin{defi}
\begin{itemize}
	\item[$\cdot$] \textbf{Family:}  $\mathcal{B}\mathrel{\mathop:}=\left\{\zeta^{\mathfrak{m}}\left({x_{1}, \cdots,x_{p}\atop  1, \cdots , 1, \xi }\right) (2i \pi)^{s,\mathfrak{m}}, x_{i} \geq 1, s \geq 0 \right\}$. 
	\item[$\cdot$] \textbf{Level:} 
$$\begin{array}{lll}
\text{ The $(k_{N}/k_{N},P/1)$-level } & \text{ is defined as } & \text{ the number of $x_{i}$ equal to 1 }\\
\text{ The $(k_{N}/\mathbb{Q},P/P)$-level } & \text{  } & \text{ the number of  $x_{i}$ even }\\
\text{ The $(k_{N}/\mathbb{Q},P/1)$-level } & \text{  } & \text{ the number of even $x_{i}$ or equal to $1$ }
\end{array}	$$
	\item[$\cdot$]  \textbf{Filtrations by the motivic level:} 
$\mathcal{F}^{\mathcal{d}} _{-1} \mathcal{H}^{N}=0$ and $\mathcal{F}^{\mathcal{d}} _{i} \mathcal{H}^{N}$ is the largest submodule of $\mathcal{H}^{N}$ such that $\mathcal{F}^{\mathcal{d}}_{i}\mathcal{H}^{N}/\mathcal{F}^{\mathcal{d}} _{i-1}\mathcal{H}^{N}$ is killed by $\mathscr{D}^{\mathcal{d}}$, where

$$\mathscr{D}^{\mathcal{d}}	= \begin{array}{ll}
\lbrace D^{\xi}_{1} \rbrace &  \text{ for } \mathcal{d}=(k_{N}/k_{N},P/1)\\
\lbrace(D^{\xi}_{2r})_{r>0} \rbrace & \text{ for } \mathcal{d}=(k_{N}/\mathbb{Q},P/P)\\
\lbrace D^{\xi}_{1},(D^{\xi}_{2r})_{r>0} \rbrace &  \text{ for } \mathcal{d}=(k_{N}/\mathbb{Q},P/1) \\
\end{array}.
$$
\end{itemize}
\end{defi}
\textsc{Remarks}: 
\begin{itemize}
\item[$\cdot$] As before, the increasing -or decreasing- filtration that we could define by the number of 1 (resp. number of even) appearing in the motivic multiple zeta values is not preserved by the coproduct, since the number of 1 can either diminish or increase (at most 1), so is not motivic. 
\item[$\cdot$] An effective way of seing those motivic level filtrations, giving a recursive criteria:
$$\mathcal{F}_{i}^{k_{N}/\mathbb{Q},P/P }\mathcal{H}= \left\{ Z \in \mathcal{H}, \textrm{ s. t. } \forall r > 0 \text{  , }  D^{\xi}_{2r}(Z) \in \mathcal{F}_{i-1}^{k_{N}/\mathbb{Q},P/P}\mathcal{H} \text{  , } \forall r \geq 0 \text{  , }  D^{\xi}_{2r+1}(Z)\in \mathcal{F}_{i}^{ k_{N}/\mathbb{Q},P/P}\mathcal{H} \right\}.$$
\end{itemize}

We deduce from results in $\S 4.2$ a result of P. Deligne ($i=0$, cf. $\cite{De}$):
\begin{coro}
The elements of $\mathcal{B}^{N}_{n,p, \geq i}$ form a basis of $gr_{p}^{\mathfrak{D}} \mathcal{H}_{n}/ \mathcal{F}_{i-1} \mathcal{H}_{n}$.\\
In particular the map $\mathcal{G}^{\mathcal{MT}_{N}} \rightarrow \mathcal{G}^{\mathcal{MT}_{N}'}$ is an isomorphism.
The elements of $\mathcal{B}_{n}^{N}$, form a basis of motivic multiple zeta value relative to $\mu_{N}$,  $\mathcal{H}_{n}^{N}$.
\end{coro}
The level $0$ of the filtrations considered for $N' \vert N\in \left\lbrace 3,4 \right\rbrace $ gives the Galois descents:
\begin{coro}
 A basis of $\mathcal{H}_{n}^{N'} $ is formed by elements of $\mathcal{B}_{n}^{N}$ of level $0$ each corrected by linear combination of elements  $\mathcal{B}_{n}^{N}$ of level $ \geq 1$. In particular, with $\xi$ primitive:
\begin{itemize}
	\item[$\cdot$] \textbf{Galois descent} from $N'=1$ to $N=3,4$: A basis of motivic multiple zeta values: 
$$\mathcal{B}^{1 ; N}  \mathrel{\mathop:}=  \left\{  \zeta^{\mathfrak{m}}\left({2x_{1}+1, \cdots, 2x_{p}+1\atop  1, \cdots, 1, \xi} \right)  \zeta^{\mathfrak{m}}(2)^{s}   + \sum_{y_{i} \geq 0 \atop \text{ at least one even or } = 1} \alpha_{\textbf{x},\textbf{y}} \zeta^{\mathfrak{m}} \left({y_{1}, \cdots, y_{p}\atop 1, \cdots, 1, \xi } \right)\zeta^{\mathfrak{m}}(2)^{s}  \right.$$
$$ \left. + \sum_{\text{ lower depth } q<p,  \atop \text{ at least one even or } = 1} \beta_{\textbf{x},\textbf{z}}  \zeta^{\mathfrak{m}}\left({z_{1}, \cdots, z_{q}\atop 1, \cdots, 1, \xi } \right)\zeta^{\mathfrak{m}}(2)^{s} \text{ ,  } x_{i}>0 , \alpha_{\textbf{x},\textbf{y}} , \beta_{\textbf{x},\textbf{z}} \in\mathbb{Q} \right\}. $$
	\item[$\cdot$]  \textbf{Galois descent} from $N'=2$ to $N=4$: A basis of motivic Euler sums:
$$\mathcal{B}^{2; 4}\mathrel{\mathop:}= \left\{  \zeta^{\mathfrak{m}} \left({2x_{1}+1, \cdots, 2x_{p}+1\atop  1, \cdots, 1, \xi_{4}} \right)\zeta^{\mathfrak{m}}(2)^{s} + \sum_{y_{i}>0 \atop \text{at least one even}} \alpha_{\textbf{x},\textbf{y}} \zeta^{\mathfrak{m}}\left({y_{1}, \cdots, y_{p}\atop 1, \cdots, 1, \xi_{4} } \right)\zeta^{\mathfrak{m}}(2)^{s} \right.$$
$$ \left.   +\sum_{\text{lower depth } q<p \atop z_{i}>0, \text{at least one even}} \beta_{\textbf{x},\textbf{z}} \zeta^{\mathfrak{m}}\left( z_{1}, \cdots, z_{q} \atop 1, \cdots, 1, \xi_{4} \right) \zeta^{\mathfrak{m}}(2)^{s}   \text{  ,  }  x_{i}\geq 0 , \alpha_{\textbf{x},\textbf{y}}, \beta_{\textbf{x},\textbf{z}}\in\mathbb{Q}  \right\} .$$
\item[$\cdot$] Similarly, replacing $\xi_{4}$ by $\xi_{3}$ in $\mathcal{B}^{2; 4}$, this gives a basis of:
$$\mathcal{F}^{k_{3}/\mathbb{Q},3/3}_{0} \mathcal{H}_{n}^{3}=\boldsymbol{\mathcal{H}_{n}^{\mathcal{MT}(\mathbb{Z}[\frac{1}{3}])}}.$$
\item[$\cdot$] A basis of $\mathcal{F}^{k_{N}/k_{N},P/1}_{0} \mathcal{H}_{n}^{N}=\boldsymbol{\mathcal{H}_{n}^{\mathcal{MT}(\mathcal{O}_{N})}}$, with $N= 3,4$:
$$\mathcal{B}^{N \text{ unram}}\mathrel{\mathop:}= \left\{  \zeta^{\mathfrak{m}} \left({x_{1}, \cdots, x_{p} \atop  1, \cdots, 1, \xi} \right)\zeta^{\mathfrak{m}}(2)^{s} + \sum_{y_{i}>0 \atop \text{at least one } 1} \alpha_{\textbf{x},\textbf{y}} \zeta^{\mathfrak{m}}\left({y_{1}, \cdots, y_{p}\atop 1, \cdots, 1, \xi} \right)\zeta^{\mathfrak{m}}(2)^{s} \right.$$
$$ \left.   +\sum_{\text{lower depth } q<p \atop z_{i}>0, \text{at least one } 1} \beta_{\textbf{x},\textbf{z}} \zeta^{\mathfrak{m}}\left( z_{1}, \cdots, z_{q} \atop 1, \cdots, 1, \xi \right) \zeta^{\mathfrak{m}}(2)^{s}   \text{  ,  }  x_{i} > 0 , \alpha_{\textbf{x},\textbf{y}}, \beta_{\textbf{x},\textbf{z}}\in\mathbb{Q}  \right\} .$$
\end{itemize}
\end{coro}
\texttt{Nota Bene:} Notice that for the last two level $0$ spaces, $\mathcal{H}_{n}^{\mathcal{MT}(\mathcal{O}_{N})}$, $N=3,4$ and $\mathcal{H}_{n}^{\mathcal{MT}(\mathbb{Z}[\frac{1}{3}])}$, we still do not have another way to reach them, since they are not known to be associated to a fundamental group.

\subsubsection{\textsc{The case } $N=8$.}

For $N=8$ there are two generators in each degree $\geq 1$ and three possible Galois descents: with $\mathcal{H}^{4}$, $\mathcal{H}^{2}$ or $\mathcal{H}^{1}$.\\

\begin{defi}
\begin{itemize}
	\item[$\cdot$] \textbf{Family:} $\mathcal{B}\mathrel{\mathop:}=\left\{\zeta^{\mathfrak{m}}\left( {x_{1}, \cdots,x_{p}\atop  \epsilon_{1}, \cdots , \epsilon_{p-1},\epsilon_{p} \xi }\right)(2i \pi)^{s,\mathfrak{m}}, x_{i} \geq 1, \epsilon_{i}\in \left\{\pm 1\right\} s \geq 0 \right\}$. 
		\item[$\cdot$] \textbf{Level}, denoted $i$: 
$$\begin{array}{lll}
 \text{ The $(k_{8}/k_{4},2/2)$-level } &  \text{ is the number of } &  \text{  $\epsilon_{j}$ equal to $-1$ } \\
  \text{ The $(k_{8}/\mathbb{Q},2/2)$-level } &  \text{  } &  \text{ $\epsilon_{j}$ equal to $-1$ $+$ even $x_{j}$ } \\
   \text{ The $(k_{8}/\mathbb{Q},2/1)$-level } &  \text{  } &  \text{ $\epsilon_{j}$ equal to $-1$, $+$ even $x_{j}$ $+$ $x_{j}$ equal to $1$. } 
\end{array}$$

\item[$\cdot$]  \textbf{Filtrations by the motivic level:} 
$\mathcal{F}^{\mathcal{d}} _{-1} \mathcal{H}^{8}=0$ and $\mathcal{F}^{\mathcal{d}} _{i} \mathcal{H}^{8}$ is the largest submodule of $\mathcal{H}^{8}$ such that $\mathcal{F}^{\mathcal{d}}_{i}\mathcal{H}^{8}/\mathcal{F}^{\mathcal{d}} _{i-1}\mathcal{H}^{8}$ is killed by $\mathscr{D}^{\mathcal{d}}$, where
$$\mathscr{D}^{\mathcal{d}}	= \begin{array}{ll}
\left\lbrace (D^{\xi}_{r}- D^{-\xi}_{r})_{r>0} \right\rbrace  &  \text{ for } \mathcal{d}=(k_{8}/k_{4},2/2)\\
\left\lbrace  (D^{\xi}_{2r+1}- D^{-\xi}_{2r+1})_{r\geq 0}, (D^{\xi}_{2r})_{r>0},( D^{-\xi}_{2r})_{r>0}  \right\rbrace  & \text{ for } \mathcal{d}=(k_{8}/\mathbb{Q},2/2)\\
\left\lbrace  (D^{\xi}_{2r+1}- D^{-\xi}_{2r+1})_{r> 0}, D^{\xi}_{1},  D^{-\xi}_{1}, (D^{\xi}_{2r})_{r>0},( D^{-\xi}_{2r})_{r>0}  \right\rbrace  &  \text{ for } \mathcal{d}=(k_{8}/\mathbb{Q},2/1) \\
\end{array}.
$$			
\end{itemize}
\end{defi}

\begin{coro}
 A basis of $\mathcal{H}_{n}^{N'} $ is formed by elements of $\mathcal{B}_{n}^{N}$ of level $0$ each corrected by linear combination of elements  $\mathcal{B}_{n}^{N}$ of level $ \geq 1$. In particular, with $\xi$ primitive:
\begin{description}
\item[$\boldsymbol{8 \rightarrow 1} $:] A basis of MMZV:
$$\mathcal{B}^{1;8}\mathrel{\mathop:}= \left\{ \zeta^{\mathfrak{m}}\left( 2x_{1}+1, \cdots, 2x_{p}+1 \atop  1, \cdots, 1, \xi \right)\zeta^{\mathfrak{m}}(2)^{s} + \sum_{y_{i} \text{at least one even or } =1 \atop { or  one } \epsilon_{i}=-1 } \alpha_{\textbf{x},\textbf{y}} \zeta^{\mathfrak{m}}\left( y_{1}, \cdots, y_{p} \atop \epsilon_{1}, \cdots, \epsilon_{p-1}, \epsilon_{p}\xi  \right)\zeta^{\mathfrak{m}}(2)^{s} \right.$$
$$\left. + \sum_{q<p \text{ lower depth, level } \geq 1} \beta_{\textbf{x},\textbf{z}} \zeta^{\mathfrak{m}}\left(z_{1}, \cdots, z_{q} \atop  \tilde{\epsilon}_{1}, \cdots, \tilde{\epsilon}_{q}\xi \right)\zeta^{\mathfrak{m}}(2)^{s}  \text{  , }x_{i}>0 , \alpha_{\textbf{x},\textbf{y}}, \beta_{\textbf{x},\textbf{z}}\in\mathbb{Q}\right\}.$$
\item[$\boldsymbol{8 \rightarrow 2 } $:] A basis of motivic Euler sums:
$$\mathcal{B}^{2;8} \mathrel{\mathop:}=  \left\{  \zeta^{\mathfrak{m}} \left(2x_{1}+1, \cdots, 2x_{p}+1 \atop  1, \cdots, 1, \xi\right)\zeta^{\mathfrak{m}}(2)^{s}  +\sum_{y_{i} \text{ at least one even} \atop \text{or one }\epsilon_{i}=-1} \alpha_{\textbf{ x},\textbf{y}} \zeta^{\mathfrak{m}}\left( y_{1}, \cdots, y_{p} \atop \epsilon_{1}, \cdots, \epsilon_{p-1}, \epsilon_{p}\xi  \right)\zeta^{\mathfrak{m}}(2)^{s}  \right.$$
$$\left. + \sum_{\text{lower depth} q<p \atop \text{with level} \geq 1 } \beta_{\textbf{x},\textbf{z}} \zeta^{\mathfrak{m}}\left(z_{1}, \cdots, z_{q} \atop \tilde{\epsilon}_{1}, \cdots, \tilde{\epsilon}_{q}\xi\right)\zeta^{\mathfrak{m}}(2)^{s} \text{  ,  }x_{i}\geq 0,\alpha_{\textbf{x},\textbf{y}}, \beta_{\textbf{x},\textbf{y}} \in\mathbb{Q} \right\}.$$
	\item[$\boldsymbol{ 8 \rightarrow 4 } $:] A basis of MMZV relative to $\mu_{4}$:
$$\mathcal{B}^{4;8}\mathrel{\mathop:}= \left\{ \zeta^{\mathfrak{m}}\left( x_{1}, \cdots, x_{p} \atop  1, \cdots, 1, \xi \right)(2i\pi)^{s}  + \sum_{\text{ at least one }\epsilon_{i}=-1} \alpha_{\textbf{x}, \textbf{y}} \zeta^{\mathfrak{m}}\left(y_{1}, \cdots, y_{p} \atop \epsilon_{1}, \cdots, \epsilon_{p-1},\epsilon_{p} \xi \right)(2 i \pi)^{s} \right.$$
$$ \left. + \sum_{\text{lower depth, level } \geq 1} \beta_{\textbf{x},\textbf{z}} \zeta^{\mathfrak{m}}\left( z_{1}, \cdots,z_{q} \atop  \tilde{\epsilon}_{1}, \cdots, \tilde{\epsilon}_{q}\xi  \right)(2i \pi)^{s} \alpha_{\textbf{x},\textbf{y}}, \beta_{\textbf{x},\textbf{z}}\in\mathbb{Q} \right\}.$$
\end{description}
\end{coro}

\subsubsection{\textsc{The case } $N=6$.}

For the unramified category $\mathcal{MT}(\mathcal{O}_{6})$, there is one generator in each degree $>1$ and one Galois descent with $\mathcal{H}^{1}$.\\
\\
First, let us point out this sufficient condition for a MMZV$_{\mu_{6}}$ to be unramified:
\begin{lemm}
$$\text{Let  } \quad  \zeta^{\mathfrak{m}} \left( n_{1},\cdots,n_{p} \atop \epsilon_{1}, \cdots, \epsilon_{p} \right)  \in\mathcal{H}^{\mathcal{MT}(\mathcal{O}_{6} \left[ \frac{1}{6}\right] )} \text{ a motivic MZV} _{\mu_{6}}, \quad \text{ such as : \footnote{In the iterated integral writing, the associated roots of unity are $\eta_{i}\mathrel{\mathop:}= (\epsilon_{i}\cdots \epsilon_{p})^{-1}$.}}$$ 
$$\begin{array}{l}
 \text{ Each } \eta_{i} \in \lbrace 1, \xi_{6} \rbrace    \\
\textsc{ or }  \text{ Each } \eta_{i} \in \lbrace 1, \xi^{-1}_{6} \rbrace
\end{array}  \quad \quad \text{ Then, } \quad \zeta^{\mathfrak{m}} \left( n_{1},\cdots,n_{p} \atop \epsilon_{1}, \cdots, \epsilon_{p} \right)  \in \mathcal{H}^{\mathcal{MT}(\mathcal{O}_{6})}$$
\end{lemm}
\begin{proof}
Immediate, by Corollary, $\autoref{ramif346}$, and with the expression of the derivations $(\autoref{drz})$ since those families are stable under the coaction. 
\end{proof}

\begin{defi}
\begin{itemize}
	\item[$\cdot$]  \textbf{Family}: $\mathcal{B}\mathrel{\mathop:}=\left\{\zeta^{\mathfrak{m}}\left( {x_{1}, \cdots,x_{p}\atop  1, \cdots , 1,\xi) } \right)(2i \pi)^{s,\mathfrak{m}}, x_{i} > 1, s \geq 0 \right\}$. 
			\item[$\cdot$] \textbf{Level:} The $(k_{6}/\mathbb{Q},1/1)$-level, denoted $i$, is defined as the number of even $x_{j}$.
				\item[$\cdot$] \textbf{Filtration by the motivic }   $(k_{6}/\mathbb{Q},1/1)$-\textbf{level}:
\begin{center}
	$\mathcal{F}^{(k_{6}/\mathbb{Q},1/1)} _{-1} \mathcal{H}^{6}=0$ and $\mathcal{F}^{(k_{6}/\mathbb{Q},1/1)} _{i} \mathcal{H}^{6}$ is the largest submodule of $\mathcal{H}^{6}$ such that $\mathcal{F}^{(k_{6}/\mathbb{Q},1/1)}_{i}\mathcal{H}^{6}/\mathcal{F}^{(k_{6}/\mathbb{Q},1/1)} _{i-1}\mathcal{H}^{6}$ is killed by $\mathscr{D}^{(k_{6}/\mathbb{Q},1/1)}=\left\lbrace D^{\xi}_{2r} , r>0 \right\rbrace $.
	\end{center}	
\end{itemize}
\end{defi}

\begin{coro} Galois descent from $N'=1$ to $N=6$ unramified. A basis of MMZV:
$$\mathcal{B}^{1;6} \mathrel{\mathop:}= \left\{  \zeta^{\mathfrak{m}}\left( 2x_{1}+1, \cdots, 2x_{p}+1 \atop  1, \cdots, 1, \xi \right)\zeta^{\mathfrak{m}}(2)^{s} + \sum_{y_{i} \text{ at least one even}} \alpha_{\textbf{x},\textbf{y}}\zeta^{\mathfrak{m}}\left( y_{1}, \cdots, y_{p} \atop 1, \cdots, 1, \xi \right)\zeta^{\mathfrak{m}}(2)^{s} \right.$$ 
$$\left.   +\sum_{\text{lower depth, level } \geq 1}\beta_{\textbf{x},\textbf{z}} \zeta^{\mathfrak{m}} \left( z_{1}, \cdots, z_{q} \atop 1, \cdots, 1, \xi  \right)\zeta^{\mathfrak{m}}(2)^{s} \text{  , } \alpha_{\textbf{x},\textbf{y}}, \beta_{\textbf{x},\textbf{z}}\in \mathbb{Q},  x_{i}>0 \right\}.$$
\end{coro}

\appendix
\section{Examples in small depths}

\subsection{$N=2$: Depth $2,3$}
Here we have to consider only one Galois descent, from $\mathcal{H}^{2}$ to $\mathcal{H}^{1}$. \\
In depth $1$ all the $\zeta^{\mathfrak{m}}(\overline{s})$, $s>1$ are MMZV. Let us detail the case of depth 2 and 3 as an application of the results of $\S 4.2$. In depth 2, coefficients are explicit:
\begin{lemm} 
The depth $2$ - part of the basis of the motivic multiple zeta values:
$$\left\{ \zeta^{\mathfrak{m}}(2a+1, \overline{2b+1})- \binom {2(a+b)}{2b}  \zeta^{\mathfrak{m}}(1,\overline{2(a+b)+1}), a,b> 0 \right\}.$$
\end{lemm}
\begin{proof}
Indeed, we have if $a,b>0$, $D_{1}(\zeta^{\mathfrak{m}}(2a+1, \overline{2b+1}))=0$ and for $r>0$:
$$D_{2r+1,2}(\zeta^{\mathfrak{m}}(2a+1, \overline{2b+1}))= \zeta^{\mathfrak{l}}(2r+1)\otimes \zeta^{\mathfrak{m}}(\overline{2(a+b-r)+1})$$ $$\left(-\delta_{a \leq r < a+b} \binom{2r}{2a} + \delta_{r=a} + \delta_{b\leq r < a+b} \binom{2r}{2b}(2^{-2r}-1)+ \delta_{r=a+b}(2^{-2r}-2)\binom{2(a+b)}{2b}\right). $$
There is only the case $r=a+b$ where a term ($\zeta^{\mathfrak{m}}(\overline{1}))$ which does not belong to $\mathcal{F}_{0}\mathcal{H}$ appears:
$$D_{2r+1,2}(\zeta^{\mathfrak{m}}(2a+1, \overline{2b+1}))\equiv \delta_{r=a+b}(2^{-2r}-2)\binom{2(a+b)}{2b}  \zeta^{\mathfrak{l}}(2r+1)\otimes \zeta^{\mathfrak{m}}(\overline{1}) \text{ in the quotient } \mathcal{H}^{\geq 1}.$$
Referring to the previous results, we can correct $\zeta^{\mathfrak{m}}(2a+1, \overline{2b+1})$ with terms of the same weight, same depth, and with at least one $1$ -not at the end-, which here corresponds only to $\zeta^{\mathfrak{m}}(1,\overline{2(a+b)+1})$.\\
Besides, the last equality being true in the quotient $\mathcal{H}^{\geq 1}$:
$$D_{2r+1,2}(\zeta^{\mathfrak{m}}(1, \overline{2(a+b)+1}))= \zeta^{\mathfrak{l}}(2r+1)\otimes (-\delta_{r < a+b}+ \delta_{r=a+b}(2^{-2r}-2)) \zeta^{\mathfrak{m}}(\overline{2(a+b-r)+1})$$
$$\equiv \delta_{r=a+b}(2^{-2r}-2) \zeta^{\mathfrak{l}}(2r+1)\otimes \zeta^{\mathfrak{m}}(\overline{1}) . $$
According to these calculations of infinitesimal coactions:
$$\zeta^{\mathfrak{m}}(2a+1, \overline{2b+1})- \binom {2(a+b)}{2b}  \zeta^{\mathfrak{m}}(1,\overline{2(a+b)+1}) \text{ belongs to } \mathcal{F}_{0}\mathcal{H} \text{ , i.e. is a MMZV.}$$
\end{proof}
\texttt{Examples:} Here are some motivic multiple zeta values:
$$\zeta^{\mathfrak{m}}(3, \overline{3})-6 \zeta^{\mathfrak{m}}(1,\overline{5}) \text{ , } \zeta^{\mathfrak{m}}(3, \overline{5})-15 \zeta^{\mathfrak{m}}(1,\overline{7}) \text{ , }\zeta^{\mathfrak{m}}(5, \overline{3})-15 \zeta^{\mathfrak{m}}(1,\overline{7}) \text{ ,  } \zeta^{\mathfrak{m}}(5, \overline{7})-210 \zeta^{\mathfrak{m}}(1,\overline{11}) .$$
\\
\textsc{Remarks:}
\begin{itemize}
	\item[$\cdot$] The corresponding Euler sums $\left\{ \zeta(2a+1, \overline{2b+1})- \binom {2(a+b)}{2b}  \zeta(1,\overline{2(a+b)+1}), a,b> 0 \right\}$ are a generating family of MZV in depth $2$.
	\item[$\cdot$] We can similarly prove that the following elements are (resp. motivic) MZV, if no $\overline{1}$:
	$$\begin{array}{ll}
	\zeta(\overline{A}, \overline{B}) & \zeta(A, \overline{B}) \text{ and } \zeta(\overline{A}, B),  A+B \text{ odd },  B\neq 1 \\
	\zeta(A,\overline{B}) +\zeta(\overline{A},B) \text{ , } A,B \text{ odd } & \zeta(A, \overline{B}) + (-1)^{A} \binom{A+B-2}{A-1} \zeta(1,\overline{A+B-1}), A+B \text{ even} \\
\zeta(\overline{1},\overline{1}) -\frac{1}{2}\zeta(\overline{1})^{2} \text{  ;  }  \zeta(1,\overline{1}) -\frac{1}{2}\zeta(\overline{1})^{2} & 	\zeta(\overline{A}, B)- (-1)^{A} \binom{A+B-2}{A-1} \zeta(1,\overline{A+B-1}) \text{ , } A+B \text{ even, }  A,B\neq 1 
	\end{array}$$
\end{itemize}

\begin{lemm} 
The depth $2$ part of the basis of $\mathcal{F}_{1}\mathcal{H}$:
$$\left\{ \zeta^{\mathfrak{m}}(2a+1, \overline{2b+1}) , (a,b)\neq(0, 0) \right\}.$$
\end{lemm}
\begin{proof}
No need of correction ($\mathcal{B}_{n,2,\geq 2}$ is empty for $n\neq 2$), these elements belong to $\mathcal{F}_{1}\mathcal{H}$.
\end{proof}

\begin{lemm} The depth $3$ part of the basis of motivic multiple zeta values:
\begin{multline}
 \left\{\zeta^{\mathfrak{m}}(2a+1,2b+1,\overline{2c+1})-\sum_{k=1}^{a+b+c}\alpha_{k}^{a,b,c}\zeta^{\mathfrak{m}}(1,2(a+b+c-k)+1, \overline{2k+1}) \right. \\ 
\left. -\binom {2(b+c)}{2c}\zeta^{\mathfrak{m}}(2a+1,1,\overline{2(b+c)+1}),a,b,c>0 \right\}.
 \end{multline}
 where $\alpha_{k}^{a,b,c} \in\mathbb{Z}_{\text{odd}}$ are solutions of $M_{3}X=A^{a,b,c}$. With 
$A^{a,b,c}$ such as $r^{\text{th}}-$coefficient is:
$$\delta_{b \leq r < a+b} \binom{2(n-r)}{2c}\binom{2r}{2b} - \delta_{a < r< a+b} \binom{2(n-r)}{2c}\binom{2r}{2a}-\delta_{b\leq r < b+c}\binom{2(n-r)}{2a}\binom{2r}{2b} $$
$$- \delta_{r\leq a}\binom{2(n-r)}{2(b+c)}\binom{2(b+c)}{2c} + \delta_{r<b+c}\binom{2(n-r)}{2a}\binom{2(b+c)}{2c} + \delta_{c\leq r < b+c}\binom{2r}{2c}\binom{2(n-r)}{2a}(2^{-2r}-1).$$ 
$M_{3}$ the matrix whose $(r,k)^{\text{th}}$ coefficient is:
$$\delta_{r=a+b+c}(2^{-2r}-2)\binom{2n}{2k}+ \delta_{k \leq r < n} \binom{2r}{2k}(2^{-2r}-1) - \delta_{r<n-k} \binom{2(n-r)}{2k} - \delta_{n-k \leq r<n} \binom{2r}{2(n-k)}. $$
\end{lemm}
\begin{proof}
Let $\zeta^{\mathfrak{m}}(2a+1,2b+1,\overline{2c+1})$, $a,b,c >0$ fixed, and substract elements of the same weight, of depth 3 until it belongs to $gr_{3} \mathcal{F}_{0}\mathcal{H}$.\\
Let calculate infinitesimal coproducts referring to the formula ($\ref{Deriv2}$) in the quotient $\mathcal{H}^{\geq 1}$ and use previous results for depth 2, with $n=a+b+c$:
$$D_{2r+1,3}(\zeta^{\mathfrak{m}}(2a+1,2b+1,\overline{2c+1}))\equiv \zeta^{\mathfrak{l}}(2r+1)\otimes  \left[ \delta_{r=b+c} \binom{2(b+c)}{2c}(2^{-2r}-2) \zeta^{\mathfrak{m}}(2a+1, \overline{1}) \right.$$ 
$$\left. + \zeta^{\mathfrak{m}}(1, \overline{2(n-r)+1}) \left( \delta_{a = r} \binom{2(n-r)}{2c}+  \delta_{b \leq r < a+b} \binom{2r}{2b}\binom{2(n-r)}{2c} - \delta_{a\leq r<a+b} \binom{2r}{2a}\binom{2(n-r)}{2c} \right. \right.$$
$$\left. \left. - \delta_{b\leq r<b+c} \binom{2r}{2b}\binom{2(n-r)}{2a}  +\delta_{c \leq r <b+c} \binom{2r}{2c} \binom{2(n-r)}{2a}(2^{-2r}-1)\right) \right].$$
At first, let substract $\binom {2(b+c)}{2c}\zeta(2a+1,1,\overline{2(b+c)+1})$ such that the $D^{-1}_{1,2}  D^{1}_{2r+1,3} $ are equal to zero, which comes to eliminate the term $\zeta^{\mathfrak{m}}(2a+1, \overline{1})$ appearing (case $r=b+c$).\\
So, we are left to substract a linear combination 
$$\sum_{k=1}^{a+b+c} \alpha_{k}^{a,b,c}  \zeta^{\mathfrak{m}}(1,2(a+b+c-k)+1, \overline{2k+1})$$
 such that the coefficients $\alpha_{k}^{a,b,c}$ are solutions of the system $M_{3}X=A^{a,b,c}$ where $A^{a,b,c}= (A^{a,b,c}_{r})_{r}$ satisfying in $\mathcal{H}^{\geq 1}$:
$$D_{2r+1,3} \left( \zeta^{\mathfrak{m}}(2a+1,2b+1,\overline{2c+1})- \binom {2(b+c)}{2c}\zeta(2a+1,1,\overline{2(b+c)+1})\right) \equiv$$
$$ A^{a,b,c}_{r}\zeta^{\mathfrak{l}}(2r+1)\otimes \zeta^{\mathfrak{m}}(1, \overline{2(n-r)+1}),$$
and $M_{3}= (m_{r,k})_{r,k}$ matrix such that: 
$$D_{2r+1,3}( \zeta^{\mathfrak{m}}(1,2(a+b+c-k)+1, \overline{2k+1}))= m_{r, k}   \zeta^{\mathfrak{l}}(2r+1)\otimes \zeta^{\mathfrak{m}}(1, \overline{2(n-r)+1}).$$
This system has solutions since, according to $\S 4.2$ results, the matrix $M_{3}$ is invertible.\footnote{Indeed, modulo 2, $M_{3}$ is an upper triangular matrix with $1$ on diagonal.}\\
Then, the following linear combination will be in $\mathcal{F}_{0}\mathcal{H}$:
$$\zeta^{\mathfrak{m}}(2a+1,2b+1,\overline{2c+1})-\sum_{k=1}^{a+b+c} \alpha_{k}^{a,b,c}  \zeta^{\mathfrak{m}}(1,2(a+b+c-k)+1, \overline{2k+1})-\binom {2(b+c)}{2c}\zeta(2a+1,1,\overline{2(b+c)+1}).$$ 
The coefficients $\alpha_{k}^{a,b,c}$ belong to $\mathbb{Z}_{\text{odd}}$ since coefficients are integers, and $\det(M_{3})$ is odd. Referring to the calculus of infinitesimal coactions, $A^{a,b,c}$ and $M_{3}$ are as claimed in lemma.
\end{proof}
\texttt{Examples:}
\begin{itemize}
\item[$\cdot$] By applying this lemma, with $a=b=c=1$ we obtain the following MMZV:
$$\zeta^{\mathfrak{m}}(3,3,\overline{3})+ \frac{774}{191} \zeta^{\mathfrak{m}}(1,5, \overline{3})  - \frac{804}{191} \zeta^{\mathfrak{m}}(1,3, \overline{5})  + \frac{450}{191}\zeta^{\mathfrak{m}}(1,1, \overline{7})  -6 \zeta^{\mathfrak{m}}(3,1,\overline{5}).$$
Indeed, in this case, with the previous notations:
$$M_{3}=\begin{pmatrix}
\frac{27}{4}&-1&-1 \\
-\frac{53}{8}&-\frac{111}{16}&-1\\
-\frac{1905}{64}&-\frac{1905}{64}&-\frac{127}{64}
\end{pmatrix} \text{ , } \quad A^{1,1,1}=\begin{pmatrix}
\frac{51}{2} \\
0\\
0
\end{pmatrix} .$$
\item[$\cdot$] Similarly, we obtain the following motivic multiple zeta value:
$$\zeta^{\mathfrak{m}}(3,3,\overline{5})+ \frac{850920}{203117}\zeta^{\mathfrak{m}}(1,7, \overline{3}) +\frac{838338}{203117}\zeta^{\mathfrak{m}}(1,5, \overline{5}) -\frac{3673590}{203117}\zeta^{\mathfrak{m}}(1,3, \overline{7})+ \frac{20351100}{203117} \zeta^{\mathfrak{m}}(1,1, \overline{9}) -15\zeta^{\mathfrak{m}}(3,1,\overline{7}).$$
$$\text{ There: } \quad \quad M_{3}=\begin{pmatrix}
-\frac{63}{4}& 15 & -1& -1\\
-\frac{93}{8}&-\frac{31}{16}&-6&-1 \\
-\frac{1009}{64}&-\frac{1905}{64}&-\frac{1023}{64}&-1\\
-\frac{3577}{64}&-\frac{17885}{128}&-\frac{3577}{64}&-\frac{511}{256}
\end{pmatrix} \text{ , } \quad \quad A^{1,1,2}=\begin{pmatrix}
210\\
\frac{387}{8} \\
0\\
0
\end{pmatrix} .$$
\end{itemize}

\begin{lemm} 
The depth $3$ part of the basis of $\mathcal{F}_{1}\mathcal{H}$:
$$\left\{\zeta^{\mathfrak{m}}(2a+1,2b+1,\overline{2c+1})-\delta_{a=0 \atop\text{ or } c=0} (-1)^{\delta_{c=0}} \binom{2(a+b+c)}{2b} \zeta^{\mathfrak{m}}(1,1, \overline{2(a+b+c)+1}) \right.$$
$$ \left. - \delta_{c=0} \binom{2(a+b)}{2b} \zeta(1,2(a+b)+1,\overline{1}), \text{ at most one of } a,b,c \text{ equals zero }\right\}.$$
\end{lemm}
\begin{proof} Let $\zeta^{\mathfrak{m}}(2a+1,2b+1,\overline{2c+1})$ with at most one $1$.
\begin{center}
 Our goal is to annihilate $D^{-1}_{1,3}$ and $\lbrace D^{-1}_{1,3} \circ D^{1}_{2r+1}\rbrace_{r>0}$, in the quotient $\mathcal{H}^{\geq 1}$.
 \end{center}
Let first cancel $D^{-1}_{1,3}$: if $c\neq 0$, it is already zero; else, for $c=0$, in $\mathcal{H}^{\geq 1}$, according to the results in depth $2$ for $\mathcal{F}_{0}$, we can substract $\binom{2(a+b)}{2a} \zeta(1,2(a+b)+1,\overline{1})$ since:
$$D_{1,3} (\zeta^{\mathfrak{m}}(2a+1,2b+1,\overline{1}))\equiv \binom{2(a+b)}{2a}\zeta^{\mathfrak{m}}(1,\overline{2(a+b)+1})\equiv \binom{2(a+b)}{2a} D_{1,3} (\zeta^{\mathfrak{m}}(1,2(a+b)+1,\overline{1})).$$
Besides, with $\equiv$ standing for an equality in $\mathcal{H}^{\geq 1}$:
$$ D^{-1}_{1,2} D^{1}_{2r+1,3} (\zeta^{\mathfrak{m}}(2a+1,2b+1,\overline{2c+1}))= \delta_{r=b+c} \binom{2r}{2c}(2^{-2r}-2)\zeta^{\mathfrak{m}}(\overline{2a+1})\equiv \delta_{r=b+c \atop a=0} \binom{2(b+c)}{2c}(2^{-2(b+c)}-2)\zeta^{\mathfrak{m}}(\overline{1}).$$ 
$$ D^{-1}_{1,2} D^{1}_{2r+1,3} (\zeta^{\mathfrak{m}}(1,1,\overline{2(a+b+c)+1}))= \delta_{r=a+b+c} (2^{-2(a+b+c)}-2)\zeta^{\mathfrak{m}}(\overline{1}).$$ 
$$ D^{-1}_{1,2} D^{1}_{2r+1,3} (\zeta^{\mathfrak{m}}(1,2(a+b+c)+1, \overline{1}))= \delta_{r=a+b+c} (2^{-2(a+b+c)}-2)\zeta^{\mathfrak{m}}(\overline{1}).$$ 
Therefore, to cancel $D^{-1}_{1,2}\circ D^{1}_{2r+1,3}$:
\begin{itemize}
\item[$\cdot$] If $a=0$ we substract $\binom{2(b+c)}{2c}\zeta^{\mathfrak{m}}(1,1,\overline{2(b+c)+1}) $.
\item[$\cdot$] If $c=0$, we add $\binom{2(b+c)}{2c}\zeta^{\mathfrak{m}}(1,1,\overline{2(a+b)+1})$.
\end{itemize}
\end{proof}

\paragraph{Depth $4$.} The simplest example in depth $4$ of MMZV obtained by this way, with $\alpha_{i}\in\mathbb{Q}$:
$$-\zeta^{\mathfrak{m}}(3, 3, 3, \overline{3})-\frac{3678667587000}{4605143289541}\zeta^{\mathfrak{m}}(1, 1, 1, \overline{9})+\frac{9187768536750}{4605143289541}\zeta^{\mathfrak{m}}(1, 1, 3, \overline{7})+\frac{41712466500}{4605143289541}\zeta^{\mathfrak{m}}(1, 1, 5, \overline{5})$$
$$-\frac{9160668717750}{4605143289541} \zeta^{\mathfrak{m}}(1, 1, 7, \overline{3})+\frac{11861255103300}{4605143289541}\zeta^{\mathfrak{m}}(1, 3, 1, \overline{7})+\frac{202283196216}{4605143289541}\zeta^{\mathfrak{m}}(1, 3, 3, \overline{5})$$
$$-\frac{993033536436}{4605143289541}\zeta^{\mathfrak{m}}(1, 3, 5, \overline{3})+\frac{8928106562124}{4605143289541}\zeta^{\mathfrak{m}}(1, 5, 1, \overline{5})-\frac{1488017760354}{4605143289541}\zeta^{\mathfrak{m}}(1, 5, 3, \overline{3})$$
$$-\frac{450}{191}\zeta^{\mathfrak{m}}(3, 1, 1, \overline{7})+\frac{804}{191}\zeta^{\mathfrak{m}}(3, 1, 3, \overline{5})-\frac{774}{191}\zeta^{\mathfrak{m}}(3, 1, 5, \overline{3})+6\zeta^{\mathfrak{m}}(3, 3, 1, \overline{5})$$ 
$$+ \alpha_{1} \zeta^{\mathfrak{m}}(1,-11)+ \alpha_{2} \zeta^{\mathfrak{m}}(1,-9)\zeta^{\mathfrak{m}}(2)+ \alpha_{3} \zeta^{\mathfrak{m}}(1,-7)\zeta^{\mathfrak{m}}(2)^{2}+ \alpha_{4} \zeta^{\mathfrak{m}}(1,-5)\zeta^{\mathfrak{m}}(2)^{3}+ \alpha_{5} \zeta^{\mathfrak{m}}(1,-3)\zeta^{\mathfrak{m}}(2)^{4}.$$

\subsection{ $N=3,4$: Depth $2$}

Let us detail the case of depth 2 as an application of the results in $\S 4.2$ and start by defining some coefficients appearing in the next examples:
\begin{defi}
Set $\alpha^{a,b}_{k}\in\mathbb{Z}$ such that $M(\alpha^{a,b}_{k})_{b+1 \leq k \leq \frac{n}{2}-1 }= A^{a,b}$ with $n=2(a+b+1)$:
$$M\mathrel{\mathop:}= \left( \binom{2r-1}{2k-1} \right)_{b+1 \leq r,k \leq \frac{n}{2}-1} \quad A^{a,b}\mathrel{\mathop:}=\left(-\binom{2r-1}{2b}\right)_{b+1 \leq r \leq \frac{n}{2}-1} \quad \beta^{a,b}\mathrel{\mathop:}= \binom{n-2}{2b} + \sum_{k=b+1}^{a+b} \alpha_{k} \binom{n-2}{2k-1}.$$
\end{defi}
\texttt{Nota Bene}: The matrix $M$ having integers as entries and determinant equal to $1$, and $A$ having integer components, the coefficients $\alpha^{a,b}_{k}$ are obviously integers; the matrix $M$ and its inverse are lower triangular with $1$ on the diagonal. Remark also that:\\
$$\begin{array}{llll}
\alpha^{a,b}_{b+i}= (-1)^{i} \binom{2b+2i-1}{2i-1} c_{i} & \text{ where } c_{i}\in\mathbb{N} & \text{ does not depend } & \text{ neither on $b$ nor on $a$}  \\
\alpha^{a,b}_{b+1}=-(2b+1) & \alpha^{a,b}_{b+2}=2\binom{2b+3}{3} & \alpha^{a,b}_{b+3}=-16\binom{2b+5}{5} & \alpha^{a,b}_{b+4}=272\binom{2b+7}{7}
\end{array}$$

\begin{lemm} 
The depth $2$ part of the basis of MMZV, for even weight $n=2(a+b+1)$:
$$\left\{ \zeta^{\mathfrak{m}}\left( 2a+1, 2b+1 \atop 1, \xi \right)- \beta^{a,b} \zeta^{\mathfrak{m}}\left(1,n-1 \atop 1, \xi \right) - \sum_{k=b+1}^{\frac{n}{2}-1} \alpha^{a,b}_{k} \zeta^{\mathfrak{m}}\left( n-2k, 2k \atop 1, \xi \right), a,b> 0 \right\}.$$
\end{lemm}
\begin{proof}\footnote{We omit the exponent $\xi$ indicating the projection on the second factor of the derivations $D_{r}$, to lighten the notations.} Let $Z=\zeta^{\mathfrak{m}}(2a+1, \overline{2b+1})$ fixed, with $a,b>0$.\\
First we substract a linear combination of $\zeta^{\mathfrak{m}}\left(n-2k, 2k \atop 1, \xi \right)$ in order to cancel $\lbrace D_{2r}\rbrace$. It is possible since in depth 2, because $\zeta^{\mathfrak{l}}\left( 2r \atop  1\right) =0$:
$$ D_{2r} (\zeta^{\mathfrak{m}}(x_{1}, \overline{x_{2}}))= \delta_{x_{2} \leq 2r \leq x_{1}+x_{2}-1} (-1)^{x_{2}} \binom{2r-1}{x_{2}-1} \zeta^{\mathfrak{l}}\left( 2 r \atop  \xi\right)\otimes \zeta^{\mathfrak{m}}\left( x_{1}+x_{2}-r\atop  \xi \right).$$
Hence it is sufficient to choose $\alpha_{k}$ such that $M\alpha^{a,b}= A^{a,b}$ as in Definition $A.5$.\\
Now, it remains to satisfy $D_{1}\circ D_{2r+1}(\cdot)=0$ (for $r=n-1$ only) in order to have an element of $\mathcal{F}^{k_{N}/\mathbb{Q},P/1}_{0}\mathcal{H}_{n}$. In that purpose, let substract $\beta^{a,b} \zeta^{\mathfrak{m}}(1,n-1 ;  1, \xi)$ with $\beta^{a,b}$ as in Definition $A.5$) according to the calculation of $D_{1}\circ D_{2r+1}(\cdot)$, left to the reader.\\
\end{proof}

\texttt{Examples}: 
\begin{itemize}
\item[$\cdot$] $\zeta^{\mathfrak{m}}\left(5,3 \atop 1, \xi \right)  -75 \zeta^{\mathfrak{m}}\left( 1,7 \atop 1, \xi \right) + 3 \zeta^{\mathfrak{m}}\left(4, 4 \atop 1, \xi \right) - 20 \zeta^{\mathfrak{m}}\left( 2, 6 \atop 1, \xi \right).$
\item[$\cdot$] $\zeta^{\mathfrak{m}}\left(3,5 \atop 1, \xi \right) +15 \zeta^{\mathfrak{m}}\left( 1, 7 \atop 1, \xi \right) + 5 \zeta^{\mathfrak{m}}\left(6, 2 \atop 1, \xi \right) .$
\item[$\cdot$] $\zeta^{\mathfrak{m}}\left(5,5 \atop 1, \xi \right) -350 \zeta^{\mathfrak{m}}\left( 1, 9 \atop 1, \xi \right) + 5 \zeta^{\mathfrak{m}}\left( 4, 6 \atop 1, \xi \right) -70 \zeta^{\mathfrak{m}}\left( 2, 8 \atop 1, \xi \right).$
\item[$\cdot$] $\zeta^{\mathfrak{m}}\left( 7, 5 \atop 1, \xi \right) +12810 \zeta^{\mathfrak{m}}\left( 1, 11 \atop 1, \xi \right) + 5 \zeta^{\mathfrak{m}}\left( 6, 6 \atop 1, \xi \right) -70 \zeta^{\mathfrak{m}}\left( 4,8 \atop 1, \xi \right)+ 2016 \zeta^{\mathfrak{m}}\left( 2, 10 \atop 1, \xi \right).$

\item[$\cdot$]$\zeta^{\mathfrak{m}}\left(9, 5 \atop 1, \xi \right) -685575 \zeta^{\mathfrak{m}}\left( 1, 13 \atop 1, \xi \right) + 5 \zeta^{\mathfrak{m}}\left( 8, 6 \atop 1, \xi \right) -70 \zeta^{\mathfrak{m}}\left( 6, 8 \atop 1, \xi \right)+ 2016 \zeta^{\mathfrak{m}}\left( 4, 10 \atop 1, \xi \right)- 89760 \zeta^{\mathfrak{m}}\left( 2, 12 \atop 1, \xi \right).$
\end{itemize}

\begin{lemm} 
The depth $2$ part of the basis of $\mathcal{F}^{k_{N}/\mathbb{Q},P/1}_{1}\mathcal{H}_{n}$ is for even $n$:
$$\left\{ \zeta^{\mathfrak{m}}\left(2a+1, 2b+1 \atop  1, \xi\right)- \sum_{k=b+1}^{\frac{n}{2}-1} \alpha^{a,b}_{k} \zeta^{\mathfrak{m}}\left(n-2k, 2k \atop 1, \xi \right), a,b\geq 0, (a,b)\neq(0,0) \right\},$$
For odd $n$, the part in depth $2$ of the basis of $\mathcal{F}^{k_{N}/\mathbb{Q},P/1}_{1}\mathcal{H}_{n}$ is:
$$\left\{ \zeta^{\mathfrak{m}}\left(x_{1}, x_{2} \atop 1, \xi \right)+ (-1)^{x_{2}+1}  \binom{n-2}{x_{2}-1} \zeta^{\mathfrak{m}}\left(1, n-1 \atop 1, \xi \right), x_{1},x_{2} >1, \text{ one even, the other odd } \right\}.$$
\end{lemm}
\begin{proof}
\begin{itemize}
\item[$\cdot$] For even $n$, we need to cancel $D_{2r}$ (else $D_{2s}\circ D_{2r}(\cdot) \neq 0$), so we substract the same linear combination than in the previous lemma.
\item[$\cdot$] For odd $n$, we need to cancel $D_{1}\circ D_{2r}$. Since $D_{1}\circ D_{2r}(Z)= (-1)^{x_{2}} \binom{n-2}{x_{2}-1}$, we substract $(-1)^{x_{2}+1}  \binom{n-2}{x_{2}-1} \zeta^{\mathfrak{m}}(1, \overline{n-1})$.
\end{itemize}
\end{proof}

\begin{lemm} 
The depth $2$ part of the basis of $\mathcal{F}^{k_{N}/\mathbb{Q},P/P}_{0}\mathcal{H}_{n}$ ($=\mathcal{H}_{n}^{\mathcal{MT}_{2}}$ if $N=4$) is:
$$\left\{ \zeta^{\mathfrak{m}}\left(2a+1, 2b+1 \atop 1, \xi \right)- \sum_{k=b+1}^{\frac{n}{2}-1} \alpha^{a,b}_{k} \zeta^{\mathfrak{m}}\left(n-2k, 2k \atop 1, \xi \right), a,b \geq 0 \right\}.$$
\end{lemm}
\begin{proof}
To cancel $D_{2r}$, we substract the same linear combination than above. 
\end{proof}

\begin{lemm} 
The depth $2$ part of the basis of $\mathcal{F}^{k_{N}/\mathbb{Q},P/P}_{1}\mathcal{H}_{n}$ is for even $n$:
$$\left\{ \zeta^{\mathfrak{m}}\left(2a+1, 2b+1 \atop 1, \xi \right)- \sum_{k=b+1}^{\frac{n}{2}-1} \alpha^{a,b}_{k} \zeta^{\mathfrak{m}}\left(n-2k, 2k \atop 1, \xi \right), a,b\geq 0, \right\},$$
And for odd $n$, the part in depth $2$ of the basis of $\mathcal{F}^{k_{N}/\mathbb{Q},P/P}_{1}\mathcal{H}_{n}$ is:
$$\left\{ \zeta^{\mathfrak{m}}\left( x_{1}, x_{2} \atop 1, \xi \right), x_{1},x_{2} \geq 1 \text{, one even, the other odd } \right\}.$$
\end{lemm}
\begin{proof}
If $n$ is even, to cancel $\lbrace D_{2r}\rbrace$, we use the same linear combination than above.\\
If $n$ is odd, we already have $\zeta^{\mathfrak{m}}(x_{1}, x_{2};  1, \xi)\in \mathcal{F}^{k_{N}/\mathbb{Q},P/P}_{1}\mathcal{H}_{n}$.
\end{proof}

\subsection{$N=8$: Depth $2$}

Let us explicit the results for the depth 2; proofs being similar -albeit longer- as in the previous sections are left to the reader; same notations than the previous case. 
%From $\ref{Deriv8}$, we deduce the coactions in depth 2:
%\begin{lemm}
%For $\epsilon_{i}\in \left\{\pm 1\right\}$, $x_{i}\geq 1$:
%$$D_{r,2}(\zeta^{\mathfrak{m}}(x_{1},x_{2};  \epsilon_{1},\epsilon_{2}\xi))= \delta_{r=x_{1}} \zeta^{\mathfrak{l}}(x_{1}; \epsilon_{1}) \otimes \zeta^{\mathfrak{m}}(x_{2}; \epsilon_{2}\xi)) $$
%$$+ \delta_{x_{1}\leq r <x_{1}+x_{2}-1}(-1)^{x_{2}}\binom{r-1}{r-x_{1}} \zeta^{\mathfrak{l}}(r ; \epsilon_{1}) \otimes \zeta^{\mathfrak{m}}(x_{1}+x_{2}-r; \epsilon_{1}\epsilon_{2}\xi)$$
%$$+ \delta_{x_{2}\leq r \leq x_{1}+x_{2}-1}(-1)^{r+x_{2}}\binom{r-1}{r-x_{2}} \zeta^{\mathfrak{l}}(r ; \epsilon_{2}\xi) \otimes \zeta^{\mathfrak{m}}(x_{1}+x_{2}-r; \epsilon_{1}\epsilon_{2}\xi).$$
%So:
%$$\left(D_{r,2}^{\xi}- D_{r,2}^{-\xi}  \right) \left(\zeta^{\mathfrak{m}}(x_{1},x_{2};  \epsilon_{1},\epsilon_{2}\xi) \right)= \epsilon_{2} \delta_{x_{2}\leq r \leq x_{1}+x_{2}-1}(-1)^{r+x_{2}}\binom{r-1}{r-x_{2}} \zeta^{\mathfrak{m}}(x_{1}+x_{2}-r; \epsilon_{1}\epsilon_{2}\xi) .$$
%And 
%$$D_{2r,2}(\zeta^{\mathfrak{m}}(x_{1},x_{2};  \epsilon_{1},\epsilon_{2}\xi))= \delta_{x_{2}\leq 2r \leq x_{1}+x_{2}-1}(-1)^{x_{2}}\binom{2r-1}{2r-x_{2}} \zeta^{\mathfrak{l}}(2r ; \epsilon_{2}\xi) \otimes \zeta^{\mathfrak{m}}(x_{1}+x_{2}-2r; \epsilon_{1}\epsilon_{2}\xi). $$ % Pb of signe??????????? (-1)^{x_{2}+1} ?????
%\end{lemm}

\begin{lemm} 
\begin{itemize}
	\item[$\cdot$] The depth $2$ part of the basis of MMZV$_{\mu_{4}}$:
$$\left\{ \zeta^{\mathfrak{m}}\left(x_{1}, x_{2} \atop  1, \xi\right)+ \zeta^{\mathfrak{m}}\left(x_{1},x_{2} \atop -1, -\xi\right)+ \zeta^{\mathfrak{m}}\left(x_{1},x_{2} \atop 1, -\xi\right)+ \zeta^{\mathfrak{m}}\left(x_{1},x_{2} \atop -1, \xi\right), x_{i} \geq 1 \right\}.$$
	\item[$\cdot$] The depth $2$ part of the basis of motivic Euler sums:
$$\left\{ \zeta^{\mathfrak{m}}\left(2a+1, 2b+1 \atop  1, \xi\right)+ \zeta^{\mathfrak{m}}\left(2a+1, 2b+1 \atop  -1, -\xi\right) + \zeta^{\mathfrak{m}}\left(2a+1, 2b+1 \atop  1, -\xi\right)+ \zeta^{\mathfrak{m}}\left(2a+1, 2b+1 \atop  -1, \xi\right) \right. $$
$$\left. - \sum_{k=b+1}^{\frac{n}{2}-1} \alpha^{a,b}_{k} \left( \zeta^{\mathfrak{m}}\left(n-2k, 2k \atop 1, \xi\right) + \zeta^{\mathfrak{m}}\left(n-2k, 2k \atop -1, -\xi\right) + \zeta^{\mathfrak{m}}\left(n-2k, 2k \atop 1, -\xi\right) + \zeta^{\mathfrak{m}}\left(n-2k, 2k \atop -1, \xi\right)\right)\right\}_{a,b \geq 0}$$
	\item[$\cdot$] The depth $2$ part of the basis of MMZV:
$$\left\{ \zeta^{\mathfrak{m}}\left(2a+1, 2b+1\atop  1, \xi\right) + \zeta^{\mathfrak{m}}\left(2a+1, 2b+1\atop  -1, -\xi\right)  + \zeta^{\mathfrak{m}}\left(2a+1, 2b+1\atop  1, -\xi\right)  + \zeta^{\mathfrak{m}}\left(2a+1, 2b+1\atop  -1, \xi\right) \right.$$
$$ \left. - \sum_{k=b+1}^{\frac{n}{2}-1} \alpha^{a,b}_{k} \left( \zeta^{\mathfrak{m}}\left(n-2k, 2k\atop 1, \xi\right)+ \zeta^{\mathfrak{m}}\left(n-2k, 2k\atop -1,-\xi\right)  + \zeta^{\mathfrak{m}}\left(n-2k, 2k\atop 1, -\xi\right)  + \zeta^{\mathfrak{m}}\left(n-2k, 2k\atop -1, \xi\right)   \right) \right.$$
$$\left. - \beta^{a,b} \left( \zeta^{\mathfrak{m}}\left(1,n-1\atop  1, \xi\right)+ \zeta^{\mathfrak{m}}\left(1,n-1\atop  -1, \xi\right)+ \zeta^{\mathfrak{m}}\left(1,n-1\atop  1, -\xi\right)+ \zeta^{\mathfrak{m}}\left(1,n-1\atop  -1, -\xi\right) \right),  a,b> 0 \right\} $$ 
\end{itemize}
\end{lemm}

\begin{lemm} 
\begin{itemize}
	\item[$\cdot$] The depth $2$ part of the basis of $\mathcal{F}^{k_{8}/k_{4},2/2}_{1}\mathcal{H}_{n}$ is, for even $n$:
$$\left\{ \zeta^{\mathfrak{m}}\left( x_{1}, x_{2}\atop   1, \xi\right)+ \zeta^{\mathfrak{m}}\left( x_{1},x_{2} \atop  -1, -\xi\right), \zeta^{\mathfrak{m}}\left( x_{1},x_{2} \atop  1, -\xi\right)- \zeta^{\mathfrak{m}}\left( x_{1},x_{2} \atop  -1, -\xi\right), \zeta^{\mathfrak{m}}\left( x_{1},x_{2} \atop  -1, \xi\right)+ \zeta^{\mathfrak{m}}\left( x_{1},x_{2} \atop  -1, -\xi\right),  x_{i} \geq 1 \right\}.$$ 
	\item[$\cdot$] The depth $2$ part of the basis of $\mathcal{F}^{k_{8}/\mathbb{Q},2/2}_{1}\mathcal{H}_{n}$ is for odd $n$:
$$\left\{ \zeta^{\mathfrak{m}}\left( x_{1}, x_{2}\atop   1, \xi\right)+ \zeta^{\mathfrak{m}}\left( x_{1}, x_{2}\atop   -1, -\xi\right) + \zeta^{\mathfrak{m}}\left( x_{1}, x_{2}\atop   1, -\xi\right)+ \zeta^{\mathfrak{m}}\left( x_{1}, x_{2}\atop   -1, \xi\right) , \text{ exactly one even } x_{i} \right\}.$$ 
The depth $2$ part of the basis of $\mathcal{F}^{k_{8}/\mathbb{Q},2/2}_{1}\mathcal{H}_{n}$ is for even $n$:
$$\left\{ \zeta^{\mathfrak{m}}\left( 2a+1, 2b+1\atop   -1, \xi\right)+ \zeta^{\mathfrak{m}}\left( 2a+1, 2b+1\atop   -1, -\xi\right) - \sum_{k=b+1}^{\frac{n}{2}-1} \alpha^{a,b}_{k} \left( \zeta^{\mathfrak{m}}\left( n-2k, 2k\atop  -1, \xi\right) + \zeta^{\mathfrak{m}}\left( n-2k, 2k\atop  -1, -\xi\right) \right)\right\}_{a,b\geq 0}$$
$$\cup \left\{ \zeta^{\mathfrak{m}}\left( 2a+1, 2b+1\atop   1, -\xi\right)- \zeta^{\mathfrak{m}}\left( 2a+1, 2b+1\atop   -1, -\xi\right)- \sum_{k=b+1}^{\frac{n}{2}-1} \alpha^{a,b}_{k} \left( \zeta^{\mathfrak{m}}\left( n-2k, 2k\atop  1, -\xi\right) - \zeta^{\mathfrak{m}}\left( n-2k, 2k\atop  -1, -\xi\right) \right)\right\}_{ a,b\geq 0}.$$
	\item[$\cdot$]  The depth $2$ part of the basis of $\mathcal{F}^{k_{8}/\mathbb{Q},2/1}_{1}\mathcal{H}_{n}$ is for odd $n$:
$$\left\{ \zeta^{\mathfrak{m}}\left( x_{1}, x_{2}\atop   1, \xi\right)+ \zeta^{\mathfrak{m}}\left( x_{1}, x_{2}\atop   -1, -\xi\right)+ \zeta^{\mathfrak{m}}\left( x_{1}, x_{2}\atop   1, -\xi\right) + \zeta^{\mathfrak{m}}\left( x_{1}, x_{2}\atop   -1, \xi\right) \right. $$
$$ \left. - \gamma^{x_{1},x_{2}} \left( \zeta^{\mathfrak{m}}\left( 1, n-1\atop   1, \xi\right) + \zeta^{\mathfrak{m}}\left( 1, n-1\atop   -1, -\xi\right)+ \zeta^{\mathfrak{m}}\left( 1, n-1\atop   -1, \xi\right)+ \zeta^{\mathfrak{m}}\left( 1, n-1\atop   1, -\xi\right) \right), \text{ exactly one even } x_{i} \right\}.$$
In even weight $n$, depth $2$ part of the basis of $\mathcal{F}^{k_{8}/\mathbb{Q},2/1}_{1}\mathcal{H}_{n}$ is:
$$\left\{ \zeta^{\mathfrak{m}}\left( 1, n-1\atop   1, \xi\right)+ \zeta^{\mathfrak{m}}\left( 1, n-1\atop   -1,-\xi \right)+ \zeta^{\mathfrak{m}}\left( 1, n-1\atop   1,-\xi \right) + \zeta^{\mathfrak{m}}\left( 1, n-1\atop   -1,\xi \right) \right\}  $$
$$\cup \left\{ \zeta^{\mathfrak{m}}\left(  n-1,1\atop   1, \xi\right)+ \zeta^{\mathfrak{m}}\left(  n-1,1\atop   -1,-\xi \right)+ \zeta^{\mathfrak{m}}\left( n-1,1\atop   1,-\xi \right) + \zeta^{\mathfrak{m}}\left( n-1,1\atop   -1,\xi \right) + \right.$$
$$\left. -\sum_{k=1}^{\frac{n}{2}-1} \alpha^{0,\frac{n}{2}-1}_{k} \left( \zeta^{\mathfrak{m}}\left(  n-2k, 2k\atop   1, \xi\right)  + \zeta^{\mathfrak{m}}\left(  n-2k, 2k\atop   -1, -\xi\right) + \zeta^{\mathfrak{m}}\left(  n-2k, 2k\atop   1, -\xi\right)+ \zeta^{\mathfrak{m}}\left(  n-2k, 2k\atop   -1, \xi\right)  \right)\right\}$$
$$\cup \left\{ \zeta^{\mathfrak{m}}\left( 2a+1, 2b+1\atop   \epsilon_{1}, \epsilon_{2}\xi\right)+ \epsilon_{2} \zeta^{\mathfrak{m}}\left( 2a+1,2b+1\atop   -1, -\xi\right) \right. - \beta^{a,b} \left( \zeta^{\mathfrak{m}}\left( 1, n-1\atop   \epsilon_{1}, \epsilon_{2} \xi\right) + \epsilon_{2} \zeta^{\mathfrak{m}}\left( 1,n-1\atop   -1, -\xi\right)  \right)$$
$$\left. -\sum_{k=b+1}^{\frac{n}{2}-1} \alpha^{a,b}_{k} \left( \zeta^{\mathfrak{m}}\left( n-2k, 2k\atop   \epsilon_{1}, \epsilon_{2}\xi \right) + \epsilon_{2} \zeta^{\mathfrak{m}}\left( n-2k, 2k\atop   -1, -\xi \right)  \right), a,b >0 , \epsilon_{i}\in\left\{\pm 1\right\}, \epsilon_{1}=- \epsilon_{2}  \right\} .$$
Where $\gamma^{x_{1},x_{2}}=(-1)^{x_{2}} \binom{2r-1}{2r-x_{2}}$.
\end{itemize}
\end{lemm}

\subsection{$N=6$: Depth $2$}

In depth 2, coefficients are explicit as previously:
      
\begin{lemm} 
The depth $2$ part of the basis of MMZV, for even weight $n$:
$$\left\{ \zeta^{\mathfrak{m}}\left(2a+1, 2b+1 \atop 1, \xi \right)-  \sum_{k=b+1}^{\frac{n}{2}-1} \alpha^{a,b}_{k} \zeta^{\mathfrak{m}}\left(n-2k, 2k\atop 1, \xi\right), a,b> 0 \right\},$$
\end{lemm}
Proof being similar than the cases $N=3,4$ is left to the reader.


\begin{thebibliography}{99}

\bibitem{An} Y. Andr\'{e}.
\emph{Une introduction aux motifs (motifs purs, motifs mixtes, periodes)}.
Publie par la Societe mathematique de France, AMS dans Paris, Providence, RI, 2004.

\bibitem{Bo} A. Borel.
\emph{Cohomologie r\'{e}elle stable de groupes S-arithm\'{e}tiques classiques}.
C. R. Acad. Sci. Paris Ser. A-B \textbf{274} (1972).

\bibitem{Br1} F. Brown.
\emph{On the decomposition of motivic multiple zeta values}.
Galois-Teichmuller theory and Arithmetic Geometry, Adv. Stud. Pure Math., 63, (2012).,  arXiv:1102.1310[NT].

\bibitem{Br2} F. Brown.
\emph{Mixed Tate motives over $\mathbb{Z}$}
Annals of Math., volume 175, no. 1, 949-976, (2012), arXiv:1102.1312[AG]  .

\bibitem{Br3} F. Brown.
\emph{Depth-graded motivic multiple zeta values}
(2012) http://arxiv.org/abs/1301.3053.

\bibitem{BBV} J. Blümlein, D.J. Broadhurst, J.A.M. Vermaseren
\emph{The Multiple Zeta Value Data Mine}.
Comput. Phys. Commun. 181 (2010), $582-625$;  arXiv:0907.2557v2 [math-ph] .


\bibitem{Ca} P. Cartier.
\emph{Fonctions polylogarithmes, nombres polyzêtas and groupes pro-unipotents}.
Seminaire Bourbaki $2000-2001$, exp 885, publie dans Asterisque.

\bibitem{Ch} K. T. Chen.
\emph{Iterated path integrals}, 
Bull. Amer. Math. Soc. 83, (1977), 831-879.

\bibitem{De} P. Deligne.
\emph{Le groupe fondamental unipotent motivique de $G_{m}\backslash \mu_{N}$ pour $N=2,3,4,6$ or $8$}, 
in Publications Math\'{e}matiques de L'IHES, ISSN 0073-8301, Vol. 112, Nº. 1, 2010 , pp. 101 to 141.

\bibitem{Del} P. Deligne.
\emph{Lettre a D. Zagier et F. Brown}, 
.Moscou, janvier 2012, et Princeton, 28 avril 2012.

\bibitem{DG} P. Deligne, A.B. Goncharov.
\emph{Groupes fondamentaux motiviques de Tate mixte},
in Ann. Scient. Ec. Norm. Sup., 4e serie, t. 38, 2005, pp. 1 to 56.

\bibitem{Go} A.B. Goncharov.
\emph{Galois symmetries of fundamental groupoids and noncommutative geometry}, 
in Duke Math. J. 128, no. 2 (2005), pp. 209 to 284.

\bibitem{Go2} A.B. Goncharov.
\emph{Multiple polylogarithms and mixed Tate motives}, 
arXiv: math.AG/0103059.

\bibitem{Go3} A.B. Goncharov.
\emph{The dihedral Lie algebra and Galois symmetries of $\pi_{1}^{(l)}(\mathbb{P}^{1} - (\left\{0, \infty\right\} \cup \mu_{N}))$}, 
in Duke Math. J. 110, (2001), pp. 397-487.

\bibitem{Le} M. Levine.
\emph{Tate motives and the vanishing conjectures for algebraic K-theory}, 
in Algebraic K-theory and algebraic topology, Lake Louise,1991, in NATO Adv. Sci. Inst. Ser. C Math. Phys.Sci., 407, Kluwer, 167-188, (1993).

\bibitem{Le2} M. Levine.
\emph{Mixed Motives}, 
in Mathematical Survey and Monographs, \textbf{57} American Mathematical Society, 1998.

\bibitem{Mi} J. Milne
\emph{Algebraic Groups, Lie Groups, and their Arithmetic Subgroups}
www.jmilne.org/math/ , 2011 

\bibitem{Ra} G. Racinet.
\emph{Doubles m\'{e}langes des polylogarithmes multiples aux racines de l'unit\'{e}}, 
Publ. Math. Inst. Hautes Etudes Sci. 95 (2002), pp. 185-231.

\bibitem{So} I. Souderes.
\emph{Motivic double shuffle},
Int. J. Number Theory 6 (2010), 339-370.

\bibitem{Wo} Z. Wojtkowiak
\emph{Lie algebras of Galois representations on fundamental groups}
in Galois-Teichmueller theory and Arithmetic Geometry, 
Proceedings for conferences in Kyoto (October 2010), Advanced Studies in Pure Mathematics 63, $pp. 601-627$.

\bibitem{Za} D. B. Zagier. 
\emph{Evaluation of the multiple zeta values $\zeta(2,\ldots,2,3,2,\ldots,2)$ }, 
Annals of Math. 175 ($2012$), 977-1000.




%Bibtex entrie:
%@article{greenwade93,
%    author  = "George D. Greenwade",
 %   title   = "The {C}omprehensive {T}ex {A}rchive {N}etwork ({CTAN})",
 %  year    = "1993",
 %   journal = "TUGBoat",
 %   volume  = "14",
%    number  = "3",
%    pages   = "342--351"
%}

\end{thebibliography}
\end{document}